\documentclass[10pt,twoside]{article}
\usepackage[T1]{fontenc}
\usepackage{times}
\usepackage[
  paperwidth=466.252bp,
  paperheight=715.318bp,
  left=54.795bp,
  right=52.799bp,
  top=88.245bp,
  bottom=53bp,
  headheight=11pt,
  headsep=10pt
]{geometry}
\usepackage{amsmath,amssymb,amsthm,mathtools,mathrsfs}
\usepackage{microtype}
\usepackage{hyperref}
\usepackage{cite}
\usepackage{fancyhdr}
\hypersetup{hidelinks,
pdftitle={Principal eigenvalues of nonlocal operators with advection: scaling limits and spectral phase transitions},
pdfauthor={Hoang-Hung Vo},
pdfkeywords={generalized principal eigenvalue, nonlocal dispersal operator, advection, maximum principle, nonlocal-to-local limit, singular scaling, spectral asymptotics}}
\allowdisplaybreaks
\newtheorem{theorem}{Theorem}[section]
\newtheorem{corollary}[theorem]{Corollary}
\newtheorem{lemma}[theorem]{Lemma}
\newtheorem{proposition}[theorem]{Proposition}
\numberwithin{equation}{section}
\newcommand{\R}{\mathbb R}
\newcommand{\cA}{\mathcal A}
\newcommand{\cB}{\mathcal B}
\newcommand{\cM}{\mathcal M}
\newcommand{\cT}{\mathcal T}
\newcommand{\cH}{\mathscr H}

\newcommand{\supp}{\operatorname{supp}}
\newcommand{\dist}{\operatorname{dist}}
\newcommand{\Lp}{\lambda_p}
\newcommand{\papertitle}{Principal eigenvalues of nonlocal operators with advection:\\sharp scaling limits and spectral phase transitions}
\newcommand{\shorttitle}{Principal eigenvalues of nonlocal operators with advection}
\makeatletter
\renewcommand{\tableofcontents}{%
  \begin{center}\bfseries Contents\end{center}
  \vspace{-4pt}
  \@starttoc{toc}}
\renewcommand*{\l@section}{\@dottedtocline{1}{0em}{1.5em}}
\renewenvironment{thebibliography}[1]
     {\begin{center}\normalsize Bibliography\end{center}\vspace{-5pt}%
      \@mkboth{BIBLIOGRAPHY}{BIBLIOGRAPHY}%
      \list{\@biblabel{\@arabic\c@enumiv}}%
           {\settowidth\labelwidth{\@biblabel{#1}}%
            \leftmargin\labelwidth
            \advance\leftmargin\labelsep
            \usecounter{enumiv}%
            \let\p@enumiv\@empty
            \renewcommand\theenumiv{\@arabic\c@enumiv}}%
      \sloppy\clubpenalty4000\widowpenalty4000\sfcode`\.\@m}
     {\def\@noitemerr{\@latex@warning{Empty `thebibliography' environment}}\endlist}
\makeatother
\begin{document}
\raggedbottom
\pagestyle{fancy}
\fancyhf{}
\fancyhead[LE]{\small\thepage}
\fancyhead[CE]{\small HOANG-HUNG VO}
\fancyhead[CO]{\small\MakeUppercase{\shorttitle}}
\fancyhead[RO]{\small\thepage}
\renewcommand{\headrulewidth}{0pt}
\renewcommand{\footrulewidth}{0pt}
\thispagestyle{empty}

\begin{center}
\vspace*{18pt}
{\fontsize{15}{17}\selectfont \papertitle\par}
\vspace{14pt}
{\normalsize HOANG-HUNG VO\footnote{Email: \href{mailto:vhhung@hcmiu.edu.vn}{\texttt{vhhung@hcmiu.edu.vn}}}\par}
\vspace{2pt}
{\small $^{1}$Department of Mathematics, International University, Ho Chi Minh City, Vietnam\par}
{\small $^{2}$Vietnam National University Ho Chi Minh City, Ho Chi Minh City, Vietnam\par}
\end{center}
\vspace{10pt}
\begin{center}\bfseries Abstract\end{center}
\vspace{-5pt}
\begin{quote}\small
We consider generalized principal eigenvalues of one-dimensional nonlocal dispersal operators with a drift of fixed sign. A principal difficulty in this non-self-adjoint setting is the absence of a variational characterization of the principal value. In the symmetric drift-free problem, the critical nonlocal-to-local limit can be treated through a quadratic variational structure and Sobolev-seminorm approximation as Berestycki-Coville-Vo \cite{BCV}. The drift destroys this structure and, on a bounded interval, introduces at the same time a one-sided inflow--outflow boundary geometry. Replacing the missing variational argument by estimates which remain stable under singular rescaling is therefore a central technical issue. Our approach is direct : we work with a single generalized principal value and prove the maximum principle, simplicity in the positive cone and the scaling limits from the equation itself, without passing through auxiliary generalized principal eigenvalues analogous to $\lambda_1'$ and $\lambda_1''$ in Berestycki--Rossi \cite{BR}. The replacement mechanisms are directional Harnack inequalities and coefficient barriers, direct--adjoint identities, logarithmic Collatz--Wielandt transforms, Fourier coercivity of zero extensions and scale-dependent localization. At the critical diffusive scale, Fourier compactness and an exact energy--transport identity recover the missing inflow Dirichlet condition and identify the local Dirichlet limit without a Rayleigh quotient. For variable coefficients we determine the complete small-range phase diagram on bounded intervals and on the line, and we obtain sharp large-range three-term asymptotics by a rank-one reduction and a corner Laplace analysis in the advective travel-time variable. In the homogeneous whole-line problem the critical correction is of order $\sigma^2$.
\end{quote}

\begin{quote}\small

\end{quote}

\vspace{8pt}
\tableofcontents
\clearpage

\section{Definitions and main results}

We study principal eigenvalues of nonlocal operators with a first-order drift on bounded and unbounded intervals. The questions considered here concern the maximum principle on the line, existence and simplicity of positive eigenfunctions, dependence on the domain and on the coefficients, and the behavior of the principal value when the dispersal range varies.

For second-order elliptic operators, the generalized principal-eigenvalue framework and its relation with the maximum principle on general domains are developed in \cite{BNV}. On unbounded domains, \cite[Theorem~1.4]{BR} characterizes the positive Dirichlet spectrum, while \cite[Theorem~1.6]{BR} relates the maximum principle to the auxiliary quantities $\lambda_1'$ and $\lambda_1''$. Continuous dependence and large-parameter properties of the local principal value are established in \cite[Propositions~9.1--9.2 and Theorems~9.3,~9.5]{BR}. The earlier whole-space theory is \cite{BR06}. For degenerate, oblique-boundary and parabolic extensions see \cite{BCPR,Rossi20,BNR25}; for propagation questions in which principal eigenvalues enter, see \cite{RossiAiM17}.

For nonlocal dispersal operators, existence of a positive principal eigenfunction is not automatic; bounded-domain existence and nonexistence criteria are studied in \cite{GMR,Coville10,Coville13}. Periodic and time-periodic principal spectrum points are treated in \cite{RawalShen,SX,SXapprox,ShenVo}. Berestycki, Coville and Vo establish equivalence properties for generalized principal values and study dispersal-range rescalings in \cite[Sections~2--5]{BCV}. The critical symmetric drift-free scaling considered in \cite[Section~5]{BCV} converges to the Dirichlet principal eigenvalue of the associated local second-order operator. Its proof uses the symmetric variational structure together with Sobolev-seminorm approximation results; see \cite{BBM,Brezis02,PonceCV,Ponce}.

Coville and Hamel \cite{CH} considered, on an interval $\Omega$, the operator
\[
\mathcal M_\Omega\phi(x)=q(x)\phi'(x)+\int_\Omega J(x,y)\phi(y)\,dy+a(x)\phi(x),
\qquad \mathcal M_\Omega\phi+\lambda\phi=0\quad\hbox{in }\Omega.
\]
For a bounded interval and positive drift, \cite[Proposition~1.1]{CH} constructs a positive eigenpair with the outflow condition at the right endpoint. The equality of the directional value with the generalized principal value is \cite[Theorem~1.2]{CH}. Domain, potential and kernel monotonicity, together with the $L^\infty$-Lipschitz dependence on the potential, are \cite[Proposition~1.3\textup{(i)--(iii),(v)}]{CH}. The local Harnack inequality used below is \cite[Lemma~3.1]{CH}. Exhaustion by bounded intervals and attainment of the limiting generalized principal value are \cite[Theorem~1.4]{CH}. Thus, if $q>0$ and $\Omega=(\alpha,\beta)$, the directional boundary condition is imposed at $\beta$, not at $\alpha$.

The presence of the drift changes both the unbounded-domain problem and the zero-range limit. We keep throughout a single generalized principal value $\lambda_p$. The maximum principle and simplicity are proved by direct comparison arguments, Lyapunov functions, Harnack estimates and minimal-growth arguments. We do not introduce quantities corresponding to $\lambda_1'$ and $\lambda_1''$ of \cite{BR}. This permits the hypotheses in our maximum-principle and simplicity results to be stated directly in terms of the kernel and the coefficients.

At the critical scaling, the nonlocal eigenfunction has only the outflow boundary condition for every positive range, whereas the local limit has two Dirichlet boundary conditions. The variational route of the symmetric problem \cite[Section~5]{BCV} is unavailable: after the drift is introduced there is no Rayleigh quotient for the principal value, and multiplication of the eigenvalue equation by the eigenfunction produces the additional transport flux $q(-L)\phi_\sigma(-L)^2/2$. Thus convergence of the nonlocal quadratic energy alone cannot identify either the spectral limit or the missing trace. We replace the variational argument by a direct compactness scheme. Zero extension and the Fourier estimate $1-\widehat J(\eta)\asymp\min\{\eta^2,1\}$ yield strong $L^2$ compactness and an $H^1(\R)$ limit whose support is contained in $[-L,L]$; this support information is exactly what produces $H_0^1(-L,L)$. Symmetry is then used only to transfer the nonlocal operator onto smooth test functions, while the exact energy--transport identity identifies the energy limit and forces the inflow trace to vanish.

The other scalings lead to different limits. These limits are obtained on bounded intervals for variable $q(x)>0$ and variable $a(x)$. If $q_*=\min q$, then the leading terms in the transport-dominated regimes depend on $q_*$: for $1<m<2$ only the second moment of the kernel remains, at $m=1$ the full exponential moment appears, and for $m<1$ the edge of the support of the kernel enters. For $m>2$ the first Dirichlet mode determines the leading order. On the whole line the critical limit retains the variable coefficients, while in the homogeneous case an exact formula permits the computation of the $\sigma^2$ correction. The large-range limit is expressed in terms of the advective travel time.

Let $\Omega\subset\R$ be an open interval, possibly unbounded, and define
\[
\cA_\Omega\phi(x)=d\left(\int_\Omega J(x-y)\phi(y)\,dy-\phi(x)\right)+q(x)\phi'(x)+a(x)\phi(x),
\]
where $d>0$. Throughout the paper, unless otherwise specified, we assume that
\[
J\in C_c^1(\R),\qquad J\ge0,\qquad J(z)=J(-z),\qquad \int_\R J(z)\,dz=1,
\]
and $J(0)>0$. We denote by $R_J>0$ a number such that $\supp J\subset[-R_J,R_J]$. The coefficients satisfy $q,a\in C(\R)$, $0<q_0\le q(x)\le q_1$ and $a\in L^\infty(\R)$.
Additional $C^1$ or $C^2$ regularity is required explicitly when derivatives of the coefficients are used. The case $q\le-q_0<0$ is obtained by reflection and will not be repeated.

We set
\[
\Lp(\cA_\Omega)=\sup\left\{\lambda\in\R:\ \exists\phi\in C^1(\Omega),\ \phi>0,\ \cA_\Omega\phi+\lambda\phi\le0\ \hbox{in }\Omega\right\}.
\]
When $\Omega$ has a finite outflow endpoint, the boundary regularity is understood in the sense specified in Section~2. The generalized value used here is the nonlocal quantity introduced in \cite{BCV}. In the fixed-sign drift setting, the external facts used below are precisely the positive one-sided eigenpair \cite[Proposition~1.1]{CH}, its identification with the generalized value \cite[Theorem~1.2]{CH}, the order properties \cite[Proposition~1.3]{CH}, the Harnack inequality \cite[Lemma~3.1]{CH}, and exhaustion and attainment \cite[Theorem~1.4]{CH}.

A point which will be important later is that on a bounded interval $\Omega=(\alpha,\beta)$ with $q>0$, the principal eigenfunction vanishes at $\beta$ and no Dirichlet condition is imposed at $\alpha$. We prove directly that the corresponding kernel is one-dimensional, including the possible first contact at the outflow endpoint. This geometric simplicity is what permits the adjoint and perturbation arguments below.

We first consider the maximum principle and simplicity on the whole line. Put
\[
M(s):=\int_\R J(z)e^{sz}\,dz,\qquad \cH(v):=\sup_{s\in\R}\{vs-M(s)+1\}.
\]
We also set $R_+:=\sup\{z>0:J(z)>0\}$ and $H_d(v):=d\,\cH(v/d)$ for $v\ge0$. For $\mu\in\R$,
\[
\mathfrak G(\mu):=\sup_{s\ge0}\left\{\inf_{x\in\R}(q(x)s-a(x)-\mu)-d(M(s)-1)\right\}.
\]
For a bounded interval $\Omega_0=(\alpha_0,\beta_0)$ and $\kappa>0$, set
\[
w_\kappa(x):=e^{\kappa(\beta_0-x)}-1,\qquad B_{\Omega_0,\kappa}:=d\int_{\Omega_0}J(\beta_0-y)w_\kappa(y)\,dy-q(\beta_0)\kappa.
\]
When $B_{\Omega_0,\kappa}>0$, we define
\[
\mathfrak U_{\Omega_0,\kappa}:=\sup_{x\in\Omega_0}\left(-\frac{\cA_{\Omega_0}w_\kappa(x)}{w_\kappa(x)}\right),\qquad
\mathfrak U_*:=\inf_{B_{\Omega_0,\kappa}>0}\mathfrak U_{\Omega_0,\kappa},
\]
where the last infimum is taken over bounded $\Omega_0\Subset\R$ and $\kappa>0$.

\begin{theorem}\label{thm:mainwholeline}
\textup{(i)} If $\mathfrak G(\mu)>0$ and $\limsup_{|x|\to\infty}(a(x)+\mu)<0$, then $\cA_\R+\mu$ satisfies the bounded maximum principle. The condition $\mathfrak G(\mu)>0$ is, in particular, implied by $\sup_\R(a+\mu)<H_d(q_0)$; the tail condition is still required.

\textup{(ii)} For every bounded $\Omega_0$, one has $B_{\Omega_0,\kappa}>0$ for all sufficiently large $\kappa$. Moreover $\mathfrak U_*\in\R$ and $\Lp(\cA_\R)\le\mathfrak U_*$. Put $\Lambda_*:=\max\{\|a\|_\infty,|\mathfrak U_*|\}$ and let $\gamma_*$ be a Harnack growth exponent depending only on $d,J,q_0,q_1,\|a\|_\infty,\Lambda_*$. If there exist $\beta_+,\beta_->\gamma_*$ such that
\[
\mathfrak U_*+\limsup_{x\to+\infty}\{a(x)+d(M(\beta_+)-1)+q(x)\beta_+\}<0
\]
and
\[
\mathfrak U_*+\limsup_{x\to-\infty}\{a(x)+d(M(\beta_-)-1)-q(x)\beta_-\}<0,
\]
then $\Lp(\cA_\R)$ is simple in the positive cone.
\end{theorem}

The hypotheses of Theorem~\ref{thm:mainwholeline} involve only $J,d,q$ and $a$. Part \textup{(i)} follows from a global strict supersolution and a contact argument. Part \textup{(ii)} uses finite-interval upper bounds and exponential barriers at the two ends of the line.

We next return to bounded intervals. Here the directional boundary condition makes the direct and adjoint problems asymmetric.

\begin{theorem}\label{thm:mainbounded}
\textup{(i)} For every bounded interval $\Omega=(\alpha,\beta)$, $\Lp(\cA_\Omega)$ is finite and is associated with a positive eigenfunction $\phi_\Omega\in C^1(\Omega)\cap C(\overline \Omega)$ satisfying $\phi_\Omega(\beta)=0$. It is unique up to a positive factor. If $q\in C^1(\overline \Omega)$, the adjoint eigenfunction is positive in $\Omega$ and vanishes at $\alpha$. Increasing exhaustions converge to the generalized principal value on every bounded or unbounded interval. For $\Omega_L=(-L,L)$, assume $J\in C_c^2(\R)$, $q,a\in C^2(\R)$ and $q\ge q_0>0$, write $\lambda_p(L):=\Lp(\cA_{\Omega_L})$, and normalize $\int_{-L}^{L}\phi_L\phi_L^*=1$. Then $L\mapsto\lambda_p(L)$ belongs to $C^1((0,\infty))$ and
\[
\lambda_p'(L)=
q(L)\phi_L'(L)\phi_L^*(L)
-d\phi_L(-L)\int_{-L}^{L}J(x+L)\phi_L^*(x)\,dx<0.
\]

\textup{(ii)} Let $\mathcal L\subset(0,\infty)$ be open. Assume that the pulled-back coefficients satisfy the analytic hypothesis $\mathrm{(A_{an})}$ of Section~5: locally uniformly for $L\in\mathcal L$, the functions $J(z(\xi-\eta))$, $q(z\xi)$ and $a(z\xi)$ admit bounded holomorphic extensions with respect to the complex length variable $z$ in a neighborhood of $L$, uniformly for $\xi,\eta\in[-1,1]$. Then $L\mapsto\lambda_p(L)$ is real analytic on $\mathcal L$. After pull-back to $(-1,1)$ and normalization at the origin, the direct eigenfunction is real analytic as a $C^1([-1,1])$-valued map; the normalized adjoint eigenfunction is analytic as well.
\end{theorem}

The derivative formula is obtained after pulling the problem back to a fixed interval and pairing the differentiated equation with the adjoint eigenfunction. Both terms are negative. The analytic assertion requires the additional hypothesis $\mathrm{(A_{an})}$; under this assumption the pulled-back operators form a holomorphic family and the augmented linearized problem is invertible. No global analyticity in $L$ is asserted under the standing compact-support assumptions on $J$.

We next consider the critical scale. Set $d_J:=\frac12\int_\R z^2J(z)\,dz$ and, for a function $u$ on $(-L,L)$ with zero extension $\widetilde u$,
\[
E_\sigma(u):=\frac1{2\sigma^2}\iint_{\R^2}J_\sigma(x-y)(\widetilde u(x)-\widetilde u(y))^2\,dx\,dy.
\]

\begin{theorem}\label{thm:maincritical}
Assume $J\in C_c^2(\R)$ satisfies the standing hypotheses, $q,a\in C^2([-L,L])$ and $q\ge q_0>0$. Then
\[
\Lp(\cA_{\sigma,2,L})\longrightarrow\lambda_1^D(d_J\partial_{xx}+q(x)\partial_x+a(x),(-L,L)).
\]
If $\phi_\sigma$ and $\phi_0$ are the positive nonlocal and local eigenfunctions normalized in $L^2(-L,L)$, then $\phi_\sigma\to\phi_0$ strongly in $L^2(-L,L)$,
\[
E_\sigma(\phi_\sigma)\longrightarrow d_J\int_{-L}^{L}|\phi_0'|^2\,dx,\qquad \phi_\sigma(-L)\longrightarrow0.
\]
Thus the limiting eigenfunction satisfies both Dirichlet conditions although the approximating eigenfunctions are prescribed only at the outflow endpoint.
\end{theorem}

Theorem~\ref{thm:maincritical} is proved without a variational formula. The obstruction is not merely formal. In the drift-free symmetric setting, the quadratic form controls the eigenvalue and the compactness simultaneously. Here the drift destroys that characterization and leaves, after integration by parts, the boundary flux $q(-L)\phi_\sigma(-L)^2/2$, although no condition is imposed at $-L$ for $\sigma>0$.

The replacement is split into two independent estimates collected in Proposition~\ref{prop:criticaltools}. Part \textup{(i)} uses the local Dirichlet eigenfunction as a nonlocal supersolution, including the $O(\sigma)$ boundary collars, and yields the spectral upper bound. Part \textup{(ii)} proves directly from the Fourier symbol that bounded nonlocal energy of the zero extensions implies strong $L^2$ compactness and an $H^1(\R)$ limit. Since the zero extensions are supported in $[-L,L]$, the limit belongs to $H_0^1(-L,L)$ without imposing an inflow boundary condition on the approximating problems. The equation is then passed to the limit by transferring the symmetric nonlocal part to smooth test functions. Finally, the exact energy--transport identity shows that the limiting lower bound for the energy is attained and therefore forces $\phi_\sigma(-L)\to0$. The boundary condition is thus recovered from compactness and the transport flux, rather than inserted into the approximation.

We now turn to the dependence on the zeroth-order and dispersal amplitudes. This is the nonlocal fixed-drift counterpart of the last part of Section~9 in \cite{BR}. Let $\Omega\subset\R$ be an interval, let $V\in C_b(\Omega)$ and, for $\gamma\in\R$, $\alpha>0$, set
\[
\begin{aligned}
\cA^{V}_{\gamma,\Omega}u(x)
&=d\left(\int_\Omega J(x-y)u(y)\,dy-u(x)\right)+q(x)u'(x)+\gamma V(x)u(x),\\
\cA^{D}_{\alpha,\Omega}u(x)
&=\alpha d\left(\int_\Omega J(x-y)u(y)\,dy-u(x)\right)+q(x)u'(x)+V(x)u(x),
\end{aligned}
\]
and
\[
\Lambda_V^\Omega(\gamma)=\Lp(\cA^V_{\gamma,\Omega}),\qquad
\Lambda_D^\Omega(\alpha)=\Lp(\cA^D_{\alpha,\Omega}).
\]

\begin{theorem}\label{thm:coeffamplitudes}
Put $V_+=\sup_\Omega V$ and $V_-=\inf_\Omega V$. Then $\Lambda_V^\Omega$ is concave, $\Lambda_V^\Omega(0)\ge0$ and
\[
|\Lambda_V^\Omega(\gamma_1)-\Lambda_V^\Omega(\gamma_2)|\le\|V\|_\infty|\gamma_1-\gamma_2|,
\qquad
\Lambda_V^\Omega(\gamma)=-\gamma V_\pm+o(|\gamma|)\quad(\gamma\to\pm\infty).
\]
Assume in addition that $\Omega$ is bounded, $q\in C^1(\overline\Omega)$ and $V\in C(\overline\Omega)$. Then $\Lambda_V^\Omega$ and $\Lambda_D^\Omega$ are real analytic on $\R$ and $(0,\infty)$, respectively. With $\int_\Omega\phi\phi^*=1$,
\[
(\Lambda_V^\Omega)'(\gamma)=-\int_\Omega V\phi_\gamma\phi_\gamma^*,\qquad
(\Lambda_D^\Omega)'(\alpha)=-d\int_\Omega\phi_\alpha^*(J*_\Omega\phi_\alpha-\phi_\alpha).
\]
Moreover,
\[
(\Lambda_V^\Omega)'(\gamma)=-V_\pm+o(1)\quad(\gamma\to\pm\infty),
\]
and $\Lambda_V^\Omega$ is strictly concave if $V$ is nonconstant. No sign is imposed in general on $(\Lambda_D^\Omega)'$.
\end{theorem}

In the bounded setting of the second part of Theorem~\ref{thm:coeffamplitudes}, for every finite $\gamma_0\in\R$ and $\alpha_0>0$,
\[
\Lambda_V^\Omega(\gamma)=\Lambda_V^\Omega(\gamma_0)-(\gamma-\gamma_0)\int_\Omega V\phi_{\gamma_0}\phi_{\gamma_0}^*+O((\gamma-\gamma_0)^2),
\]
and
\[
\Lambda_D^\Omega(\alpha)=\Lambda_D^\Omega(\alpha_0)-d(\alpha-\alpha_0)\int_\Omega\phi_{\alpha_0}^*(J*_\Omega\phi_{\alpha_0}-\phi_{\alpha_0})+O((\alpha-\alpha_0)^2).
\]
\begin{corollary}\label{cor:constant-potential-amplitude}
If $V\equiv V_0$, then $\Lambda_V^\Omega(\gamma)=\Lambda_V^\Omega(0)-\gamma V_0$.
\end{corollary}

\begin{proof}
If $V\equiv V_0$, then $\cA^V_{\gamma,\Omega}=\cA^V_{0,\Omega}+\gamma V_0 I$. Adding a constant zeroth-order term $\gamma V_0$ shifts the generalized principal eigenvalue by $-\gamma V_0$. Hence
\[
\Lambda_V^\Omega(\gamma)=\Lambda_V^\Omega(0)-\gamma V_0.
\]
\end{proof}

Even in the bounded setting, continuity of $V$ alone does not determine a universal order-one term from $V_+$ and $V_-$. The analyticity assertions use only the affine dependence on $\gamma$ and $\alpha$ and the invertibility of the augmented linearized problem. Strict concavity follows from the equality case in the geometric-interpolation argument.

On the whole line the argument also applies without assuming that $q$ and $V$ are constant. The appropriate object is not a scalar Rayleigh quotient but a nonlinear Collatz--Wielandt Hamiltonian obtained after the logarithmic transform of a positive test function. For $\psi\in C^1(\R)$ define
\[
\mathfrak H_{\alpha}[\psi](x)
=q(x)\psi'(x)-V(x)
-\alpha d\left(\int_\R J(z)e^{\psi(x)-\psi(x-z)}\,dz-1\right).
\]

\begin{theorem}\label{thm:heteroHamiltonian}
Assume $\Omega=\R$, $q,V\in C_b(\R)$ and $q\ge q_0>0$. For every $\alpha>0$,
\[
\Lambda_D^\R(\alpha)=\sup_{\psi\in C^1(\R)}\inf_{x\in\R}\mathfrak H_\alpha[\psi](x).
\]
The supremum is attained at $\psi_\alpha=-\log\phi_\alpha$, where $\phi_\alpha$ is a positive principal eigenfunction, and $\mathfrak H_\alpha[\psi_\alpha]\equiv\Lambda_D^\R(\alpha)$. Moreover,
\[
\Lambda_D^\R(\alpha)\ge\sup_{s\ge0}\left\{\inf_{x\in\R}(q(x)s-V(x))-\alpha d(M(s)-1)\right\}
\ge\alpha d\,\cH\!\left(\frac{q_0}{\alpha d}\right)-\sup_\R V,
\]
and the maximizing affine slope is unique. The kernel Hamiltonian satisfies
\[
\cH(v)=\frac{v^2}{2D_2(J)}-\frac{D_4(J)}{24D_2(J)^4}v^4+O(v^6)\qquad(v\to0)
\]
and
\[
\frac{\cH(v)}{v\log v}\longrightarrow\frac1{R_+}\qquad(v\to+\infty).
\]
If $J(R_+-t)=Kt^\nu(1+o(1))$ as $t\downarrow0$ for some $K>0$ and $\nu>1$, with $p=\nu+1$ and $A=K\Gamma(p)$, then
\[
\cH(v)=\frac v{R_+}\left[\log v+p\log\log v-p\log R_+-\log(R_+A)-1+o(1)\right].
\]
\end{theorem}
\begin{corollary}\label{cor:homogeneous-wholeline-amplitude}
Assume $q\equiv c>0$ and $V\in C_b(\R)$. Then
\[
\alpha d\,\cH\!\left(\frac{c}{\alpha d}\right)-\sup_\R V
\le\Lambda_D^\R(\alpha)\le
\alpha d\,\cH\!\left(\frac{c}{\alpha d}\right)-\inf_\R V.
\]
Consequently,
\[
\frac{\Lambda_D^\R(\alpha)}{|\log\alpha|}\longrightarrow\frac{c}{R_+}
\qquad(\alpha\downarrow0).
\]
Under the edge hypothesis of Theorem~\ref{thm:heteroHamiltonian},
\[
\Lambda_D^\R(\alpha)
=\frac{c}{R_+}\left[\log\frac{c}{\alpha d}
+p\log\log\frac{c}{\alpha d}-p\log R_+-\log(R_+A)-1\right]+O(1).
\]
If, in addition, $V\equiv a_0$, then
\[
\Lambda_D^\R(\alpha)=\alpha d\,\cH\!\left(\frac{c}{\alpha d}\right)-a_0
=\sup_{s\in\R}\{cs-\alpha d(M(s)-1)\}-a_0.
\]
If $s_\alpha>0$ is the optimizer, then $\alpha dM'(s_\alpha)=c$ and
\[
(\Lambda_D^\R)'(\alpha)=-d(M(s_\alpha)-1)<0,\qquad
(\Lambda_D^\R)''(\alpha)=\frac{c^2}{\alpha^3dM''(s_\alpha)}>0.
\]
Moreover,
\[
\Lambda_D^\R(\alpha)=-a_0+\frac{c^2}{2\alpha dD_2(J)}
-\frac{D_4(J)c^4}{24\alpha^3d^3D_2(J)^4}+O(\alpha^{-5})
\qquad(\alpha\to+\infty),
\]
and, under the edge hypothesis,
\[
\Lambda_D^\R(\alpha)=\frac{c}{R_+}\left[\log\frac{c}{\alpha d}
+p\log\log\frac{c}{\alpha d}-p\log R_+-\log(R_+A)-1+o(1)\right]-a_0
\]
as $\alpha\downarrow0$. Finally, for
$\cA_{\sigma,m,\R}u=\sigma^{-m}(J_\sigma*u-u)+cu'+a_0u$,
\[
\Lp(\cA_{\sigma,m,\R})=\sigma^{-m}\cH(c\sigma^{m-1})-a_0.
\]
\end{corollary}

\begin{proof}
Put $K_\alpha=\alpha d\,\cH(c/(\alpha d))$. Potential monotonicity between the
constant potentials $\inf_\R V$ and $\sup_\R V$ yields
\[
K_\alpha-\sup_\R V\le\Lambda_D^\R(\alpha)\le K_\alpha-\inf_\R V.
\]
The assertions for $\alpha\downarrow0$ follow from the large-argument asymptotics of
$\cH$ in Theorem~\ref{thm:heteroHamiltonian}; the bounded potential contributes only
an $O(1)$ term.

Assume now $V\equiv a_0$. For $s\in\R$, $\phi_s(x)=e^{-sx}$ satisfies
$J*\phi_s=M(s)\phi_s$ and
\[
\frac{\cA^D_{\alpha,\R}\phi_s}{\phi_s}
=\alpha d(M(s)-1)-cs+a_0.
\]
Theorem~\ref{thm:heteroHamiltonian} yields the lower bound after optimization over $s$.
For the reverse inequality, exhaust $\R$ by $(-R,R)$ and use
Proposition~\ref{thm:largeintervalrate}; the $R^{-2}$ correction vanishes. Hence
$\Lambda_D^\R(\alpha)=K_\alpha-a_0$.

Strict convexity of $M$ implies that the optimizer $s_\alpha$ is unique and satisfies
$\alpha dM'(s_\alpha)=c$. Differentiation yields
\[
(\Lambda_D^\R)'(\alpha)=-d(M(s_\alpha)-1),\qquad
s_\alpha'=-\frac{c}{\alpha^2dM''(s_\alpha)},
\]
and the second-derivative formula follows. The asymptotic expansions are those of
Theorem~\ref{thm:heteroHamiltonian}.

For the scaled operator take $\phi(x)=e^{-sx/\sigma}$. Then
\[
-\frac{\cA_{\sigma,m,\R}\phi}{\phi}
=\sigma^{-m}\{c\sigma^{m-1}s-M(s)+1\}-a_0.
\]
The same exhaustion argument proves
$\Lp(\cA_{\sigma,m,\R})=\sigma^{-m}\cH(c\sigma^{m-1})-a_0$.
\end{proof}

Theorem~\ref{thm:heteroHamiltonian} is formulated for variable $q$ and $V$. The finite increment $\psi(x)-\psi(x-z)$ is essential; for heterogeneous coefficients it cannot be replaced by a pointwise function of $\psi'(x)$. Corollary~\ref{cor:homogeneous-wholeline-amplitude} is a specialization of the exact formula, not an assumption used in its proof.

The zeroth-order assertions are analogous to \cite[Theorem~9.3]{BR}. In contrast with the self-adjoint result \cite[Theorem~9.5]{BR}, the derivative of the dispersal-amplitude function has no fixed sign on a heterogeneous bounded interval. In the homogeneous whole-line case it is strictly negative and the function is strictly convex.

The next result is the analogue of \cite[Propositions~9.1 and~9.2]{BR} for the present operator.

\begin{theorem}\label{thm:wholestability}
For $n\in\mathbb N$, let
\[
\cA_nu=d(J_n*u-u)+q_n(x)u'+a_n(x)u,
\qquad
\cA u=d(J*u-u)+q(x)u'+a(x)u.
\]
Assume that the kernels satisfy common support, normalization and nondegeneracy hypotheses and that $q_n,q\ge q_0>0$.

\smallskip
\noindent\emph{(i) Semicontinuity under local convergence.}
Suppose that $J_n\to J$ in $C_c^1(\R)$ with a common compact support, and that $q_n\to q$, $a_n\to a$ locally uniformly in $\R$, with common local bounds. Then $\Lp(\cA)\ge \limsup_{n\to\infty}\Lp(\cA_n)$.

\smallskip
\noindent\emph{(ii) Quantitative global stability.}
Let $\cA_1,\cA_2$ satisfy common global structural bounds, including a common bound for the supports and $C^1$ norms of their kernels, the bounds for $q_i,a_i$, and the positive lower bound for $q_i$. There exists $C>0$, depending only on these structural data, such that
\[
\big|\Lp(\cA_1)-\Lp(\cA_2)\big|
\le C\Big(
\|J_1-J_2\|_{L^1(\R)}
+\|q_1-q_2\|_{L^\infty(\R)}
+\|a_1-a_2\|_{L^\infty(\R)}
\Big).
\]
Consequently, under uniform convergence of the coefficients and $C_c^1$ convergence of the kernels with common structural bounds, $\Lp(\cA_n)\to\Lp(\cA)$. If $\phi_n>0$ is a principal eigenfunction normalized by $\phi_n(0)=1$, then every subsequence contains a further subsequence converging in $C^1_{\rm loc}(\R)$ to a positive eigenfunction associated with $\Lp(\cA)$. If $\Lp(\cA)$ is simple in the positive cone, the whole sequence converges to the normalized principal eigenfunction.
\end{theorem}
Upper semicontinuity follows from local Harnack compactness. For the two-sided estimate, the equation controls $|\phi'/\phi|$ and Harnack controls $\phi(x-z)/\phi(x)$ on the support of the kernels. This bounds $(\cA_2-\cA_1)\phi/\phi$ by the stated coefficient distances.

The second half of the paper concerns the scaled operators
\[
\begin{aligned}
\cA_{\sigma,m,L}\phi(x)
&=\frac1{\sigma^m}\left(\int_{-L}^{L}J_\sigma(x-y)\phi(y)\,dy-\phi(x)\right)
+q(x)\phi'(x)+a(x)\phi(x),\\
J_\sigma(z)&=\sigma^{-1}J(z/\sigma).
\end{aligned}
\]
When $0\le m<2$, the transport dominates the centered nonlocal diffusion and a different regime occurs.  We state the bounded phase diagram for variable coefficients.

\begin{theorem}\label{thm:mainsubcritical}
Assume $J\in C_c^2(\R)$ satisfies the standing hypotheses, $q,a\in C^2([-L,L])$, and
$q_*:=\min_{[-L,L]}q>0$. Then, as $\sigma\downarrow0$,
\[
\begin{array}{ll}
m>2:&\displaystyle \sigma^{m-2}\Lp(\cA_{\sigma,m,L})\longrightarrow d_J\frac{\pi^2}{4L^2},\\[1mm]
m=2:&\displaystyle \Lp(\cA_{\sigma,2,L})\longrightarrow
\lambda_1^D(d_J\partial_{xx}+q(x)\partial_x+a(x),(-L,L)),\\[1mm]
1<m<2:&\displaystyle \sigma^{2-m}\Lp(\cA_{\sigma,m,L})\longrightarrow\frac{q_*^2}{2D_2(J)},\\[1mm]
m=1:&\displaystyle \sigma\Lp(\cA_{\sigma,1,L})\longrightarrow\cH(q_*),\\[1mm]
0\le m<1:&\displaystyle \frac{\sigma}{|\log\sigma|}\Lp(\cA_{\sigma,m,L})
\longrightarrow\frac{q_*(1-m)}{R_+}.
\end{array}
\]
For $m>2$, the zero extensions of the $L^2$-normalized eigenfunctions converge strongly in $L^2(\R)$ to the positive normalized first Dirichlet eigenfunction of $-\partial_{xx}$ on $(-L,L)$; the nonlocal energies converge to $d_J\pi^2/(4L^2)$ and $\sigma^{m-2}\phi_{\sigma,m}(-L)^2\to0$.
\end{theorem}

\begin{corollary}\label{cor:constant-drift-subcritical}
Assume $q\equiv c>0$ and $a\in C([-L,L])$. For every fixed $0\le m<2$,
\[
\Lp(\cA_{\sigma,m,L})
=\sigma^{-m}\cH(c\sigma^{m-1})-\max_{[-L,L]}a+o(1).
\]
For $1<m<2$ the Taylor series of $\cH$ determines every term which does not vanish as $\sigma\downarrow0$. If $0\le m<1$ and the edge hypothesis of Theorem~\ref{thm:heteroHamiltonian} holds, the corresponding logarithmic edge expansion follows. If $a\equiv a_0$, these formulae reduce to the constant-potential expressions.
\end{corollary}

\begin{proof}
The sharp representation is Proposition~\ref{thm:subcritical}. For $1<m<2$, Proposition~\ref{thm:higher} gives the complete expansion through every nonvanishing order. For $0\le m<1$, substituting $v=c\sigma^{m-1}$ into the edge asymptotics of $\cH$ in Theorem~\ref{thm:heteroHamiltonian} gives the stated logarithmic expansion. If $a\equiv a_0$, then $\max_{[-L,L]}a=a_0$, and the formulas reduce to the constant-potential case.
\end{proof}

Theorem~\ref{thm:mainsubcritical} contains no constancy assumption on $q$ or $a$, and none of its five limits is obtained from a Rayleigh quotient. The proof changes with the scale. An affine exponential test function yields a common lower bound through the convex Hamiltonian $\cH$. For $1<m<2$, the matching upper bound is obtained by blowing up around points where $q$ is close to $q_*$. The rescaled problem becomes a critical nonlocal problem on a fixed interval, so Lemma~\ref{lem:criticalstability} transfers the critical compactness mechanism to this new scale. At $m=1$ the same localization leaves a fixed nonlocal kernel and the upper bound follows from bounded-interval stability followed by exhaustion. For $m<1$, neither a quadratic expansion nor critical compactness is appropriate: comparison with a constant-drift problem produces a monotone eigenfunction, exponential conjugation removes the leading drift, and a centered boundary barrier controls the remaining weak-dispersal operator. For $m>2$, after the normalization $\sigma^{m-2}$ the drift and the potential vanish at leading order; Fourier compactness and the exact energy identity reduce the limit to the sharp Dirichlet Poincar\'e constant. The occurrence of $D_2(J)$, then the full exponential moment, and finally the support edge $R_+$ is therefore a consequence of three distinct non-variational mechanisms rather than of one formal Taylor expansion.
\begin{theorem}\label{thm:mainwholephase}
Assume $J\in C_c^2(\R)$ satisfies the standing hypotheses, $q,a\in C_b^2(\R)$ and $q_*:=\inf_\R q>0$. Let $\lambda_{\sigma,m}:=\Lp(\cA_{\sigma,m,\R})$. Then, as $\sigma\downarrow0$,
\[
\begin{array}{ll}
m>2:&\displaystyle \sigma^{m-2}\lambda_{\sigma,m}\to0,\\[1mm]
m=2:&\displaystyle \lambda_{\sigma,2}\to\Lp(d_J\partial_{xx}+q(x)\partial_x+a(x),\R),\\[1mm]
1<m<2:&\displaystyle \sigma^{2-m}\lambda_{\sigma,m}\to\frac{q_*^2}{2D_2(J)},\\[1mm]
m=1:&\displaystyle \sigma\lambda_{\sigma,1}\to\cH(q_*),\\[1mm]
0\le m<1:&\displaystyle \frac{\sigma}{|\log\sigma|}\lambda_{\sigma,m}\to\frac{q_*(1-m)}{R_+}.
\end{array}
\]

\end{theorem}
No periodicity, recurrence or limit of the coefficients at infinity is assumed in Theorem~\ref{thm:mainwholephase}. The lack of a variational formula is handled here by a different device: the logarithmic transform of a positive test function yields the exact nonlinear Collatz--Wielandt representation of Theorem~\ref{thm:heteroHamiltonian}. Its finite increment $\psi(x)-\psi(x-z)$ retains the full nonlocal geometry and replaces the scalar Rayleigh quotient. Affine phases yield the sharp lower bounds, while the matching upper bounds are obtained by localization near points $x_\sigma$ for which $q(x_\sigma)\to q_*$. At the critical scale, bounded-interval exhaustion supplies the upper bound and a positive generalized eigenfunction of the local whole-line operator supplies the reverse inequality through a uniform Taylor estimate. Thus the critical row retains the full heterogeneous local operator. In the transport-dominated rows the localization selects only the slowest drift $q_*$ at leading order, whereas the bounded potential is lower order. In the supercritical row the normalization eliminates both drift and potential. An unscaled limit for $m>2$ is not universal for arbitrary heterogeneous $a$ and is therefore not asserted.

Theorems~\ref{thm:mainsubcritical} and~\ref{thm:mainwholephase} summarize the bounded and whole-line small-range limits. The critical limit is Dirichlet on a bounded interval and remains a whole-line problem on $\R$.

For variable drift, Theorems~\ref{thm:mainsubcritical} and~\ref{thm:mainwholephase} concern the sharp leading orders. Higher-order additive expansions depend on the geometry of the minimizing set of $q$ and are not asserted without further hypotheses. When $q$ is constant, the exact conjugation below permits all nonvanishing corrections to be read from the Taylor expansion of $\cH$; for $m<1$ the edge of the support of $J$ then determines the sharper logarithmic expansion. We shall also use the variable-coefficient estimate
\[
\Lp(\cA_{\sigma,m,L})\ge
\sup_{s\ge0}\left\{
\inf_{x\in[-L,L]}\left(\frac{q(x)s}{\sigma}-a(x)\right)
-\frac{M(s)-1}{\sigma^m}
\right\}.
\]
If $q_0=\min_{[-L,L]}q$, this implies the simpler estimate
\[
\Lp(\cA_{\sigma,m,L})\ge \sigma^{-m}\cH(q_0\sigma^{m-1})-\max_{[-L,L]}a.
\]
In particular, $\Lp(\cA_{\sigma,m,L})\to+\infty$ as $\sigma\to0$ for every $0\le m<2$, without assuming that $q$ is constant.

At the opposite end, put $T_q:=\int_{-L}^{L}q(x)^{-1}\,dx$, $A_q:=\int_{-L}^{L}a(x)q(x)^{-1}\,dx$ and
\[
C_m:=\mathbf1_{\{m=0\}}+T_q^{-1}\left[2\log\frac{m+1}{T_q}-\log(J(0)q(-L))-A_q\right].
\]

\begin{theorem}\label{thm:mainlarge}
Assume $q,a\in C([-L,L])$, $q\ge q_0>0$, $m\ge0$ and $J(0)>0$. Then
\[
\Lp(\cA_{\sigma,m,L})=\frac{m+1}{T_q}\log\sigma+\frac{2}{T_q}\log\log\sigma+C_m+o(1)\qquad(\sigma\to\infty).
\]
For $0\le m<2$, the map $\sigma\mapsto\Lp(\cA_{\sigma,m,L})$ attains its minimum on a nonempty compact subset of $(0,\infty)$.
\end{theorem}

Theorem~\ref{thm:mainlarge} is also obtained without a variational characterization. For $\sigma\to\infty$, the gain kernel is uniformly squeezed between two rank-one kernels because $J((x-y)/\sigma)\to J(0)$ on the bounded interval. Kernel monotonicity transfers this squeeze to the principal value. The resulting rank-one eigenproblem is still non-self-adjoint, but the first-order equation can be integrated exactly and reduced to one scalar equation. After the travel-time change of variable $T(x)=\int_{-L}^xq^{-1}$, its large spectral parameter is concentrated at the corner joining the inflow point to the outflow point. A two-dimensional corner Laplace expansion then yields the logarithmic and order-one terms. This explains why the expansion records the travel time $T_q$, the inflow value $q(-L)$, the path integral $\int a/q$ and the local kernel value $J(0)$, rather than a quadratic energy.

We use only $\Lp$ throughout. The relations involving $\lambda_1'$ and $\lambda_1''$ in \cite{BR} are not needed in the arguments below. We also do not consider exterior perturbations of unbounded domains.

\textbf{The organization of the paper is as follows.} Section~2 records the external fixed-sign spectral input and the basic stability statements. Section~3 concentrates the maximum-principle tools in Proposition~\ref{prop:MPcore} and Lemma~\ref{lem:MPcontact}. Section~4 collects the minimal-growth, tail and fixed-domain perturbation arguments in Propositions~\ref{prop:simplicitycore}, \ref{prop:tailcore} and~\ref{prop:perturbcore}. Section~5 contains the interval-size theory in Propositions~\ref{prop:sizecore} and~\ref{prop:sizeanalytic}. Section~6 contains the critical compactness, crossover, supercritical and adjoint limits in Proposition~\ref{prop:criticaltools}, Lemma~\ref{lem:criticalstability}, and Propositions~\ref{prop:crossovers}, \ref{prop:supercritical} and~\ref{prop:criticaladjoint}. Sections~7--8 treat the subcritical regimes and their constant-drift refinements. Section~9 is devoted to large dispersal ranges. Section~10 proves the coefficient-amplitude and whole-line asymptotic results stated in Section~1.

\section{Preliminary properties of the principal eigenvalue}

For an interval $\Omega\subset\R$, define
\[
\Lp(\Omega)=\sup\left\{\lambda\in\R:\ \exists\phi\in C^1(\Omega),\ \phi>0,\ \cA_\Omega\phi+\lambda\phi\le0\ \hbox{in }\Omega\right\}.
\]
If $\Omega=(\alpha,\beta)$ is bounded, functions are assumed continuous up to the endpoints. When $q>0$, the positive principal eigenfunction vanishes at $\beta$.

\begin{proposition}\label{thm:CHbounded}
Let $\Omega=(\alpha,\beta)$ be a bounded interval. Assume that $J$ satisfies the structural hypotheses stated above and that $q,a\in C(\overline \Omega)$ with $q\ge q_0>0$. Then $\Lp(\Omega)$ is finite and there exists $\phi_\Omega\in C^1(\overline \Omega)$ such that
\[
\phi_\Omega>0\ \hbox{in }\Omega,\quad \phi_\Omega(\beta)=0,\quad \cA_\Omega\phi_\Omega+\Lp(\Omega)\phi_\Omega=0\ \hbox{in }\Omega.
\]
On
\[
X_\Omega=\{u\in C^1(\overline \Omega):u(\beta)=0\}
\]
one has the full geometric-simplicity statement
\[
\ker(\cA_\Omega+\Lp(\Omega))=\operatorname{span}\{\phi_\Omega\}.
\]
Moreover, if $\Omega_n\uparrow \Omega$ is an increasing exhaustion of an interval $\Omega$, bounded or unbounded, then $\Lp(\Omega_n)\downarrow\Lp(\Omega)$,
and, after normalization at one fixed interior point, the corresponding eigenfunctions are locally compact in $C^1$ and converge along subsequences to a positive eigenfunction associated with $\Lp(\Omega)$.
\end{proposition}

\begin{proof}[\textbf{Proof of Proposition~\ref{thm:CHbounded}}]
For bounded intervals, existence of the positive eigenpair, finiteness of the value and its identification with the generalized value are the fixed-sign case of the results of Coville and Hamel \cite[Theorem~1.2 and Proposition~1.1]{CH}, after absorbing the loss term $-d\phi$ into the zeroth-order coefficient. Their hypotheses on the kernel are satisfied because $J$ is continuous, $J(0)>0$, and hence there exist $\delta>0$ and constants $0<j_0\le j_1$ such that
\[
j_0\mathbf 1_{(-\delta,\delta)}(x-y)\le J(x-y)\le j_1\mathbf 1_{(-R_J,R_J)}(x-y)
\]
after decreasing $\delta$ and increasing $j_1$ if necessary. The fixed-sign assumption is $q\ge q_0>0$, and the outflow endpoint is $\beta$. The equation expresses $\phi_\Omega'$ as a continuous function on $\overline \Omega$, so $\phi_\Omega\in C^1(\overline \Omega)$. Passing to $x=\beta$ yields
\[
q(\beta)\phi_\Omega'(\beta)
=-d\int_\Omega J(\beta-y)\phi_\Omega(y)\,dy<0.
\]

We prove the kernel statement. Let $w\in X_\Omega$ solve $(\cA_\Omega+\Lp(\Omega))w=0$. Since $w(\beta)=\phi_\Omega(\beta)=0$ and $\phi_\Omega'(\beta)<0$, the quotient $w/\phi_\Omega$ extends continuously to $\beta$ by
\[
\left.\frac{w}{\phi_\Omega}\right|_{x=\beta}
=\frac{w'(\beta)}{\phi_\Omega'(\beta)}.
\]
Set
\[
t=\max_{[\alpha,\beta]}\frac{w}{\phi_\Omega},
\qquad z=w-t\phi_\Omega,
\]
where the quotient at $\beta$ has this extended value. Then $z\le0$ and the quotient has a contact point in $\overline \Omega$. At an interior contact point $x_0$, one has $z(x_0)=z'(x_0)=0$ and therefore
\[
0=(\cA_\Omega+\Lp(\Omega))z(x_0)
=d\int_\Omega J(x_0-y)z(y)\,dy.
\]
The integrand is nonpositive and $J$ is strictly positive near the origin; hence $z=0$ in a neighborhood of $x_0$. If $(a_*,b_*)$ is the maximal interval containing $x_0$ on which $z=0$, evaluation at a finite endpoint of that interval, where also $z'=0$, extends the zero set across the endpoint. Iteration in steps smaller than the positivity radius of $J$ yields $z\equiv0$ on $[\alpha,\beta]$.

If contact occurs at the inflow endpoint $\alpha$, then $z'(\alpha)\le0$ and
\[
q(\alpha)z'(\alpha)+d\int_\Omega J(\alpha-y)z(y)\,dy=0.
\]
Both terms are nonpositive, so both vanish. Kernel positivity yields $z=0$ on a right neighborhood of $\alpha$, and the preceding propagation again yields $z\equiv0$.

It remains to consider contact occurring only at the outflow endpoint. There
\[
q(\beta)z'(\beta)=-d\int_\Omega J(\beta-y)z(y)\,dy\ge0.
\]
If the integral is strictly negative, then $z'(\beta)>0$ and
\[
\lim_{x\uparrow\beta}\frac{w(x)}{\phi_\Omega(x)}
=t+\frac{z'(\beta)}{\phi_\Omega'(\beta)}<t.
\]
Thus the maximum defining $t$ is attained away from a left neighborhood of $\beta$, reducing to the interior or inflow case; the identity $z\equiv0$ then contradicts strict negativity of the integral. Hence the integral is zero and $z'(\beta)=0$. Positivity of $J(\beta-\cdot)$ near $\beta$ yields $z=0$ on a left neighborhood of $\beta$. Propagating its left endpoint as above again yields $z\equiv0$. Therefore $w=t\phi_\Omega$ and the kernel is one-dimensional.

For an unbounded interval, the exhaustion statement and the existence of a positive eigenfunction at the limiting value are \cite[Theorem~1.4]{CH}. We recall the compactness argument because it will be used below. If $\Omega_n\uparrow \Omega$, domain monotonicity yields a decreasing sequence $\Lp(\Omega_n)$. The constant test function gives
\[
\Lp(\Omega_n)\ge-\sup_{\Omega_n}a\ge-\|a\|_{L^\infty(\Omega)},
\]
because $\int_{\Omega_n}J(x-y)\,dy\le1$. Thus the decreasing eigenvalue sequence is bounded from below and hence bounded. Normalize $\phi_n$ by $\phi_n(x_0)=1$ at a fixed $x_0\in \Omega_1$. The Harnack estimate \cite[Lemma~3.1]{CH}, applied on every compact interval $K\Subset \Omega$, yields constants $C_K$ independent of $n$ such that $C_K^{-1}\le\phi_n(x)\le C_K\quad x\in K$
for all sufficiently large $n$. The eigenvalue equation can be written
\[
q(x)\phi_n'(x)=-d\int_{\Omega_n}J(x-y)\phi_n(y)\,dy+\big(d-a(x)-\Lp(\Omega_n)\big)\phi_n(x).
\]
The right-hand side is uniformly bounded on $K$, and equicontinuity follows first for $\phi_n$ and then for $\phi_n'$. Arzel\`a--Ascoli and a diagonal extraction yield $C^1_{\rm loc}(\Omega)$ convergence. Passing to the limit in the equation yields a positive eigenfunction on $\Omega$. The identification of the exhaustion limit with $\Lp(\Omega)$ and attainment at that value are \cite[Theorem~1.4]{CH}; this completes the proof.
\end{proof}

\begin{lemma}\label{thm:monotonicity}
Let $\Omega_1\subset \Omega_2$ be intervals. Then $\Lp(\Omega_1)\ge\Lp(\Omega_2)$. If $a_1\le a_2$ on a fixed interval $\Omega$, then
\[
\Lp(d\mathcal J_\Omega+q\partial_x+a_1)\ge\Lp(d\mathcal J_\Omega+q\partial_x+a_2),
\]
where $\mathcal J_\Omega u=\int_\Omega J(x-y)u(y)\,dy-u(x)$. Moreover,
\[
|\Lp(a_1)-\Lp(a_2)|\le\|a_1-a_2\|_{L^\infty(\Omega)}.
\]
If $K_1,K_2$ are nonnegative kernels on $\Omega\times \Omega$ and $K_1\le K_2$, then, with the drift and zeroth-order coefficient fixed,
\[
\Lp(dK_1+q\partial_x+a)\ge\Lp(dK_2+q\partial_x+a).
\]
\end{lemma}

\begin{proof}[\textbf{Proof of Lemma~\ref{thm:monotonicity}}]
Let $\lambda<\Lp(\Omega_2)$ and choose $\phi>0$ in $\Omega_2$ such that $\cA_{\Omega_2}\phi+\lambda\phi\le0$. For $x\in \Omega_1$,
\[
\cA_{\Omega_1}\phi(x)=\cA_{\Omega_2}\phi(x)-d\int_{\Omega_2\setminus \Omega_1}J(x-y)\phi(y)\,dy\le\cA_{\Omega_2}\phi(x).
\]
Hence the restriction of $\phi$ is admissible for $\lambda$ on $\Omega_1$, and $\Lp(\Omega_1)\ge\Lp(\Omega_2)$. The potential monotonicity follows directly from the defining inequality. If $\delta=\|a_1-a_2\|_\infty$, then $a_2-\delta\le a_1\le a_2+\delta$. Adding a constant $c$ to the potential shifts the generalized principal eigenvalue by $-c$, and therefore $\Lp(a_2)-\delta\le\Lp(a_1)\le\Lp(a_2)+\delta$.
If $K_1\le K_2$, then for every positive test function $\phi$,
\[
\int_\Omega K_1(x,y)\phi(y)\,dy\le
\int_\Omega K_2(x,y)\phi(y)\,dy.
\]
Every level admissible for $K_2$ is therefore admissible for $K_1$.
\end{proof}

\begin{proposition}\label{thm:comparison}
Let $\Omega=(\alpha,\beta)$ be bounded and assume $q\in C^1(\overline \Omega)$. Let $\phi_\Omega$ be the positive principal eigenfunction associated with $\Lp(\Omega)$. Then the formal adjoint
\[
\cA_\Omega^*v=d\left(\int_\Omega J(x-y)v(y)\,dy-v(x)\right)-(qv)'(x)+a(x)v(x)
\]
has a positive eigenfunction $\phi_\Omega^*\in C^1(\Omega)\cap C(\overline \Omega)$ satisfying
\[
\cA_\Omega^*\phi_\Omega^*+\Lp(\Omega)\phi_\Omega^*=0\ \hbox{in }\Omega,\quad \phi_\Omega^*(\alpha)=0,\quad \phi_\Omega^*>0\ \hbox{in }\Omega.
\]
If $v\in C^1(\Omega)\cap C(\overline \Omega)$ satisfies $v>0$ in $\Omega$, $v(\beta)=0$ and $\cA_\Omega v+\mu v\ge0$ in $\Omega$, then $\Lp(\Omega)\le\mu$. If equality holds and $v$ is an eigenfunction, then $v$ is proportional to $\phi_\Omega$.
\end{proposition}

\begin{proof}[\textbf{Proof of Proposition~\ref{thm:comparison}}]
The adjoint drift is $-q$. After the reflection $x\mapsto-x$, it has a positive drift bounded away from zero. Proposition~\ref{thm:CHbounded} therefore provides a positive adjoint eigenfunction, with its outflow condition at the reflected right endpoint, which is $\alpha$ in the original variable. Let the corresponding adjoint eigenvalue be $\mu_*$. For $u,v\in C^1(\overline \Omega)$ with $u(\beta)=0$ and $v(\alpha)=0$, symmetry of $J$ and Fubini's theorem yield
\[
\int_\Omega v(x)\int_\Omega J(x-y)u(y)\,dy\,dx=\int_\Omega u(x)\int_\Omega J(x-y)v(y)\,dy\,dx,
\]
while integration by parts yields
\[
\int_\Omega vq u'\,dx=[qvu]_{\alpha}^{\beta}-\int_\Omega u(qv)'\,dx=-\int_\Omega u(qv)'\,dx.
\]
Thus
\[
\int_\Omega v\cA_\Omega u\,dx=\int_\Omega u\cA_\Omega^*v\,dx.
\]
Taking $u=\phi_\Omega$ and $v=\phi_\Omega^*$ yields
\[
-\Lp(\Omega)\int_\Omega\phi_\Omega\phi_\Omega^*\,dx=-\mu_*\int_\Omega\phi_\Omega\phi_\Omega^*\,dx.
\]
The integral is strictly positive, hence $\mu_*=\Lp(\Omega)$.

Now let $v$ be the supersolution in the statement. Multiplying its inequality by $\phi_\Omega^*$ and integrating yields
\[
0\le\int_\Omega\phi_\Omega^*(\cA_\Omega v+\mu v)\,dx=\int_\Omega v(\cA_\Omega^*\phi_\Omega^*+\mu\phi_\Omega^*)\,dx=(\mu-\Lp(\Omega))\int_\Omega v\phi_\Omega^*\,dx.
\]
Since $v\phi_\Omega^*>0$ in $\Omega$, this proves $\mu\ge\Lp(\Omega)$. If $\mu=\Lp(\Omega)$ and $v$ is an eigenfunction, geometric simplicity in Proposition~\ref{thm:CHbounded} yields $v=C\phi_\Omega$ for some $C>0$.
\end{proof}

\begin{proposition}\label{thm:stability}
Let $\Omega_n=(\alpha_n,\beta_n)$ converge to $\Omega=(\alpha,\beta)$, and suppose that, after the affine identification of $\overline \Omega_n$ with $\overline \Omega$, the coefficients satisfy
\[
J_n\to J\ \hbox{in }C^1_c(\R),\quad q_n\to q\ \hbox{in }C^1(\overline \Omega),\quad a_n\to a\ \hbox{in }C(\overline \Omega),
\]
with a common support bound for $J_n$, a common nondegeneracy neighborhood at the origin, and $q_n\ge q_0>0$. Then
\[
\Lp(\cA_{\Omega_n}^{(n)})\longrightarrow\Lp(\cA_\Omega).
\]
If the principal eigenfunctions are normalized at a fixed interior point after pull-back, then they converge in $C^1(\overline \Omega)$ to the normalized principal eigenfunction of the limiting operator.
\end{proposition}

\begin{proof}[\textbf{Proof of Proposition~\ref{thm:stability}}]
After the affine pull-back we may work on one fixed interval $\Omega$. The comparison with the functions $\beta-x$ and $1$ yields a uniform two-sided bound for the principal eigenvalues. Indeed, the constant function yields the lower bound $\Lp(\cA_\Omega^{(n)})\ge-\|a_n\|_\infty$,
because the truncated gain satisfies $\int_\Omega J_n(x-y)\,dy\le1$. For the upper bound choose $w_\kappa(x)=e^{\kappa(\beta-x)}-1$.
This function is positive in $\Omega$, vanishes at $\beta$ and satisfies $w_\kappa'(\beta)=-\kappa$. The common lower bound for $J_n$ near the origin implies
\[
d\int_\Omega J_n(\beta-y)w_\kappa(y)\,dy-q_n(\beta)\kappa>0
\]
for one sufficiently large $\kappa$, uniformly for all large $n$: the integral contains a fixed positive multiple of $e^{\kappa\delta/2}$, whereas the negative drift term is only linear in $\kappa$. By uniform continuity, $\cA_\Omega^{(n)}w_\kappa$ is positive in a fixed collar of $\beta$. On the complement of that collar, $w_\kappa$ has a positive minimum and all coefficients are uniformly bounded. Hence a fixed constant $C$ can be chosen so that $\cA_\Omega^{(n)}w_\kappa+Cw_\kappa\ge0$
for all large $n$. Enlarging $C$ handles the finitely many remaining indices. Proposition~\ref{thm:comparison} yields $\Lp(\cA_\Omega^{(n)})\le C$.

Let $\lambda_n$ denote the principal values and normalize $\phi_n(x_0)=1$. We spell out the endpoint-uniform consequence of the bounded-interval Harnack estimate. Let $r_0,j_0>0$ be such that $J_n(z)\ge j_0$ for $|z|\le r_0$, uniformly in $n$. Fix
\[
0<\eta<\min\{r_0/4,(\beta-\alpha)/4\}.
\]
A finite chain of overlapping intervals joining $x_0$ to $[\alpha,\beta-\eta]$ gives
\[
C_\eta^{-1}\le \phi_n(x)\le C_\eta
\qquad (\alpha\le x\le\beta-\eta),
\]
with $C_\eta$ independent of $n$. On the last collar the eigenvalue equation and the terminal value $\phi_n(\beta)=0$ give
\[
\phi_n(x)=d\int_x^\beta\frac{1}{q_n(t)}
\exp\!\left(\int_x^t\frac{a_n(r)+\lambda_n-d}{q_n(r)}\,dr\right)
\left(\int_\Omega J_n(t-y)\phi_n(y)\,dy\right)dt .
\]
All exponential factors in this formula are bounded above and below by positive constants independent of $n$. Put
\[
M_n:=\max_{[\beta-\eta,\beta]}\phi_n.
\]
Splitting the gain integral at $\beta-\eta$ and using the preceding Harnack bound gives
\[
\int_\Omega J_n(t-y)\phi_n(y)\,dy\le C_\eta+M_n
\qquad(\beta-\eta\le t\le\beta).
\]
Consequently $M_n\le C\eta(C_\eta+M_n)$. Decrease $\eta$, independently of $n$, so that $C\eta\le1/2$; then $M_n\le 2C\eta C_\eta$. Hence the gain is uniformly bounded on the terminal collar and the formula yields the upper linear estimate below. For the lower estimate, if $\beta-\eta\le t\le\beta$ and $\beta-3\eta\le y\le\beta-2\eta$, then $|t-y|\le3\eta<r_0$; therefore the Harnack lower bound gives
\[
\int_\Omega J_n(t-y)\phi_n(y)\,dy
\ge j_0\int_{\beta-3\eta}^{\beta-2\eta}\phi_n(y)\,dy
\ge j_0\eta C_\eta^{-1}.
\]
Substitution in the terminal formula proves
\[
c(\beta-x)\le\phi_n(x)\le C(\beta-x)
\qquad (\beta-\eta\le x\le\beta),
\]
where $c,C>0$ are independent of $n$. Thus $\phi_n$ is uniformly bounded on the closed interval. The identity
\[
q_n(x)\phi_n'(x)=-d\int_\Omega J_n(x-y)\phi_n(y)\,dy+\big(d-a_n(x)-\lambda_n\big)\phi_n(x)
\]
now gives a uniform $C^1$ bound on $\overline\Omega$. Moreover, the uniform convergence of $a_n,q_n,J_n,J_n'$ makes the right-hand side equicontinuous; hence $\{\phi_n'\}$ is equicontinuous. Arzel\`a--Ascoli therefore makes $\{\phi_n\}$ precompact in $C^1(\overline \Omega)$. Along a subsequence, $\lambda_n\to\lambda$ and $\phi_n\to\phi$ in $C^1$. Passing to the limit yields
\[
\cA_\Omega\phi+\lambda\phi=0,\quad \phi>0\ \hbox{in }\Omega,\quad \phi(\beta)=0,\quad \phi(x_0)=1.
\]
Proposition~\ref{thm:CHbounded} identifies $\lambda=\Lp(\cA_\Omega)$ and the normalized eigenfunction uniquely. Every subsequence has the same limit, so the whole sequence converges.
\end{proof}

\begin{proposition}\label{thm:generalKstability}
Let $\Omega=(\alpha,\beta)$ be fixed and
\[
\cT_nu(x)=\int_\Omega K_n(x,y)u(y)\,dy+q_n(x)u'(x)+b_n(x)u(x).
\]
Assume $K_n,K\in C(\overline \Omega\times\overline \Omega)$ are nonnegative,
\[
\|K_n-K\|_{C(\overline \Omega^2)}
+\|q_n-q\|_{C(\overline \Omega)}
+\|b_n-b\|_{C(\overline \Omega)}\longrightarrow0,
\]
$q_n,q\ge q_0>0$, and there exist $\rho,\kappa>0$ such that
\[
K_n(x,y),K(x,y)\ge\kappa
\quad\hbox{whenever }x,y\in\overline \Omega,\ |x-y|<\rho .
\]
Then
\[
\Lp(\cT_n,\Omega)\longrightarrow\Lp(\cT,\Omega).
\]
If the positive eigenfunctions are normalized at a fixed $x_0\in \Omega$, then they converge to the normalized limiting eigenfunction in $C^1(\overline \Omega)$.
\end{proposition}

\begin{proof}[\textbf{Proof of Proposition~\ref{thm:generalKstability}}]
We record the argument because the scaled tilted kernels used below are not convolutions. For a real parameter $\lambda$ and $f\in C(\overline \Omega)$, terminal integration of $q_nu'+(b_n+\lambda)u=-K_nf,\qquad u(\beta)=0$,
defines the compact positive operator
\[
(T_{n,\lambda}f)(x)
=\int_x^\beta\frac1{q_n(s)}
\exp\!\left(\int_x^s\frac{b_n(r)+\lambda}{q_n(r)}\,dr\right)
\left(\int_\Omega K_n(s,y)f(y)\,dy\right)ds.
\]
Work on the weighted Banach lattice
\[
X_\beta=\left\{u\in C(\overline \Omega):u(\beta)=0,\
\|u\|_{X_\beta}:=\left\|\frac{u}{\beta-\cdot}\right\|_\infty<\infty\right\}.
\]
The terminal formula maps $X_\beta$ into itself. On the unit ball of $X_\beta$, the family
\[
x\longmapsto \frac{(T_{n,\lambda}f)(x)}{\beta-x}
\]
is uniformly bounded and equicontinuous. Indeed, the quotient is the average over $[x,\beta]$ of the uniformly continuous integrand in the terminal formula. Hence $T_{n,\lambda}$ is compact on $X_\beta$. The common diagonal lower bound implies positivity improvement after finitely many iterates: if $f\ge0$, $f\not\equiv0$, positivity first appears on one interval and then propagates, through a finite chain of overlapping intervals of length smaller than $\rho$, from that interval to every compact subinterval of $\Omega$. The terminal integral then gives strict positivity in the weighted cone up to $\beta$. Thus $T_{n,\lambda}$ and $T_\lambda$ are primitive compact positive operators. By Krein--Rutman their spectral radii are algebraically simple positive eigenvalues with positive eigenvectors. Moreover, $T_{n,\lambda}\phi=\phi$ is exactly the terminal integral form of
$\cT_n\phi+\lambda\phi=0$, $\phi(\beta)=0$. The comparison principle therefore identifies the unique level-one point with $\Lp(\cT_n,\Omega)$.

On every compact set of $\lambda$-values, the displayed integral formula and the uniform coefficient convergence yield
\[
\sup_\lambda\|T_{n,\lambda}-T_\lambda\|_{\mathcal L(X_\beta)}\longrightarrow0.
\]
Indeed, both the quotient by $\beta-x$ and its limit at $\beta$ are controlled directly by the terminal formula. Norm convergence of compact operators, together with simplicity and isolation of their positive spectral radii, yields uniform convergence $r(T_{n,\lambda})\longrightarrow r(T_\lambda)$
on compact $\lambda$-intervals. This can be seen explicitly by integrating the resolvent around a circle enclosing only $r(T_\lambda)$: the resolvents converge there by a Neumann series, hence so do the rank-one Riesz projections.

If $\lambda_2>\lambda_1$, the exponential factor in the terminal formula is strictly larger for $s>x$; positivity improvement and the strict spectral-radius comparison for primitive positive compact operators give
$r(T_{n,\lambda_2})>r(T_{n,\lambda_1})$. Uniformly for large $n$, the terminal formula also gives $\|T_{n,\lambda}\|\to0$ as $\lambda\to-\infty$. In the opposite direction, choose two fixed nested subintervals separated from $\beta$ and use the diagonal lower bound along a finite interaction chain; an iterate $T_{n,\lambda}^N$ then dominates on a nontrivial positive cone a factor $Ce^{c\lambda}$, with $C,c>0$ independent of large $n$. Hence $r(T_{n,\lambda})\to+\infty$ as $\lambda\to+\infty$, uniformly in the sense needed to choose numbers $A<B$, independent of large $n$, such that
$r(T_{n,A})<1<r(T_{n,B})$. Therefore the unique level-one point $\lambda_n$ lies in $[A,B]$. If $\lambda$ is the corresponding point for the limit, then for every $\varepsilon>0$,
\[
r(T_{\lambda-\varepsilon})<1<r(T_{\lambda+\varepsilon}),
\]
and uniform spectral-radius convergence yields $\lambda-\varepsilon<\lambda_n<\lambda+\varepsilon$ for large $n$. Hence $\lambda_n\to\lambda=\Lp(\cT,\Omega)$.

Normalize $\phi_n(x_0)=1$. The level-one equation and the uniform compactness of the terminal operators make $\{\phi_n\}$ relatively compact in $X_\beta$. Every limit is the normalized positive eigenfunction of $\cT$. The first-contact proof of Proposition~\ref{thm:CHbounded} applies to $K$, and geometric simplicity yields convergence of the full sequence. Finally,
\[
\phi_n'
=-\frac{1}{q_n}\left(\int_\Omega K_n(\cdot,y)\phi_n(y)\,dy
+(b_n+\lambda_n)\phi_n\right)
\]
converges uniformly. Thus $\phi_n\to\phi$ in $C^1(\overline \Omega)$.
\end{proof}

\begin{lemma}\label{thm:growth}
Let $u>0$ solve $\cA_\R u+\lambda u=0$ in $\R$, where $|\lambda|\le\Lambda$, $\|a\|_\infty\le A$, $q_0\le q\le q_1$, and the kernel satisfies the structural assumptions above. There exist constants $C_H\ge1$ and $\gamma_H>0$, depending only on $d,J,q_0,q_1,A,\Lambda$, such that
\[
C_H^{-1}e^{-\gamma_H|x-y|}\le\frac{u(x)}{u(y)}\le C_He^{\gamma_H|x-y|}\;\hbox{for all }x,y\in\R.
\]
\end{lemma}

\begin{proof}[\textbf{Proof of Lemma~\ref{thm:growth}}]
The Harnack inequality \cite[Lemma~3.1]{CH} yields a constant $C_0\ge1$, depending only on the structural bounds, such that for every interval $[x-2,x+2]$ one has $\sup_{[x-1,x+1]}u\le C_0\inf_{[x-1,x+1]}u$.
The dependence is uniform in $x$ because the coefficients are uniformly bounded and $q$ is uniformly separated from zero. Let $x<y$ and choose integers $x=x_0<x_1<\cdots<x_N=y$ with $|x_{j+1}-x_j|\le1$ and $N\le|x-y|+1$. Applying the local estimate on overlapping intervals yields $u(x_{j+1})\le C_0u(x_j)$ and the reverse inequality with the same constant. Hence $C_0^{-N}\le\frac{u(y)}{u(x)}\le C_0^N$.
Taking $\gamma_H=\log C_0$ and enlarging the prefactor yields the assertion.
\end{proof}

\section{Maximum principle on the whole line}

We say that $\cA_\R+\lambda$ satisfies the bounded maximum principle if every $u\in C^1(\R)$ with $\sup_\R u<+\infty$ and $\cA_\R u+\lambda u\ge0$ in $\R$ satisfies $u\le0$.

\begin{proposition}\label{prop:MPcore}
The following three estimates will be used in the proof of Theorem~\ref{thm:mainwholeline}.

\textup{(i)} Assume $\lambda<\Lp(\R)$. Suppose that there exist $W\in C^1(\R)$, $W\ge1$, $W(x)\to+\infty$ as $|x|\to\infty$, a compact interval $K$, and $\delta>0$ such that
\[
\cA_\R W+\lambda W\le-\delta W\;\hbox{in }\R\setminus K.
\]
Then $\cA_\R+\lambda$ satisfies the bounded maximum principle.

\textup{(ii)} Let $\lambda<\Lp(\R)$ and assume $\limsup_{|x|\to\infty}(a(x)+\lambda)<0$.
Then $\cA_\R+\lambda$ satisfies the bounded maximum principle. In particular, if $\Lp(\R)>0$ and $\limsup_{|x|\to\infty}a(x)<0$, then $\cA_\R$ satisfies the bounded maximum principle.

\textup{(iii)} For every $\mu\in\R$,
\[
\Lp(\cA_\R+\mu)\ge \mathfrak G(\mu).
\]
In particular, $\Lp(\cA_\R+\mu)\ge H_d(q_0)-\sup_\R(a+\mu)$.
\end{proposition}

\begin{proof}[\textbf{Proof of Proposition~\ref{prop:MPcore}\textup{(i)}}]
Let $\lambda_p=\Lp(\R)$ and let $\phi_p>0$ be the eigenfunction provided by Proposition~\ref{thm:CHbounded}. Put $\theta=\lambda_p-\lambda>0$. Then $(\cA_\R+\lambda)\phi_p=-\theta\phi_p$.
Since $\phi_p$ is continuous and positive, $m_K=\min_K\phi_p>0$. Set $C_K=\max_K|(\cA_\R+\lambda)W|$.
Choose $\varepsilon>0$ so small that $\varepsilon C_K<\theta m_K/2$ and define $\Psi=\phi_p+\varepsilon W$. On $K$,
\[
(\cA_\R+\lambda)\Psi\le-\theta m_K+\varepsilon C_K<-\frac{\theta m_K}{2}.
\]
Outside $K$,
\[
(\cA_\R+\lambda)\Psi\le-\theta\phi_p-\varepsilon\delta W<0.
\]
Thus $\Psi>0$, $\Psi(x)\to+\infty$, and $(\cA_\R+\lambda)\Psi<0$ everywhere.

Let $u$ satisfy the hypotheses of the maximum principle and suppose that $u(x_0)>0$ somewhere. Since $u$ is bounded above and $\Psi\to\infty$, the number $t_* =\max_{x\in\R}\frac{u^+(x)}{\Psi(x)}$
is positive and is attained at some $x_*\in\R$. Set $w=t_*\Psi-u$. Then $w\ge0$ and $w(x_*)=0$. Consequently $w'(x_*)=0$, and
\[
d\int_\R J(x_*-y)(w(y)-w(x_*))\,dy=d\int_\R J(x_*-y)w(y)\,dy\ge0.
\]
The zero-order contribution also vanishes at $x_*$. Hence $(\cA_\R+\lambda)w(x_*)\ge0$. On the other hand,
\[
(\cA_\R+\lambda)w=t_*(\cA_\R+\lambda)\Psi-(\cA_\R+\lambda)u<0,
\]
a contradiction.
\end{proof}

\begin{proof}[\textbf{Proof of Proposition~\ref{prop:MPcore}\textup{(ii)}}]
Fix $0<\alpha<1$ and let $W(x)=(1+x^2)^{\alpha/2}$. We prove that
\[
\frac{\cA_\R W(x)+\lambda W(x)}{W(x)}=a(x)+\lambda+o(1)
\]
as $|x|\to\infty$. First,
\[
\frac{W'(x)}{W(x)}=\frac{\alpha x}{1+x^2}=O(|x|^{-1}).
\]
For $|z|\le R_J$, Taylor's formula for $\log W$ yields
\[
\left|\log\frac{W(x-z)}{W(x)}\right|\le C\frac{|z|}{1+|x|}
\]
for $|x|\ge2R_J+1$. Hence
\[
\left|\frac{W(x-z)}{W(x)}-1\right|\le C\frac{|z|}{1+|x|}.
\]
Therefore
\[
\left|d\int_\R J(z)\left(\frac{W(x-z)}{W(x)}-1\right)\,dz\right|\le\frac{C}{1+|x|},
\]
and the drift contribution satisfies $|q(x)W'(x)/W(x)|\le C/(1+|x|)$. The asserted asymptotic follows. By the strict negativity of the limsup, there exist $R$ and $\delta>0$ such that
\[
(\cA_\R+\lambda)W\le-\delta W\;\hbox{for }|x|\ge R.
\]
Proposition~\ref{prop:MPcore}\textup{(i)} applies.
\end{proof}

\begin{proof}[\textbf{Proof of Proposition~\ref{prop:MPcore}\textup{(iii)}}]
For $s\ge0$ set $\phi_s(x)=e^{-sx}$. Since $J$ is a convolution kernel,
\[
\int_\R J(x-y)\phi_s(y)\,dy=M(s)\phi_s(x),\qquad \phi_s'(x)=-s\phi_s(x).
\]
For the operator with zeroth-order coefficient $a+\mu$,
\[
\frac{(\cA_\R+\mu)\phi_s}{\phi_s}
=d(M(s)-1)-sq(x)+a(x)+\mu.
\]
Hence every number
\[
\inf_{x\in\R}\{q(x)s-a(x)-\mu\}-d(M(s)-1)
\]
is admissible in the definition of $\Lp(\cA_\R+\mu)$. Taking the supremum over $s\ge0$ proves the first assertion. The second follows from $q\ge q_0$ and $a+\mu\le\sup_\R(a+\mu)$.
\end{proof}

\begin{proof}[\textbf{Proof of Theorem~\ref{thm:mainwholeline}\textup{(i)}}]
Apply Proposition~\ref{prop:MPcore}\textup{(iii)} to the shifted operator $\cA_\R+\mu$. The hypothesis $\mathfrak G(\mu)>0$ yields $\Lp(\cA_\R+\mu)>0$. Since $\limsup_{|x|\to\infty}(a(x)+\mu)<0$, Proposition~\ref{prop:MPcore}\textup{(ii)}, with $\lambda=0$ and zeroth-order coefficient $a+\mu$, proves the bounded maximum principle. Finally,
\[
\mathfrak G(\mu)\ge H_d(q_0)-\sup_\R(a+\mu),
\]
which proves the sufficient condition stated in the theorem.
\end{proof}

The coefficient-stability theorem has been stated among the main results in Section~1. Its proof uses the same compactness argument as the exhaustion theorem, while the quantitative estimate uses the global Harnack control from Lemma~\ref{thm:growth}.

\begin{proof}[\textbf{Proof of Theorem~\ref{thm:wholestability}}]
We first prove (i). Write $\lambda_n=\Lp(\cA_n)$ and $\lambda=\Lp(\cA)$. If $\limsup_n\lambda_n=-\infty$, there is nothing to prove. We may therefore pass to a subsequence, still indexed by $n$, such that
\[
\lambda_n\longrightarrow\lambda_*:=\limsup_{k\to\infty}\lambda_k\in\R.
\]
The sequence is bounded from above: if $\Omega\Subset\R$ is one fixed interval containing the origin, domain monotonicity yields $\lambda_n\le \Lp(\cA_n,\Omega)$, and the bounded-interval stability theorem yields a common upper bound on the right-hand side. Let $\phi_n>0$ be a whole-line principal eigenfunction, normalized by $\phi_n(0)=1$. The local Harnack estimate, the common local bounds on the coefficients and the eigenvalue equations imply that, for every compact interval $K$,
\[
C_K^{-1}\le\phi_n\le C_K,
\qquad
|\phi_n'|\le C_K,
\]
for all sufficiently large $n$. Since the kernels have a common compact support and $J_n\to J$ in $C^1$, the derivative equations make $(\phi_n')_n$ locally equicontinuous. Hence, after extraction,
\[
\phi_n\longrightarrow\phi_*
\quad\hbox{in }C^1_{\rm loc}(\R),
\qquad
\phi_*(0)=1.
\]
The convolution terms converge locally uniformly, because only a fixed compact enlargement of $K$ is involved. Passing to the limit in $\cA_n\phi_n+\lambda_n\phi_n=0$ therefore yields $\cA\phi_*+\lambda_*\phi_*=0$ in $\R$. The lower Harnack bound passes to the limit and yields $\phi_*>0$ on every compact interval. Thus $\lambda_*$ is admissible in the definition of $\Lp(\cA)$, and $\lambda_*\le\lambda$. This proves (i).

We next prove (ii). Let $\lambda_i=\Lp(\cA_i)$ and let $\phi_i>0$ be normalized principal eigenfunctions. The principal values themselves are bounded by structural constants: the constant function gives $\lambda_i\ge-\|a_i\|_\infty$, while domain monotonicity on one fixed interval, followed by the bounded-interval upper barrier in Proposition~\ref{thm:stability}, gives a common upper bound. The common structural assumptions and the whole-line Harnack estimate therefore yield a constant $C_H$, independent of $i$, such that
\[
\sup_{|y-x|\le R_J}\frac{\phi_i(y)}{\phi_i(x)}\le C_H
\quad\hbox{for all }x\in\R.
\]
Dividing the eigenvalue equation by $q_i\phi_i$ and using the common bounds for $\lambda_i$, $a_i$ and the gain term yields another structural constant $C_0$ such that
\[
\frac{|\phi_i'(x)|}{\phi_i(x)}\le C_0
\quad\hbox{for all }x\in\R.
\]
Apply these estimates to $\phi_1$. For every $x\in\R$,
\[
\begin{aligned}
\frac{|(\cA_2-\cA_1)\phi_1(x)|}{\phi_1(x)}
&\le d\int_\R |J_2(x-y)-J_1(x-y)|\frac{\phi_1(y)}{\phi_1(x)}\,dy\\
&\quad +|q_2(x)-q_1(x)|\frac{|\phi_1'(x)|}{\phi_1(x)}
+|a_2(x)-a_1(x)|\\
&\le C\Big(
\|J_2-J_1\|_{L^1}
+\|q_2-q_1\|_{L^\infty}
+\|a_2-a_1\|_{L^\infty}
\Big)=:C\delta.
\end{aligned}
\]
Since $\cA_1\phi_1+\lambda_1\phi_1=0$, we obtain $\cA_2\phi_1+(\lambda_1-C\delta)\phi_1\le0$,
and therefore $\lambda_2\ge\lambda_1-C\delta$. Exchanging the roles of the two operators yields $\lambda_1\ge\lambda_2-C\delta$, which proves the claimed estimate.

The convergence of the principal values is now immediate. The compactness argument used in (i) shows that every normalized eigenfunction subsequence has a further $C^1_{\rm loc}$ limit, and the convergence of the eigenvalues shows that every such limit is a positive eigenfunction associated with $\Lp(\cA)$. Simplicity yields convergence of the whole normalized sequence whenever it holds.
\end{proof}

\begin{lemma}\label{lem:MPcontact}
Let $w\in C^1(\R)$ satisfy $w\ge0$ and $\cA_\R w+c(x)w\le0$
for a bounded continuous $c$. If $w(x_0)=0$ at some point, then $w\equiv0$.
\end{lemma}

\begin{proof}[\textbf{Proof of Lemma~\ref{lem:MPcontact}}]
At $x_0$, $w'(x_0)=0$ and the zero-order term vanishes. Therefore $0\ge d\int_\R J(x_0-y)w(y)\,dy\ge0$.
Thus $J(x_0-y)w(y)=0$ for almost every $y$. Choose $r>0$ and $j>0$ such that $J(z)\ge j$ for $|z|<r$. Continuity and nonnegativity of $w$ imply $w=0$ on $(x_0-r,x_0+r)$. More generally, at every point $x$ of the zero set
$Z=\{w=0\}$ the same contact computation gives $w=0$ on $(x-r,x+r)$. Hence $Z$ is open; it is closed by continuity and nonempty because $x_0\in Z$. Since $\R$ is connected, $Z=\R$ and $w\equiv0$.
\end{proof}

\section{Minimal growth and simplicity}

Let $W\in C^1(\R)$ be positive and tend to $+\infty$ at both ends. We write $u=o(W)$ if $u(x)/W(x)\to0$ as $|x|\to\infty$. A positive eigenfunction $\phi_p$ associated with $\Lp(\R)$ is said to have $W$-minimal growth if, for every compact interval $K$ and every positive $v\in C^1(\R\setminus K)$ satisfying
\[
\cA_\R v+\Lp(\R)v\le0\;\hbox{in }\R\setminus K,
\]
with $v=o(W)$, domination of $\phi_p$ by $v$ on an interaction collar of $\partial K$ implies the same domination outside $K$.

\begin{proposition}\label{prop:simplicitycore}
The comparison and rigidity arguments used below are collected in the following proposition.

\textup{(i)} Assume that for some compact interval $K$, $\delta>0$ and $W\ge1$,
\[
W(x)\to+\infty,\; (\cA_\R+\lambda)W\le-\delta W\;\hbox{in }\R\setminus K.
\]
Let $U$ be one of the two components of $\R\setminus K$. If $u\in C^1(U)$ extends continuously to the interaction collar of $U$, satisfies $u^+=o(W)$ at the infinite end, $(\cA_\R+\lambda)u\ge0$ in $U$, and $u\le0$ on the finite interaction collar of width $R_J$, then $u\le0$ in $U$.

\textup{(ii)} Let $\lambda_p=\Lp(\R)$. Assume that there exist a compact interval $K$, $\delta>0$ and $W\in C^1(\R)$ such that $W\ge1$, $W(x)\to+\infty$, and
\[
(\cA_\R+\lambda_p)W\le-\delta W\;\hbox{in }\R\setminus K.
\]
Suppose that every positive eigenfunction associated with $\lambda_p$ is $o(W)$. Then every such eigenfunction has $W$-minimal growth, and $\lambda_p$ is simple in the positive cone: any two positive eigenfunctions associated with $\lambda_p$ are proportional.

\textup{(iii)} Assume the hypotheses of Proposition~\ref{prop:simplicitycore}. Let $u>0$ satisfy $u=o(W)$ and $(\cA_\R+\lambda_p)u\le0$ in $\R$, where $\lambda_p=\Lp(\R)$. Then $u$ is a positive multiple of the principal eigenfunction. In particular the inequality is necessarily an equality.

\textup{(iv)} Under the same hypotheses, let $f\in C_c(\R)$ satisfy $f\ge0$ and $f\not\equiv0$. There is no positive $u\in C^1(\R)$ with $u=o(W)$ satisfying $(\cA_\R+\lambda_p)u=-f$ in $\R$.

\textup{(v)} Assume the Lyapunov hypotheses of Proposition~\ref{prop:simplicitycore}. Let $u,v>0$, $u=o(W)$, $v=o(W)$, satisfy
\[
(\cA_\R+\lambda_p)u\ge0,\qquad (\cA_\R+\lambda_p)v\le0.
\]
If $u\le Cv$ on one compact interaction collar containing the Lyapunov core, then $u\le Cv$ on $\R$. If equality of the optimal constant is attained at one point and both differential inequalities are equalities, then $u$ and $v$ are proportional.
\end{proposition}

\begin{proof}[\textbf{Proof of Proposition~\ref{prop:simplicitycore}\textup{(i)}}]
Suppose that $u$ is positive somewhere. Since $u^+=o(W)$ at the infinite end and $u\le0$ on the finite interaction collar, the quantity $t_* =\max_U\frac{u^+}{W}$
is positive and is attained at an interior point $x_*\in U$. Put $z=t_*W-u$. Then $z\ge0$ on $U$ and on the interaction collar, and $z(x_*)=0$. Therefore $z'(x_*)=0$ and $d\int J(x_*-y)z(y)\,dy\ge0$.
The loss and zero-order terms vanish at the contact point. Hence $(\cA_\R+\lambda)z(x_*)\ge0$. On the other hand,
\[
(\cA_\R+\lambda)z=t_*(\cA_\R+\lambda)W-(\cA_\R+\lambda)u\le-t_*\delta W<0,
\]
a contradiction. Thus $u\le0$.
\end{proof}

\begin{proof}[\textbf{Proof of Proposition~\ref{prop:simplicitycore}\textup{(ii)}}]
Let $\phi>0$ be a principal eigenfunction. Fix a compact interval $K_1$ containing $K$ together with an interaction collar of width $R_J$, and let $v>0$ satisfy $(\cA_\R+\lambda_p)v\le0$
outside $K_1$, with $v=o(W)$. Suppose that $\phi\le Cv$ on the interaction collar of $\partial K_1$. On either exterior component, the function $u=\phi-Cv$ satisfies $(\cA_\R+\lambda_p)u\ge0$, $u=o(W)$ and $u\le0$ on the finite collar. Proposition~\ref{prop:simplicitycore}\textup{(i)} yields $u\le0$. Thus $\phi$ has $W$-minimal growth.

Now let $\phi,\psi>0$ be two eigenfunctions. Choose $R$ so large that $[-R,R]$ contains $K$ and both interaction collars. Set $t=\max_{[-R-R_J,R+R_J]}\frac{\phi}{\psi}$.
Then $\phi\le t\psi$ on the collars. Applying the exterior maximum principle to $u=\phi-t\psi$ on both tails yields $\phi\le t\psi$ on all of $\R$. Hence the global supremum of $\phi/\psi$ is attained on the compact interval $[-R-R_J,R+R_J]$. Let $t_* =\max_{\R}\frac{\phi}{\psi}$.
The function $w=t_*\psi-\phi$ is nonnegative, solves $(\cA_\R+\lambda_p)w=0$, and vanishes at a point where the maximum ratio is attained. The strong comparison principle yields $w\equiv0$. Therefore $\phi=t_*\psi$.
\end{proof}

\begin{proof}[\textbf{Proof of Proposition~\ref{prop:simplicitycore}\textup{(iii)}}]
Let $\phi>0$ be a principal eigenfunction. Proposition~\ref{prop:simplicitycore}\textup{(ii)} shows that $\phi$ has $W$-minimal growth. Choose $R$ so large that the strict Lyapunov inequality holds outside $[-R,R]$. On the compact interaction collar $[-R-R_J,R+R_J]$, both $u$ and $\phi$ are positive. Hence there exists $t>0$ such that $\phi\le tu$
on that collar. The function $\phi-tu$ satisfies
\[
(\cA_\R+\lambda_p)(\phi-tu)=-t(\cA_\R+\lambda_p)u\ge0
\]
on both exterior components and is $o(W)$. The exterior maximum principle yields $\phi\le tu$ on all of $\R$.

Fix one interaction collar $C$ containing the Lyapunov core and set $t_* =\max_C\frac{\phi}{u}$.
The maximum exists and is positive. Applying the exterior maximum principle on the two components of $\R\setminus C$ to $\phi-t_*u$ yields $\phi\le t_*u$ on $\R$. Hence $t_*$ is also the global maximum of $\phi/u$, and it is attained at some $x_*\in C$. Put $w=t_*u-\phi$. Then $w\ge0$, $w(x_*)=0$, and $(\cA_\R+\lambda_p)w=t_*(\cA_\R+\lambda_p)u\le0$.
At the zero $x_*$, $w'(x_*)=0$ and the loss and zero-order terms vanish. Thus $\cA_{\R}w(x_*)=d\int_\R J(x_*-y)w(y)\,dy\ge0$.
The differential inequality yields the reverse inequality. Therefore the integral vanishes, and $w$ vanishes in a neighborhood of $x_*$. At every zero the same argument applies. Positivity of $J$ near the origin propagates the zero set through the connected line, so $w\equiv0$. Hence $\phi=t_*u$. Substitution into the equations yields $(\cA_\R+\lambda_p)u=0$.
\end{proof}

\begin{proof}[\textbf{Proof of Proposition~\ref{prop:simplicitycore}\textup{(iv)}}]
If such a function existed, it would satisfy the inequality in Proposition~\ref{prop:simplicitycore}\textup{(iii)}. Hence $u=C\phi$ for the principal eigenfunction $\phi$. But then $(\cA_\R+\lambda_p)u=0$, contradicting $f\not\equiv0$.
\end{proof}

\begin{proof}[\textbf{Proof of Proposition~\ref{prop:simplicitycore}\textup{(v)}}]
On each exterior component the function $u-Cv$ is $o(W)$ and satisfies $(\cA_\R+\lambda_p)(u-Cv)\ge0$.
The collar assumption yields the boundary sign. Applying Proposition~\ref{prop:simplicitycore}\textup{(i)} separately on the two tails propagates the comparison to the whole line. If the optimal ratio is attained and both functions solve the same equation, the nonnegative difference at the optimal scaling solves the homogeneous equation and has a zero. The strong comparison theorem forces the difference to vanish identically.
\end{proof}

\begin{proposition}\label{prop:tailcore}
The exponential criteria used to close the simplicity argument are the following.

\textup{(i)} Assume the hypotheses of Lemma~\ref{thm:growth}. Let $\lambda_p=\Lp(\R)$ and choose $\beta>\gamma_H$. If
\[
\limsup_{|x|\to\infty}\left[a(x)+\lambda_p+d\big(M_J(\beta)-1\big)+q_1\beta\right]<0,
\]
where $M_J(\beta)=\int_\R J(z)e^{\beta|z|}\,dz$, then $W(x)=e^{\beta\sqrt{1+x^2}}$ satisfies $(\cA_\R+\lambda_p)W\le-\delta W$
outside a compact interval for some $\delta>0$. Every positive eigenfunction associated with $\lambda_p$ is $o(W)$. Consequently $\lambda_p$ is simple in the positive cone.

\textup{(ii)} Let $\lambda_p=\Lp(\R)$. Assume that every positive principal eigenfunction satisfies the exponential Harnack bound of Lemma~\ref{thm:growth} with exponent $\gamma_H$. Suppose there exist $\beta_+,\beta_->\gamma_H$ such that
\[
\limsup_{x\to+\infty}\left[a(x)+\lambda_p+d(M(\beta_+)-1)+q(x)\beta_+\right]<0
\]
and
\[
\limsup_{x\to-\infty}\left[a(x)+\lambda_p+d(M(\beta_-)-1)-q(x)\beta_-\right]<0,
\]
where $M(\beta)=\int_\R J(z)e^{\beta z}\,dz$; symmetry makes $M$ even. Then $\lambda_p$ is simple in the positive cone.
\end{proposition}

\begin{proof}[\textbf{Proof of Proposition~\ref{prop:tailcore}\textup{(i)}}]
Let $r(x)=\sqrt{1+x^2}$. Since $|r'(x)|\le1$ and $|r(x-z)-r(x)|\le|z|$,
\[
\frac{W(x-z)}{W(x)}=e^{\beta(r(x-z)-r(x))}\le e^{\beta|z|}.
\]
Also $W'/W=\beta r'(x)$, hence $qW'/W\le q_1\beta$. Therefore
\[
\frac{(\cA_\R+\lambda_p)W}{W}\le d\int_\R J(z)(e^{\beta|z|}-1)\,dz+q_1\beta+a(x)+\lambda_p.
\]
The tail hypothesis yields the strict negative bound. By Lemma~\ref{thm:growth}, every positive eigenfunction $u$ satisfies $u(x)\le Ce^{\gamma_H|x|}$ after normalization at the origin. Since $\beta>\gamma_H$, $u/W\to0$. Proposition~\ref{prop:simplicitycore}\textup{(ii)} now yields simplicity.
\end{proof}

\begin{proof}[\textbf{Proof of Proposition~\ref{prop:tailcore}\textup{(ii)}}]
Choose $R>1$ so large that the first strict inequality holds on $(R,\infty)$ and the second on $(-\infty,-R)$. Construct $W\in C^1(\R)$, $W\ge1$, so that $W(x)=e^{\beta_+x}\quad x\ge R+R_J$,
and $W(x)=e^{-\beta_-x}\quad x\le-R-R_J$,
with an arbitrary positive smooth interpolation on the middle interval.

For $x\ge R+2R_J$, every point $x-z$ occurring in the convolution remains in the right exponential region. Therefore
\[
\frac{J*W(x)}{W(x)}=\int J(z)e^{-\beta_+z}\,dz=M(\beta_+),\qquad \frac{W'(x)}{W(x)}=\beta_+.
\]
Thus
\[
\frac{(\cA_\R+\lambda_p)W(x)}{W(x)}
=d(M(\beta_+)-1)+q(x)\beta_++a(x)+\lambda_p,
\]
which is bounded above by a negative constant for sufficiently large $x$.

For $x\le-R-2R_J$ we have instead
\[
\frac{J*W(x)}{W(x)}=M(\beta_-),\qquad \frac{W'(x)}{W(x)}=-\beta_-.
\]
Hence
\[
\frac{(\cA_\R+\lambda_p)W(x)}{W(x)}
=d(M(\beta_-)-1)-q(x)\beta_-+a(x)+\lambda_p< -\delta
\]
for large negative $x$. After enlarging the middle compact interval, $W$ is a strict Lyapunov supersolution on both tails.

Let $\phi$ be any positive principal eigenfunction normalized at the origin. Lemma~\ref{thm:growth} yields $\phi(x)\le Ce^{\gamma_H|x|}$.
Since $\beta_+,\beta_->\gamma_H$, $\frac{\phi(x)}{W(x)}\longrightarrow0$
at both ends. Proposition~\ref{prop:simplicitycore}\textup{(ii)} applies and proves simplicity.
\end{proof}

\begin{proof}[\textbf{Proof of Theorem~\ref{thm:mainwholeline}\textup{(ii)}}]
Fix $\Omega_0=(\alpha_0,\beta_0)$. We first justify that the class of admissible $\kappa$'s in the statement is nonempty. Since $J(0)>0$ and $J$ is continuous, we can choose $\rho>0$ so small that
\[
0<\rho<\min\{R_J,\beta_0-\alpha_0\},\qquad
j_0:=\min_{|z|\le\rho}J(z)>0.
\]
For $y\in[\beta_0-\rho,\beta_0-\rho/2]$ one has $J(\beta_0-y)\ge j_0$ and $\beta_0-y\in[\rho/2,\rho]$. Hence
\[
\begin{aligned}
B_{\Omega_0,\kappa}
&\ge d j_0\int_{\rho/2}^{\rho}(e^{\kappa t}-1)\,dt-q(\beta_0)\kappa\\
&=\frac{d j_0}{\kappa}\big(e^{\kappa\rho}-e^{\kappa\rho/2}\big)
-\frac{d j_0\rho}{2}-q(\beta_0)\kappa.
\end{aligned}
\]
The right-hand side tends to $+\infty$ as $\kappa\to+\infty$. Thus $B_{\Omega_0,\kappa}>0$ for all sufficiently large $\kappa$.

Fix such a $\kappa$ and write $w=w_\kappa$. Since $w>0$ in $\Omega_0$ and $w(\beta_0)=0$, while
\[
\lim_{x\uparrow\beta_0}\cA_{\Omega_0}w(x)
=d\int_{\Omega_0}J(\beta_0-y)w(y)\,dy-q(\beta_0)\kappa
=B_{\Omega_0,\kappa}>0,
\]
we have
\[
-\frac{\cA_{\Omega_0}w(x)}{w(x)}\longrightarrow-\infty
\qquad (x\uparrow\beta_0).
\]
On every compact subinterval of $\Omega_0$ the quotient is continuous. Therefore
\[
\mathfrak U_{\Omega_0,\kappa}
=\sup_{x\in\Omega_0}\left(-\frac{\cA_{\Omega_0}w(x)}{w(x)}\right)
\]
is a finite real number and, by definition,
\[
\cA_{\Omega_0}w+\mathfrak U_{\Omega_0,\kappa}w\ge0
\qquad\hbox{in }\Omega_0.
\]
The bounded-interval comparison theorem, Proposition~\ref{thm:comparison}, yields $\Lp(\Omega_0)\le\mathfrak U_{\Omega_0,\kappa}$,
and domain monotonicity yields
\begin{equation}
\Lp(\R)\le\mathfrak U_{\Omega_0,\kappa}
\label{eq:4_2}
\end{equation}
for every admissible pair $(\Omega_0,\kappa)$. Taking the infimum proves $\Lp(\R)\le\mathfrak U_*$.
The admissible class is nonempty by the first part of the proof, and every cap is finite. Moreover \eqref{eq:4_2} shows that all caps are bounded below by the finite number $\Lp(\R)$. Hence $\mathfrak U_*\in\R$.

There is also a coefficient lower bound sufficient for the Harnack constants. Since $J*1=1$ on $\R$, $\cA_\R1=a(x)$,
and hence the constant function is admissible at the level $-\sup_\R a$. Thus
\begin{equation}
-\|a\|_\infty\le-\sup_\R a\le\Lp(\R)\le\mathfrak U_*.
\label{eq:4_3}
\end{equation}
It follows that $|\Lp(\R)|\le\Lambda_*$. The uniform Harnack theorem, Lemma~\ref{thm:growth}, therefore supplies constants $C_*$ and $\gamma_*$ depending only on $d,\ J,\ q_0,\ q_1,\ \|a\|_\infty,\ \Lambda_*$,
such that every positive principal eigenfunction $\phi$, normalized by $\phi(0)=1$, satisfies
\begin{equation}
\phi(x)\le C_*e^{\gamma_*|x|}.
\label{eq:4_4}
\end{equation}

Let $\lambda_p=\Lp(\R)$. The two tail assumptions of Theorem~\ref{thm:mainwholeline} and $\lambda_p\le\mathfrak U_*$ yield
\[
\limsup_{x\to+\infty}
\left[a(x)+\lambda_p+d(M(\beta_+)-1)+q(x)\beta_+\right]<0
\]
and
\[
\limsup_{x\to-\infty}
\left[a(x)+\lambda_p+d(M(\beta_-)-1)-q(x)\beta_-\right]<0.
\]
These are exactly the strict tail inequalities of Proposition~\ref{prop:tailcore}\textup{(ii)}. Since $\beta_+,\beta_->\gamma_*$, \eqref{eq:4_4} is precisely the growth condition needed there. Proposition~\ref{prop:tailcore}\textup{(ii)} therefore yields simplicity of $\Lp(\cA_\R)$ in the positive cone.

The argument also shows the non-circular character and the optimization built into the criterion. Each $\mathfrak U_{\Omega_0,\kappa}$ is computed from $J,d,q,a$ on one finite interval through the explicit terminal test function $w_\kappa$, and the infimum over all such tests is taken before a whole-line eigenfunction is used. The principal eigenfunction enters only after \eqref{eq:4_3} has fixed a data-dependent Harnack class.
\end{proof}

\begin{proposition}\label{thm:largeintervalrate}
Assume $q\equiv c>0$ and $a\equiv a_0$. Define $H_d(c)=\sup_{s\in\R}\{cs-d(M(s)-1)\}$.
Let $s_c>0$ be the unique solution of $dM'(s_c)=c$. There exist $R_0>0$ and $C>0$, depending only on $J,d,c$, such that for $R\ge R_0$,
\[
H_d(c)-a_0\le\Lp((-R,R))\le H_d(c)-a_0+\frac{C}{R^2}.
\]
Consequently $\Lp((-R,R))\downarrow H_d(c)-a_0$
as $R\to\infty$.
\end{proposition}

\begin{proof}[\textbf{Proof of Proposition~\ref{thm:largeintervalrate}}]
We use an exact exponential conjugation with a parameter independent of $R$. Put $\kappa=s_c$ and write $\phi(x)=e^{-s_c(x+R)}u(x)$.
A direct calculation yields
\[
\cA\phi=e^{-s_c(x+R)}\left(\cB_Ru+a_0u-H_d(c)u\right),
\]
where
\[
\cB_Ru=d\left(\int_{-R}^{R}J(x-y)e^{s_c(x-y)}u(y)\,dy-M(s_c)u(x)\right)+cu'(x).
\]
Thus
\[
\Lp(\cA,(-R,R))=H_d(c)-a_0+\Lp(\cB_R,(-R,R)).
\]
We prove $0\le\Lp(\cB_R,(-R,R))\le C/R^2$.

The constant function yields the lower bound because
\[
\cB_R1=d\left(\int_{-R}^{R}J(x-y)e^{s_c(x-y)}\,dy-M(s_c)\right)\le0.
\]
For the upper bound let $k_R=\frac{\pi}{2R},\qquad w_R(x)=\sin(k_R(x+R))$.
The function is positive in $(-R,R)$ and vanishes at both endpoints. Let $W_R$ be the same sine function on the whole line. At points satisfying $\dist(x,\partial(-R,R))\ge R_J$, the full and truncated convolutions agree. Using
\[
W_R(x-z)=W_R(x)\cos(k_Rz)-\cos(k_R(x+R))\sin(k_Rz),
\]
we obtain $\cB_RW_R=A_R W_R+B_R\cos(k_R(x+R))$,
where $A_R=d\int J(z)e^{s_cz}(\cos(k_Rz)-1)\,dz$
and $B_R=ck_R-d\int J(z)e^{s_cz}\sin(k_Rz)\,dz$.
Taylor's theorem and the centering identity $dM'(s_c)=c$ yield $A_R=-\frac{dk_R^2}{2}M''(s_c)+O(k_R^4)$,
and
\[
B_R=ck_R-dk_RM'(s_c)+\frac{dk_R^3}{6}M^{(3)}(s_c)+O(k_R^5)=\frac{dk_R^3}{6}M^{(3)}(s_c)+O(k_R^5).
\]
Hence
\[
|A_R|\le\frac{C}{R^2},\qquad |B_R|\le\frac{C}{R^3}.
\]
If $w_R(x)\ge C_0/R$, then the $B_R$ term is bounded by $C R^{-2}w_R(x)$, and therefore $\cB_Rw_R+\frac{C}{R^2}w_R\ge0$.

It remains to treat points within distance $O(1)$ of the endpoints where $w_R$ is small. The full sine continuation is negative immediately outside both endpoints. Therefore replacing the full convolution by the interval convolution removes a negative contribution and creates a nonnegative correction. Near $R$, write $x=R-t$ with $0\le t\le R_J$. Then $-W_R(R-t-z)=k_R(-t-z)+O(k_R^3|t+z|^3)$
whenever $z<-t$, and the truncation correction equals
\[
d\int_{z<-t}J(z)e^{s_cz}\big(-W_R(R-t-z)\big)\,dz.
\]
For $t\le t_0$ with fixed small $t_0$, positivity of $J$ near the origin yields a lower bound $c k_R$, which dominates the full-space error of order $R^{-3}$. For $t_0\le t\le R_J$, the sine satisfies $w_R(x)\ge c/R$ and the preceding interior estimate applies after increasing $C$. The left endpoint is identical. We conclude that
\[
\cB_Rw_R+\frac{C}{R^2}w_R\ge0\quad \hbox{in }(-R,R).
\]
The bounded-interval comparison theorem yields $\Lp(\cB_R)\le C/R^2$, proving the assertion.
\end{proof}

\begin{proposition}\label{thm:plateausimple}
Let $\Omega_0=(\alpha_0,\beta_0)\Subset\R$, $\kappa>0$, and
$w_\kappa(x)=e^{\kappa(\beta_0-x)}-1$. Assume
\[
B_{\Omega_0,\kappa}
=d\int_{\Omega_0}J(\beta_0-y)w_\kappa(y)\,dy-q(\beta_0)\kappa>0
\]
and put
\[
\mathfrak C_{\Omega_0,\kappa}
:=\sup_{x\in\Omega_0}
\left[
-\frac{d\{\int_{\Omega_0}J(x-y)w_\kappa(y)\,dy-w_\kappa(x)\}
+q(x)w_\kappa'(x)}{w_\kappa(x)}
\right].
\]
Suppose $a(x)\ge a_*$ on $\Omega_0$. Set
$\Lambda_0=\max\{\|a\|_\infty,|\mathfrak C_{\Omega_0,\kappa}-a_*|\}$,
and let $\gamma_0$ be a Harnack growth exponent corresponding to
$d,J,q_0,q_1,\|a\|_\infty,\Lambda_0$. If there exist
$\beta_+,\beta_->\gamma_0$ such that
\[
\mathfrak C_{\Omega_0,\kappa}-a_*+
\limsup_{x\to+\infty}\{a(x)+d(M(\beta_+)-1)+q(x)\beta_+\}<0
\]
and
\[
\mathfrak C_{\Omega_0,\kappa}-a_*+
\limsup_{x\to-\infty}\{a(x)+d(M(\beta_-)-1)-q(x)\beta_-\}<0,
\]
then $\Lp(\cA_\R)$ is simple in the positive cone.
\end{proposition}

\begin{proof}[\textbf{Proof of Proposition~\ref{thm:plateausimple}}]
The condition $B_{\Omega_0,\kappa}>0$ implies, as in the proof of
Theorem~\ref{thm:mainwholeline}, that $\mathfrak C_{\Omega_0,\kappa}$ is finite.
Since $a\ge a_*$ on $\Omega_0$,
\[
-\frac{\cA_{\Omega_0}w_\kappa}{w_\kappa}
\le\mathfrak C_{\Omega_0,\kappa}-a_*.
\]
The bounded comparison theorem and domain monotonicity imply
\[
-\|a\|_\infty\le\Lp(\cA_\R)
\le\mathfrak C_{\Omega_0,\kappa}-a_*.
\]
Thus $|\Lp(\cA_\R)|\le\Lambda_0$, and Lemma~\ref{thm:growth} applies with
the exponent $\gamma_0$. The two assumptions at infinity are the strict tail
inequalities of Proposition~\ref{prop:tailcore}\textup{(ii)} with $\lambda=\Lp(\cA_\R)$.
The conclusion follows.
\end{proof}

\begin{proposition}\label{prop:perturbcore}
The fixed-domain perturbation formulas used later are collected here.

\textup{(i)} Assume the hypotheses of Proposition~\ref{prop:simplicitycore}, so that the principal eigenvalue is simple in the positive cone. Let $h\in C_c(\R)$ satisfy $h\ge0$ and $h\not\equiv0$. Then $\Lp(\cA_\R+h)<\Lp(\cA_\R)$.

\textup{(ii)} Let $b\in C_c(\R)$ satisfy $b\ge0$ and $b\not\equiv0$, and assume that the Lyapunov hypotheses of Proposition~\ref{prop:simplicitycore} hold for every operator $\cA_\R+rb$ under consideration. Then $r\longmapsto\Lp(\cA_\R+rb)$
is strictly decreasing and locally Lipschitz. If there exist $r_-<r_+$ such that the two values have opposite signs, there is a unique $r_*$ with $\Lp(\cA_\R+r_*b)=0$.

\textup{(iii)} Let $\Omega=(\alpha,\beta)$ be fixed. Let
\[
\cA_\tau u=d\int_\Omega\big(J(x-y)+\tau H(x,y)\big)u(y)\,dy+\big(q(x)+\tau r(x)\big)u'(x)+\big(b(x)+\tau h(x)\big)u(x),
\]
where $q,r\in C^1(\overline \Omega)$, $b,h\in C(\overline \Omega)$ and $H\in C(\overline \Omega^2)$.  We restrict $\tau$ to a real interval on which the drift is bounded away from zero and the kernels $J(x-y)+\tau H(x,y)$ are nonnegative and satisfy a common positivity condition in a neighborhood of the diagonal. Let $\lambda(\tau)=\Lp(\cA_\tau)$. Then the principal eigenvalue and the normalized principal eigenfunction are real analytic in $\tau$ on that interval. If $\phi,\phi^*>0$ are the direct and adjoint principal eigenfunctions at $\tau=0$, normalized by $\int_\Omega\phi\phi^*\,dx=1$,
then
\[
\begin{aligned}
\lambda'(0)={}&-d\int_\Omega\int_\Omega\phi^*(x)H(x,y)\phi(y)\,dy\,dx\\
&-\int_\Omega r(x)\phi'(x)\phi^*(x)\,dx
-\int_\Omega h(x)\phi(x)\phi^*(x)\,dx.
\end{aligned}
\]
In particular, for a potential perturbation alone,
\[
\lambda'(0)=-\int_\Omega h(x)\phi(x)\phi^*(x)\,dx.
\]

\textup{(iv)} On a fixed bounded interval let
\[
\cA_{d,c}u=d\left(\int_\Omega J(x-y)u(y)\,dy-u(x)\right)+c\,r(x)u'(x)+a(x)u(x),
\]
where $d>0$, $c$ ranges in an interval on which $cr(x)$ has a fixed sign bounded away from zero, and $r\in C^1(\overline \Omega)$. Normalize the direct and adjoint eigenfunctions by $\int_\Omega\phi\phi^*=1$. Then the principal eigenvalue and the normalized eigenfunction are real analytic in $(d,c)$ on every connected parameter region on which $d>0$ and $cr$ keeps a fixed sign bounded away from zero, and
\[
\partial_d\Lp=-\int_\Omega\phi^*(x)\left(\int_\Omega J(x-y)\phi(y)\,dy-\phi(x)\right)dx,
\]
\[
\partial_c\Lp=-\int_\Omega r(x)\phi'(x)\phi^*(x)\,dx.
\]
No sign is asserted for these derivatives in general.

\textup{(v)} For the scaled operator
\[
\cA_{\sigma,m,L}u=\frac1{\sigma^m}\left(\int_{-L}^{L}J_\sigma(x-y)u(y)\,dy-u(x)\right)+q(x)u'(x)+a(x)u(x),
\]
put $\rho=\sigma/L$ and $v(\xi)=u(L\xi)$. Then
\[
\Lp(\cA_{\sigma,m,L})=\Lp\left(L^{-m}\rho^{-m}(\mathcal J_\rho-I)+L^{-1}q(L\cdot)\partial_\xi+a(L\cdot),(-1,1)\right),
\]
where
\[
\mathcal J_\rho v(\xi)=\int_{-1}^{1}J_\rho(\xi-\eta)v(\eta)\,d\eta.
\]
\end{proposition}

\begin{proof}[\textbf{Proof of Proposition~\ref{prop:perturbcore}\textup{(i)}}]
Monotonicity of the defining inequality yields $\Lp(\cA_\R+h)\le\Lp(\cA_\R)$. Suppose equality holds and denote the common value by $\lambda_p$. Let $\phi>0$ and $\psi>0$ be principal eigenfunctions of $\cA_\R$ and $\cA_\R+h$, respectively. By the hypotheses of Proposition~\ref{prop:simplicitycore}\textup{(ii)}, $\phi=o(W)$ for the coercive weight $W$. Choose $R$ so that $\supp h\subset[-R,R]$ and so that the exterior maximum principle applies outside this interval. Set $t=\max_{[-R-R_J,R+R_J]}\frac{\phi}{\psi}$.
Outside $[-R,R]$ the two functions satisfy the same eigen-equation. For $z=\phi-t\psi$ one has $z^+\le\phi=o(W)$, so the exterior maximum principle yields $\phi\le t\psi$ on the whole line. Let $t_*$ be the smallest such constant. As in the proof of Proposition~\ref{prop:simplicitycore}\textup{(ii)}, the extremal ratio is attained on the compact collar. Put $w=t_*\psi-\phi$. Then $w\ge0$, $w$ vanishes somewhere, and $(\cA_\R+\lambda_p)w=-t_*h\psi\le0$.
At a zero $x_0$ of $w$, the drift and zero-order terms vanish and positivity of the kernel yields $\cA_\R w(x_0)\ge0$. Since $(\cA_\R+\lambda_p)w=-t_*h\psi\le0$, both sides at $x_0$ must vanish; in particular $h(x_0)=0$ and the nonlocal integral of $w$ is zero. Hence $w$ vanishes on a neighborhood of $x_0$. The same argument applies at every zero, so the zero set of $w$ is open; it is also closed and nonempty. Connectedness of $\R$ gives $w\equiv0$. Substitution in the equation then gives $h\psi\equiv0$, contradicting $h\not\equiv0$ and $\psi>0$. Therefore the inequality between the two principal values is strict.
\end{proof}

\begin{proof}[\textbf{Proof of Proposition~\ref{prop:perturbcore}\textup{(ii)}}]
Strict monotonicity follows from Proposition~\ref{prop:perturbcore}\textup{(i)}. The Lipschitz estimate
\[
|\Lp(a+r_1b)-\Lp(a+r_2b)|\le|r_1-r_2|\|b\|_\infty
\]
follows from Lemma~\ref{thm:monotonicity}. Continuity and the intermediate value theorem yield existence of $r_*$, and strict monotonicity yields uniqueness.
\end{proof}

\begin{proof}[\textbf{Proof of Proposition~\ref{prop:perturbcore}\textup{(iii)}}]
The proof uses the same Fredholm structure as the domain derivative below. Set $X=\{u\in C^1(\overline \Omega):u(\beta)=0\}$ and $Z=C(\overline \Omega)$. The transport operator $q\partial_x+(b+\lambda)$ is an isomorphism from $X$ onto $Z$. The integral term is compact. The full kernel statement in Proposition~\ref{thm:CHbounded} and the positive adjoint eigenfunction show that the augmented linearization of the normalized eigenproblem is invertible. After complexification, $\tau\mapsto\cA_\tau$ is an entire operator-valued affine map from $X_{\mathbb C}$ to $Z_{\mathbb C}$. The augmented linearization of the normalized eigenproblem is Fredholm of index zero and injective, hence invertible, exactly as in the proof of Proposition~\ref{prop:sizeanalytic}. The analytic implicit-function theorem therefore yields a holomorphic eigenpair branch near every real parameter value. Positivity on the real axis identifies it with the principal branch, which proves real analyticity.

Differentiate $\cA_\tau\phi_\tau+\lambda(\tau)\phi_\tau=0$
at $\tau=0$. Writing a dot for $\partial_\tau|_{0}$ yields
\[
(\cA_0+\lambda(0))\dot\phi+\dot\lambda\phi+d\int_\Omega H(x,y)\phi(y)\,dy+r(x)\phi'(x)+h(x)\phi(x)=0.
\]
Multiply by $\phi^*$ and integrate. The first term vanishes by the Green identity because $\dot\phi(\beta)=0$ and $\phi^*(\alpha)=0$. The normalization of the direct-adjoint product is one, and the displayed formula follows. The potential-only formula is the special case $H=r=0$.
\end{proof}

\begin{proof}[\textbf{Proof of Proposition~\ref{prop:perturbcore}\textup{(iv)}}]
The proof of Proposition~\ref{prop:perturbcore}\textup{(iii)} applies verbatim to the two complex parameters $(d,c)$, because $(d,c)\mapsto\cA_{d,c}$ is affine in operator norm and the same augmented linearization is invertible at every real point of the fixed-sign parameter region. Hence the principal eigenpair is real analytic in $(d,c)$. Differentiation with respect to $d$ yields
\[
(\cA_{d,c}+\lambda)\partial_d\phi+\partial_d\lambda\,\phi+\left(\int_\Omega J(x-y)\phi(y)\,dy-\phi(x)\right)=0.
\]
Pairing with the adjoint eigenfunction eliminates the first term and yields the first formula. The computation with respect to $c$ is identical:
\[
(\cA_{d,c}+\lambda)\partial_c\phi+\partial_c\lambda\,\phi+r(x)\phi'(x)=0,
\]
and pairing yields the second formula. The lack of a general sign reflects the non-self-adjoint character of the drift problem: increasing dispersal or increasing drift may move the principal level in either direction in heterogeneous environments.
\end{proof}

\begin{proof}[\textbf{Proof of Proposition~\ref{prop:perturbcore}\textup{(v)}}]
With $x=L\xi$, $y=L\eta$ and $\sigma=L\rho$,
\[
LJ_\sigma(L(\xi-\eta))=\frac{L}{\sigma}J\left(\frac{L(\xi-\eta)}{\sigma}\right)=\frac1\rho J\left(\frac{\xi-\eta}{\rho}\right)=J_\rho(\xi-\eta).
\]
Moreover $u'(L\xi)=L^{-1}v'(\xi)$ and $\sigma^{-m}=L^{-m}\rho^{-m}$. Substitution in the defining differential inequality for $\Lp$ yields the identity in both directions because the pull-back is a bijection on positive $C^1$ test functions.
\end{proof}

\section{Dependence on the size of a bounded interval}

For $L>0$ set $\Omega_L=(-L,L)$ and consider
\[
\cA_Lu(x)=d\int_{-L}^{L}J(x-y)u(y)\,dy+q(x)u'(x)+b(x)u(x),
\]
where the loss term can be included in $b$. We assume in this section that $J\in C_c^2(\R)$, $q,b\in C^2(\R)$ and $q\ge q_0>0$. We write $\lambda_p(L)=\Lp(\cA_L)$ and choose the direct eigenfunction $\phi_L$ and the adjoint eigenfunction $\phi_L^*$ by
\[
\phi_L(0)=1,\quad \int_{-L}^{L}\phi_L(x)\phi_L^*(x)\,dx=1.
\]
The direct and adjoint boundary conditions are $\phi_L(L)=0$ and $\phi_L^*(-L)=0$.

For the analytic dependence result we use the following stronger, local hypothesis. Let $\mathcal L\subset(0,\infty)$ be open.

\medskip
\noindent\textbf{$\mathrm{(A_{an})}$.}
For every compact $K\Subset\mathcal L$ there exist $\rho_K>0$ and open sets $U_{J,K},U_{q,K},U_{b,K}\subset\mathbb C$ such that $D_K:=\{z\in\mathbb C:\operatorname{dist}(z,K)<\rho_K\}$ does not meet $0$. The functions $J,q,b$ admit bounded holomorphic extensions to $U_{J,K}$, $U_{q,K}$ and $U_{b,K}$, respectively. In addition,
\[
\begin{aligned}
\{z(\xi-\eta):z\in D_K,\ \xi,\eta\in[-1,1]\}&\Subset U_{J,K},\\
\{z\xi:z\in D_K,\ \xi\in[-1,1]\}&\Subset U_{q,K}\cap U_{b,K}.
\end{aligned}
\]
This formulation is deliberately local in the length. Under the standing compact-support assumption it is, for example, satisfied on every $\mathcal L\Subset(0,R_J/2)$ if $J$ is real analytic on $(-R_J,R_J)$ and $q,b$ are real analytic on the relevant real neighborhoods. A nonzero compactly supported function cannot be analytic through an endpoint of its support. Thus no analyticity statement on all of $(0,\infty)$ is available from $J\in C_c^2(\R)$ alone.

\begin{proposition}\label{prop:sizecore}
We collect the adjoint and interval-size differentiability statements in one proposition.

\textup{(i)} The formal adjoint is
\[
\cA_L^*v=d\int_{-L}^{L}J(x-y)v(y)\,dy-(qv)'(x)+b(x)v(x).
\]
It possesses a positive principal eigenfunction $\phi_L^*$ satisfying
\[
\cA_L^*\phi_L^*+\lambda_p(L)\phi_L^*=0\ \hbox{in }(-L,L),\quad \phi_L^*(-L)=0,
\]
and this eigenfunction is unique up to multiplication. For all $u,v\in C^1([-L,L])$ satisfying $u(L)=0$ and $v(-L)=0$,
\[
\int_{-L}^{L}v\cA_Lu\,dx=\int_{-L}^{L}u\cA_L^*v\,dx.
\]
Moreover,
\[
\phi_L'(L)<0,\quad \phi_L^*(L)>0,\quad \phi_L(-L)>0.
\]

\textup{(ii)} Let $\Omega_{\alpha,\beta}=(\alpha,\beta)$ and let $\lambda(\alpha,\beta)=\Lp(\cA_{\Omega_{\alpha,\beta}})$. Assume the coefficients are $C^2$ and normalize the direct and adjoint eigenfunctions by
\[
\int_\alpha^\beta\phi_{\alpha,\beta}(x)\phi_{\alpha,\beta}^*(x)\,dx=1.
\]
Then $\lambda$ is $C^1$ in $(\alpha,\beta)$ and
\[
\partial_\beta\lambda(\alpha,\beta)=q(\beta)\phi_{\alpha,\beta}'(\beta)\phi_{\alpha,\beta}^*(\beta)<0,
\]
while
\[
\partial_\alpha\lambda(\alpha,\beta)=d\phi_{\alpha,\beta}(\alpha)\int_\alpha^\beta J(x-\alpha)\phi_{\alpha,\beta}^*(x)\,dx>0.
\]
Thus moving the outflow endpoint to the right and moving the inflow endpoint to the left both strictly decrease the principal eigenvalue, but the two shape derivatives have different structures.

\textup{(iii)} The maps $L\mapsto\lambda_p(L)$ and, after pull-back to $(-1,1)$ and the above normalization, $L\mapsto\phi_L$, $L\mapsto\phi_L^*$ are continuous. On every compact interval $[L_0,L_1]\Subset(0,\infty)$ the convergence is uniform in $C^1$ for the eigenfunctions.

\textup{(iv)} The map $L\mapsto\lambda_p(L)$ is of class $C^1$ on $(0,\infty)$. With the normalization above,
\[
\partial_L\lambda_p(L)=q(L)\phi_L'(L)\phi_L^*(L)-d\phi_L(-L)\int_{-L}^{L}J(x+L)\phi_L^*(x)\,dx.
\]
In particular, $\partial_L\lambda_p(L)<0\quad L>0$.
Thus $L\mapsto\lambda_p(L)$ is strictly decreasing.
\end{proposition}

\begin{proof}[\textbf{Proof of Proposition~\ref{prop:sizecore}\textup{(i)}}]
The Green identity follows from two exact computations. Symmetry of $J$ and Fubini's theorem yield
\[
\int_{-L}^{L}v(x)\int_{-L}^{L}J(x-y)u(y)\,dy\,dx=\int_{-L}^{L}u(y)\int_{-L}^{L}J(y-x)v(x)\,dx\,dy.
\]
For the drift term,
\[
\int_{-L}^{L}vqu'\,dx=[qvu]_{-L}^{L}-\int_{-L}^{L}u(qv)'\,dx=-\int_{-L}^{L}u(qv)'\,dx,
\]
because $u(L)=0$ and $v(-L)=0$. This identifies $\cA_L^*$.

The drift in the adjoint equation is $-q$. Reflecting $x$ to $-x$ converts it into a positive fixed-sign drift, so Proposition~\ref{thm:CHbounded} yields a positive adjoint principal eigenfunction with boundary value zero at $-L$. Let its eigenvalue be $\mu_L$. Pairing the direct and adjoint equations by the Green identity yields
\[
(\lambda_p(L)-\mu_L)\int_{-L}^{L}\phi_L\phi_L^*\,dx=0.
\]
The integral is positive; hence $\mu_L=\lambda_p(L)$. Uniqueness follows from the fixed-sign bounded-interval theorem.

The values $\phi_L(-L)$ and $\phi_L^*(L)$ are strictly positive because the respective boundary condition is imposed only at the opposite endpoint. Finally, evaluate the direct equation as $x\uparrow L$. Since $\phi_L(L)=0$ and the gain integral has a strictly positive limit,
\[
q(L)\phi_L'(L)=-d\int_{-L}^{L}J(L-y)\phi_L(y)\,dy<0.
\]
Thus $\phi_L'(L)<0$.
\end{proof}

\begin{proof}[\textbf{Proof of Proposition~\ref{prop:sizecore}\textup{(ii)}}]
We first justify the $C^1$ dependence needed for differentiation. Put
$m=(\alpha+\beta)/2$, $\ell=(\beta-\alpha)/2$ and pull the problem back by $x=m+\ell\xi$ to $(-1,1)$. The resulting operator depends $C^1$ on $(m,\ell)$ in $\mathcal L(C^1,C)$ because $J,q,b$ are $C^2$. With the fixed normalization $v(0)=1$, the derivative of the augmented eigenproblem in $(\lambda,v)$ is Fredholm of index zero; pairing with the positive adjoint eigenfunction proves injectivity exactly as in part~\textup{(iv)}, hence it is invertible. The implicit-function theorem gives a local $C^1$ eigenpair branch, and uniqueness of the positive principal eigenpair patches these branches over the parameter region $\alpha<\beta$. We now derive the endpoint formulas directly in physical variables.

Keep $\alpha$ fixed and differentiate with respect to $\beta$. Write a dot for $\partial_\beta$. At a fixed interior point $x$,
\[
\partial_\beta\int_\alpha^\beta J(x-y)\phi(y)\,dy=J(x-\beta)\phi(\beta)+\int_\alpha^\beta J(x-y)\dot\phi(y)\,dy.
\]
The first term vanishes because $\phi(\beta)=0$. Hence $(\cA+\lambda)\dot\phi+\dot\lambda\phi=0$.
The moving outflow condition yields $\dot\phi(\beta)+\phi'(\beta)=0$.
Multiply the differentiated equation by $\phi^*$ and integrate. Since $\dot\phi$ no longer vanishes at $\beta$, integration by parts produces
\[
\int_\alpha^\beta\phi^*(\cA+\lambda)\dot\phi\,dx=q(\beta)\phi^*(\beta)\dot\phi(\beta).
\]
Using the normalization of the direct-adjoint product yields
\[
q(\beta)\phi^*(\beta)\dot\phi(\beta)+\dot\lambda=0.
\]
Substitution of $\dot\phi(\beta)=-\phi'(\beta)$ yields
\[
\partial_\beta\lambda=q(\beta)\phi'(\beta)\phi^*(\beta).
\]
The sign is strictly negative by Proposition~\ref{prop:sizecore}\textup{(i)}.

Now keep $\beta$ fixed and differentiate with respect to $\alpha$. Since the lower limit moves,
\[
\partial_\alpha\int_\alpha^\beta J(x-y)\phi(y)\,dy=-J(x-\alpha)\phi(\alpha)+\int_\alpha^\beta J(x-y)\dot\phi(y)\,dy.
\]
The outflow boundary is fixed, so $\dot\phi(\beta)=0$. We obtain
\[
(\cA+\lambda)\dot\phi+\dot\lambda\phi-dJ(x-\alpha)\phi(\alpha)=0.
\]
Pairing with $\phi^*$ makes the first term vanish because the direct variation and the adjoint eigenfunction satisfy the complementary boundary conditions. Therefore
\[
\dot\lambda=d\phi(\alpha)\int_\alpha^\beta J(x-\alpha)\phi^*(x)\,dx.
\]
Every factor is positive. This proves the second formula and its sign.
\end{proof}

\begin{proof}[\textbf{Proof of Proposition~\ref{prop:sizecore}\textup{(iii)}}]
Let $L_n\to L$ and pull every problem back by $x=L_n\xi$ to $(-1,1)$. The resulting kernels, drift coefficients and zeroth-order coefficients converge in the topologies of Proposition~\ref{thm:stability}. The structural constants are uniform for $L_n$ in a compact subset of $(0,\infty)$. Proposition~\ref{thm:stability} therefore yields convergence of the eigenvalues and the normalized direct eigenfunctions. Applying the same argument to the reflected adjoint operators yields convergence of the adjoint eigenfunctions. No subsequence can have another limit because both direct and adjoint positive eigenfunctions are unique under the chosen normalizations.
\end{proof}

\begin{proof}[\textbf{Proof of Proposition~\ref{prop:sizecore}\textup{(iv)}}]
We establish differentiability first. Put $v_L(\xi)=\phi_L(L\xi)$ for $-1\le\xi\le1$. The direct eigenvalue equation is
\[
\mathscr A_Lv_L+\lambda_p(L)v_L=0\ \hbox{in }(-1,1),\quad v_L(1)=0,\quad v_L(0)=1,
\]
where
\[
\mathscr A_Lv(\xi)=dL\int_{-1}^{1}J(L(\xi-\eta))v(\eta)\,d\eta+\frac{q(L\xi)}{L}v'(\xi)+b(L\xi)v(\xi).
\]
Let
\[
X=\{v\in C^1([-1,1]):v(1)=0\},\quad Y=C([-1,1])\times\R,
\]
and define
\[
\mathscr F(L,\lambda,v)=\big(\mathscr A_Lv+\lambda v,v(0)-1\big).
\]
The map $\mathscr F:(0,\infty)\times\R\times X\to Y$ is $C^1$. Indeed, differentiation of the integral kernel is justified by $J\in C_c^2$, and the coefficients $q,b$ are $C^2$.

Fix $L$ and write $\lambda=\lambda_p(L)$, $v=v_L$. The derivative in $(\lambda,v)$ is
\[
D_{(\lambda,v)}\mathscr F(L,\lambda,v)[\mu,w]=\big((\mathscr A_L+\lambda)w+\mu v,w(0)\big).
\]
We prove that this map is an isomorphism. Injectivity follows from the adjoint. If the right-hand side is zero, pair the first component with the pulled-back adjoint eigenfunction. The term containing $(\mathscr A_L+\lambda)w$ vanishes, and positivity of the direct-adjoint pairing yields $\mu=0$. Then $w$ belongs to the kernel of $\mathscr A_L+\lambda$. Geometric simplicity yields $w=cv$, while $w(0)=0$ and $v(0)=1$ imply $c=0$.

For surjectivity, write the transport part with the terminal condition at $1$ as
\[
\mathscr T_Lw=\frac{q(L\xi)}{L}w'(\xi)+\big(b(L\xi)+\lambda\big)w(\xi),\quad w(1)=0.
\]
For every $f\in C([-1,1])$, the equation $\mathscr T_Lw=f$ has a unique solution $w\in X$, represented by the integrating-factor formula. Hence $\mathscr T_L:X\to C([-1,1])$ is an isomorphism. The nonlocal gain is compact from $X$ to $C([-1,1])$, so $\mathscr A_L+\lambda$ is Fredholm of index zero. Adding the scalar normalization does not change the index of the augmented map. Since the augmented derivative is injective, it is surjective. The implicit-function theorem yields $C^1$ dependence of $\lambda_p(L)$ and $v_L$ locally in $L$, and therefore globally on $(0,\infty)$.

We compute the derivative in the original variable. For $x$ in a compact subset of $(-L,L)$ put $\dot\phi(x)=\partial_L\phi_L(x)$ and $\dot\lambda=\partial_L\lambda_p(L)$. Differentiating the moving-domain integral at fixed $x$ yields
\[
\partial_L\int_{-L}^{L}J(x-y)\phi_L(y)\,dy=\int_{-L}^{L}J(x-y)\dot\phi(y)\,dy+J(x-L)\phi_L(L)+J(x+L)\phi_L(-L).
\]
Since $\phi_L(L)=0$, the upper-endpoint term vanishes. Differentiating the eigenvalue equation yields
\[
(\cA_L+\lambda_p(L))\dot\phi+\dot\lambda\phi_L+dJ(x+L)\phi_L(-L)=0.
\]
The moving outflow condition yields the exact identity $0=\partial_L\phi_L(L)=\dot\phi(L)+\phi_L'(L)$, so $\dot\phi(L)=-\phi_L'(L)$.

Multiply the differentiated equation by $\phi_L^*$ and integrate over $(-L,L)$. The Green formula now has a nonzero boundary contribution because $\dot\phi(L)$ need not vanish:
\[
\int_{-L}^{L}\phi_L^*(\cA_L+\lambda)\dot\phi\,dx=q(L)\phi_L^*(L)\dot\phi(L),
\]
the contribution at $-L$ being zero because $\phi_L^*(-L)=0$. Using the normalization $\int\phi_L\phi_L^*=1$, we obtain
\[
q(L)\phi_L^*(L)\dot\phi(L)+\dot\lambda+d\phi_L(-L)\int_{-L}^{L}J(x+L)\phi_L^*(x)\,dx=0.
\]
Substituting $\dot\phi(L)=-\phi_L'(L)$ yields the stated formula.

Every factor in the second term is strictly positive. The first term is strictly negative because $q(L)>0$, $\phi_L^*(L)>0$ and $\phi_L'(L)<0$. Hence $\partial_L\lambda_p(L)<0$.
\end{proof}

\begin{proposition}\label{prop:sizeanalytic}
Let $\mathcal L\subset(0,\infty)$ be open and assume $\mathrm{(A_{an})}$. Then $\lambda_p(L)=\Lp(\cA_L)$ is real analytic on $\mathcal L$. More precisely, if $v_L(\xi)=\phi_L(L\xi),\qquad -1\le\xi\le1$,
and $v_L(0)=1$, then $L\mapsto v_L$ is real analytic from $\mathcal L$ into $C^1([-1,1])$. After pull-back and a fixed nonzero normalization, the adjoint eigenfunction depends real analytically on $L$ as well.

At every $L_0\in\mathcal L$, writing $h=L-L_0$, one has convergent expansions
\[
\mathscr A_{L_0+h}=\sum_{n=0}^\infty h^n\mathscr A_n,\qquad
\lambda_p(L_0+h)=\sum_{n=0}^\infty h^n\lambda_n,\qquad
v_{L_0+h}=\sum_{n=0}^\infty h^nv_n
\]
in $\mathcal L(C^1,C)$, $\R$ and $C^1([-1,1])$, respectively. The coefficients are determined recursively by
\[
(\mathscr A_0+\lambda_0)v_n+\lambda_nv_0
=-\sum_{k=1}^n\mathscr A_kv_{n-k}
-\sum_{k=1}^{n-1}\lambda_kv_{n-k},
\qquad v_n(0)=0,
\]
for $n\ge1$. If $v_0^*$ is the adjoint eigenfunction normalized by $\int_{-1}^{1}v_0v_0^*=1$, then
\[
\lambda_n=-\int_{-1}^{1}v_0^*(\xi)
\left(
\sum_{k=1}^n\mathscr A_kv_{n-k}
+\sum_{k=1}^{n-1}\lambda_kv_{n-k}
\right)(\xi)\,d\xi.
\]
\end{proposition}

\begin{proof}[\textbf{Proof of Proposition~\ref{prop:sizeanalytic}}]
Fix $L_0\in\mathcal L$ and a compact interval $K\Subset\mathcal L$ containing $L_0$ in its interior. On the fixed interval $(-1,1)$ the pulled-back operator is
\[
\mathscr A_Lv(\xi)=dL\int_{-1}^{1}J(L(\xi-\eta))v(\eta)\,d\eta
+\frac{q(L\xi)}{L}v'(\xi)+b(L\xi)v(\xi).
\]
We first verify analyticity of $L\mapsto\mathscr A_L$ in operator norm. Let $D_K$ be the complex neighborhood in $\mathrm{(A_{an})}$ and let $z\in D_K$. Complexifying
\[
X=\{v\in C^1([-1,1];\mathbb C):v(1)=0\},\qquad
Z=C([-1,1];\mathbb C),
\]
define
\[
\mathscr A_zv(\xi)=dz\int_{-1}^{1}J(z(\xi-\eta))v(\eta)\,d\eta
+\frac{q(z\xi)}{z}v'(\xi)+b(z\xi)v(\xi).
\]
The three coefficients are holomorphic in $z$ uniformly in $\xi,\eta$. Since $0\notin D_K$, the factor $z^{-1}$ is holomorphic and uniformly bounded there. Thus $z\mapsto\mathscr A_z\in\mathcal L(X,Z)$ is holomorphic. We record the first two derivatives, because they exhibit all terms created by the moving interval. With $r=\xi-\eta$,
\[
\begin{aligned}
\partial_L\mathscr A_Lv(\xi)
={}&d\int_{-1}^{1}J(Lr)v(\eta)\,d\eta
+dL\int_{-1}^{1}rJ'(Lr)v(\eta)\,d\eta\\
&+\left(\frac{\xi q'(L\xi)}{L}-\frac{q(L\xi)}{L^2}\right)v'(\xi)
+\xi b'(L\xi)v(\xi),
\end{aligned}
\]
and
\[
\begin{aligned}
\partial_L^2\mathscr A_Lv(\xi)
={}&2d\int_{-1}^{1}rJ'(Lr)v(\eta)\,d\eta
+dL\int_{-1}^{1}r^2J''(Lr)v(\eta)\,d\eta\\
&+\left(\frac{\xi^2q''(L\xi)}{L}-\frac{2\xi q'(L\xi)}{L^2}
+\frac{2q(L\xi)}{L^3}\right)v'(\xi)
+\xi^2b''(L\xi)v(\xi).
\end{aligned}
\]
More generally, choose $0<r_K<\rho_K$ so that every closed disk $\overline{B(L,r_K)}$, $L\in K$, is contained in $D_K$. Cauchy's formula for the operator-valued holomorphic function yields
\[
\|\partial_L^n\mathscr A_L\|_{\mathcal L(X,Z)}
\le \frac{n!}{r_K^n}\sup_{|z-L|=r_K}\|\mathscr A_z\|_{\mathcal L(X,Z)}
\le C_K\frac{n!}{r_K^n},
\qquad L\in K.
\]
Hence the Taylor series of $\mathscr A_L$ converges in operator norm, locally uniformly in $L\in\mathcal L$.

We next treat the eigenpair. Define
\[
\mathscr F(z,\mu,v)=\big(\mathscr A_zv+\mu v,\ v(0)-1\big)
\]
from $D_K\times\mathbb C\times X$ to $Z\times\mathbb C$. The preceding estimates show that $\mathscr F$ is holomorphic. At the real eigenpair $(L_0,\lambda_0,v_0)$, where $\lambda_0=\lambda(L_0)$, its derivative with respect to $(\mu,v)$ is
\[
\mathscr T_0[\nu,w]
=\big((\mathscr A_{L_0}+\lambda_0)w+\nu v_0,\ w(0)\big).
\]
We show that $\mathscr T_0$ is an isomorphism. Suppose first that $\mathscr T_0[\nu,w]=0$. Pair the first component with the positive pulled-back adjoint eigenfunction $v_0^*$. The Green identity eliminates $(\mathscr A_{L_0}+\lambda_0)w$ and yields $\nu\int_{-1}^{1}v_0v_0^*\,d\xi=0$.
The integral is strictly positive, hence $\nu=0$. Geometric simplicity then yields $w=cv_0$. Since $w(0)=0$ and $v_0(0)=1$, we get $c=0$. Thus $\mathscr T_0$ is injective.

For surjectivity, split
\[
\mathscr A_{L_0}+\lambda_0=\mathscr Q_0+\mathscr K_0,
\]
where
\[
\mathscr Q_0w=\frac{q(L_0\xi)}{L_0}w'(\xi)
+\big(b(L_0\xi)+\lambda_0\big)w(\xi),\qquad w(1)=0,
\]
and
\[
\mathscr K_0w=dL_0\int_{-1}^{1}J(L_0(\xi-\eta))w(\eta)\,d\eta.
\]
Because $q\ge q_0>0$, $\mathscr Q_0:X\to Z$ is an isomorphism. Indeed, for $f\in Z$ the unique solution of $\mathscr Q_0w=f$, $w(1)=0$, is
\[
w(\xi)=-\int_\xi^1\frac{L_0f(t)}{q(L_0t)}
\exp\!\left(\int_\xi^t
\frac{L_0(b(L_0s)+\lambda_0)}{q(L_0s)}\,ds\right)dt.
\]
The integral operator $\mathscr K_0:X\to Z$ is compact. Therefore $\mathscr A_{L_0}+\lambda_0$ is Fredholm of index zero. To compute the index of the augmented map, compare $\mathscr T_0$ with the block operator
\[
(w,\nu)\longmapsto\big((\mathscr A_{L_0}+\lambda_0)w,\nu\big)
\]
from $X\times\mathbb C$ to $Z\times\mathbb C$. The two operators differ by a finite-rank map, namely the terms $\nu v_0$ in the first component and $w(0)-\nu$ in the second. Hence $\mathscr T_0$ is Fredholm of index zero. Since it is injective, it is surjective and therefore a bounded isomorphism. Its complexification is the same operator on the complex Banach spaces and remains invertible.

The analytic implicit-function theorem now yields $\varepsilon>0$ and unique holomorphic maps $z\mapsto\lambda(z)$ and $z\mapsto v(z)$ for $|z-L_0|<\varepsilon$ such that
\[
\mathscr A_zv(z)+\lambda(z)v(z)=0,\qquad v(z,1)=0,\qquad v(z,0)=1.
\]
For real $z=L$ close to $L_0$, the branch remains positive. Indeed, $v_0$ has a positive minimum on every $[-1,1-\delta]$, while $v_0'(1)<0$ by the endpoint Hopf identity; $C^1$ closeness preserves both properties for $L$ near $L_0$. The bounded-interval principal-eigenpair theorem therefore identifies $\lambda_p(L)$ with $\Lp(\cA_L)$. Thus the principal value is real analytic near $L_0$. Since $L_0$ was arbitrary, it is real analytic on $\mathcal L$. The adjoint statement follows by applying the same argument to the reflected adjoint operator, whose drift has fixed sign and whose normalization is nondegenerate.

It remains only to justify the displayed coefficient recursion. Write the convergent Taylor series without factorials,
\[
\mathscr A_{L_0+h}=\sum_{n\ge0}h^n\mathscr A_n,
\quad \lambda_p(L_0+h)=\sum_{n\ge0}h^n\lambda_n,
\quad v_{L_0+h}=\sum_{n\ge0}h^nv_n.
\]
Substituting in $(\mathscr A_{L_0+h}+\lambda_p(L_0+h))v_{L_0+h}=0$
and comparing the coefficient of $h^n$ yields, for $n\ge1$,
\[
(\mathscr A_0+\lambda_0)v_n+\lambda_nv_0
=-\sum_{k=1}^n\mathscr A_kv_{n-k}
-\sum_{k=1}^{n-1}\lambda_kv_{n-k}.
\]
The normalization $v_{L_0+h}(0)=1$ yields $v_n(0)=0$. Pairing with $v_0^*$ eliminates the first term and yields the formula for $\lambda_n$. Thus every Taylor coefficient is determined uniquely from lower-order coefficients, while the Cauchy estimate above supplies the factorial bounds which guarantee convergence. This completes the proof.
\end{proof}

\begin{proof}[\textbf{Proof of Theorem~\ref{thm:mainbounded}}]
Part \textup{(i)} follows from Proposition~\ref{thm:CHbounded}, Proposition~\ref{thm:comparison} and Proposition~\ref{prop:sizecore}\textup{(i)--(iv)}. The exhaustion assertion is the last part of Proposition~\ref{thm:CHbounded}. Part \textup{(ii)} is Proposition~\ref{prop:sizeanalytic}.
\end{proof}

\section{The critical zero-range limit}

Let $\Omega_L=(-L,L)$ and
\[
\cA_{\sigma,2,L}u(x)=\frac1{\sigma^2}\left(\int_{-L}^{L}J_\sigma(x-y)u(y)\,dy-u(x)\right)+q(x)u'(x)+a(x)u(x).
\]
Assume $J\in C_c^2(\R)$, $q,a\in C^2([-L,L])$ and $q\ge q_0>0$. Set $d_J=D_2(J)/2$ and denote by $\lambda_0$ the Dirichlet principal eigenvalue of $d_Ju''+q(x)u'+a(x)u$
on $(-L,L)$.

\begin{proposition}\label{prop:criticaltools}
The two estimates on which the critical compactness argument rests are the following.

\textup{(i)} Let $\phi_0>0$ be the normalized local Dirichlet principal eigenfunction. For every $\varepsilon>0$ there exists $\sigma_\varepsilon>0$ such that
\[
\cA_{\sigma,2,L}\phi_0+(\lambda_0+\varepsilon)\phi_0\ge0\;\hbox{in }(-L,L)
\]
for $0<\sigma<\sigma_\varepsilon$. Consequently $\Lp(\cA_{\sigma,2,L})\le\lambda_0+\varepsilon$.

\textup{(ii)} There exist constants $c_J,C_J>0$ such that
\[
c_J\min\{\eta^2,1\}\le1-\widehat J(\eta)\le C_J\min\{\eta^2,1\}\;\hbox{for all }\eta\in\R.
\]
Consequently, if $\sigma_n\to0$ and $u_n$ are supported in a fixed bounded interval with $\|u_n\|_2\le C$ and $E_{\sigma_n}(u_n)\le C$, then a subsequence converges strongly in $L^2(\R)$ to a function in $H^1(\R)$ supported in the same interval. The estimate is proved directly below from the Fourier symbol. In particular, the proof does not invoke the Sobolev-norm approximation/characterization theorems of \cite{BBM,Brezis02,PonceCV,Ponce}; retaining the zero-extension support is essential for recovering the missing inflow boundary condition in the advective problem.
For every such strongly convergent subsequence, with limit $u$, one also has the sharp lower-semicontinuity estimate
\[
d_J\int_\R |u'(x)|^2\,dx\le \liminf_{n\to\infty}E_{\sigma_n}(u_n).
\]
\end{proposition}

\begin{proof}[\textbf{Proof of Proposition~\ref{prop:criticaltools}\textup{(i)}}]
Extend $\phi_0$ to a $C^4$ function $\Phi$ across the endpoints so that $\Phi<0$ immediately outside $[-L,L]$. This is possible because $\phi_0'(-L)>0$ and $\phi_0'(L)<0$. Let $\widetilde\phi_0$ be the zero extension. Then $\widetilde\phi_0\ge\Phi$ in a fixed neighborhood of the interval. Since $\Phi$ is smooth across both endpoints, symmetry of $J$ and Taylor's formula yield, uniformly for $x\in[-L,L]$,
\[
\frac1{\sigma^2}\left(\int_\R J(z)\Phi(x-\sigma z)\,dz-\Phi(x)\right)=d_J\phi_0''(x)+O(\sigma^2).
\]
Hence
\[
\cA_{\sigma,2,L}\phi_0+(\lambda_0+\varepsilon)\phi_0\ge\varepsilon\phi_0-C\sigma^2.
\]
This proves the desired inequality whenever $\phi_0\ge2C\sigma^2/\varepsilon$.

It remains to treat the $O(\sigma)$ boundary layers. Near $L$, write $x=L-\sigma t$, $0\le t\le R_J$, and put $\beta=-\phi_0'(L)>0$. Uniformly for $|z|\le R_J$,
\[
\widetilde\phi_0(L-\sigma t-\sigma z)=\begin{cases}\beta\sigma(t+z)+O(\sigma^2),&t+z\ge0,\\0,&t+z<0.\end{cases}
\]
Therefore
\[
\int J(z)\widetilde\phi_0(L-\sigma t-\sigma z)\,dz-\phi_0(L-\sigma t)=\beta\sigma\int_{z\le-t}(-t-z)J(z)\,dz+O(\sigma^2).
\]
Since $J(0)>0$, the integral on the right is bounded below by a positive constant for $0\le t\le\eta$ with some $\eta>0$. Thus the nonlocal term is positive of order $\sigma^{-1}$ in the innermost right collar. The computation at $-L$ is the reflected one and yields the same sign. In the remaining part of the boundary neighborhood, the Hopf lemma yields $\phi_0(x)\ge c\dist(x,\partial\Omega_L)$, so $\varepsilon\phi_0-C\sigma^2\ge0$ for small $\sigma$. This proves the supersolution inequality everywhere.

Proposition~\ref{thm:comparison}, applied to the positive supersolution $\phi_0$, yields the eigenvalue bound.
\end{proof}

\begin{proof}[\textbf{Proof of Proposition~\ref{prop:criticaltools}\textup{(ii)}}]
Since $J$ is even,
\[
1-\widehat J(\eta)=\int_\R J(z)(1-\cos(\eta z))\,dz.
\]
For $|\eta|\le1/R_J$, the elementary inequalities $c t^2\le1-\cos t\le Ct^2$ for $|t|\le1$ yield
\[
c\eta^2\int z^2J(z)\,dz\le1-\widehat J(\eta)\le C\eta^2\int z^2J(z)\,dz.
\]
For $|\eta|\ge1/R_J$, continuity and the fact that $J$ is positive on an interval imply $|\widehat J(\eta)|<1$ for every $\eta\ne0$: equality in
$|\int J(z)e^{-i\eta z}dz|\le\int J=1$ would force $e^{-i\eta z}$ to have a constant phase on an interval on which $J>0$, which is impossible unless $\eta=0$. On each compact annulus the minimum of $1-\widehat J$ is therefore positive, while the Riemann--Lebesgue lemma yields $\widehat J(\eta)\to0$ as $|\eta|\to\infty$. This proves the global two-sided bound.

By Plancherel,
\[
E_\sigma(u)=\int_\R\frac{1-\widehat J(\sigma\xi)}{\sigma^2}|\widehat u(\xi)|^2\,d\xi.
\]
Choose $\kappa>0$ so that $1-\widehat J(\eta)\ge c\eta^2$ for $|\eta|\le\kappa$ and $1-\widehat J(\eta)\ge c$ for $|\eta|\ge\kappa$. For $R>0$ and $n$ so large that $R<\kappa/\sigma_n$, split the Fourier tail into two regions. The energy bound gives
\[
\begin{aligned}
\int_{R<|\xi|\le\kappa/\sigma_n}|\widehat u_n(\xi)|^2\,d\xi
&\le \frac1{R^2}\int_{|\xi|\le\kappa/\sigma_n}\xi^2|\widehat u_n(\xi)|^2\,d\xi
\le \frac{C}{R^2},\\
\int_{|\xi|>\kappa/\sigma_n}|\widehat u_n(\xi)|^2\,d\xi
&\le C\sigma_n^2E_{\sigma_n}(u_n)\le C\sigma_n^2.
\end{aligned}
\]
Thus
\[
\limsup_{n\to\infty}\int_{|\xi|>R}|\widehat u_n(\xi)|^2\,d\xi\le \frac{C}{R^2}.
\]
Translation equicontinuity now follows quantitatively from
\[
\|u_n(\cdot+h)-u_n\|_2^2=\int|e^{ih\xi}-1|^2|\widehat u_n(\xi)|^2\,d\xi
\]
by splitting at $|\xi|=R$: the low-frequency part is at most $h^2R^2\sup_n\|u_n\|_2^2$, and the high-frequency part is at most four times the preceding tail. Kolmogorov--Riesz yields strong $L^2$ compactness.

Let $u_n\to u$ strongly in $L^2$ along the extracted subsequence. Then $\widehat u_n\to\widehat u$ strongly in $L^2$ and, after one further extraction, almost everywhere. Since
\[
\frac{1-\widehat J(\sigma_n\xi)}{\sigma_n^2}\longrightarrow d_J\xi^2
\qquad\hbox{for every }\xi\in\R,
\]
Fatou's lemma gives
\[
d_J\int_\R\xi^2|\widehat u(\xi)|^2\,d\xi
\le\liminf_{n\to\infty}E_{\sigma_n}(u_n).
\]
Hence $u\in H^1(\R)$ and the sharp lower-semicontinuity estimate holds. Support is preserved by strong convergence because all $u_n$ vanish outside the same fixed interval. If that interval is $[-L,L]$, the one-dimensional $H^1$ representative is continuous; since it vanishes almost everywhere on each exterior half-line, its traces at $-L$ and $L$ are zero. Therefore its restriction belongs to $H_0^1(-L,L)$.
\end{proof}

\begin{proof}[\textbf{Proof of Theorem~\ref{thm:maincritical}}]
By Proposition~\ref{prop:criticaltools}\textup{(i)}, $\limsup_{\sigma\to0}\lambda_\sigma\le\lambda_0$.
Extend $\phi_\sigma$ by zero and use symmetry of $J$ to obtain
\[
-\int_{-L}^{L}\phi_\sigma\frac1{\sigma^2}\left(\int_{-L}^{L}J_\sigma(x-y)\phi_\sigma(y)\,dy-\phi_\sigma(x)\right)dx=E_\sigma(\phi_\sigma).
\]
Multiplying the eigenvalue equation by $\phi_\sigma$ and integrating yields
\[
E_\sigma(\phi_\sigma)+\frac{q(-L)}2\phi_\sigma(-L)^2=\int_{-L}^{L}\left(a+\lambda_\sigma-\frac{q'}2\right)\phi_\sigma^2\,dx,
\]
because $\phi_\sigma(L)=0$. The left-hand side is nonnegative, hence
\[
\lambda_\sigma\ge-\left\|a-\frac{q'}2\right\|_\infty.
\]
The eigenvalues and the energies are therefore uniformly bounded.

The uniform energy and $L^2$ bounds allow us to apply Proposition~\ref{prop:criticaltools}\textup{(ii)}. Along a subsequence, $\widetilde\phi_\sigma\to\widetilde\phi$ strongly in $L^2(\R)$, with $\widetilde\phi\in H^1(\R)$ and support contained in $[-L,L]$. Hence $\phi\in H_0^1(-L,L)$.

After extracting once more, $\lambda_\sigma\to\lambda$. Let $\zeta\in C_c^\infty(-L,L)$. Symmetry yields
\[
\int\zeta\,\cM_{\sigma,2}\phi_\sigma=\int\phi_\sigma\,\cM_{\sigma,2}\zeta,
\]
and Taylor's formula yields $\cM_{\sigma,2}\zeta\to d_J\zeta''$ uniformly. For the drift term no derivative compactness is needed: because $\zeta$ is compactly supported,
\[
\int_{-L}^{L}\zeta q\phi_\sigma'\,dx
=-\int_{-L}^{L}(q\zeta)'\phi_\sigma\,dx
\longrightarrow-\int_{-L}^{L}(q\zeta)'\phi\,dx.
\]
The zeroth-order term converges by the same strong $L^2$ convergence. We obtain $d_J\phi''+q\phi'+a\phi+\lambda\phi=0$
in distributions, with $\phi\in H_0^1(-L,L)$, $\phi\ge0$ and $\|\phi\|_2=1$. The local strong maximum principle yields $\phi>0$, hence $\lambda=\lambda_0$ and $\phi=\phi_0$. Uniqueness of the limit removes the subsequence.

Finally, return to the exact energy identity. The strong $L^2$ convergence and $\lambda_\sigma\to\lambda_0$ yield
\[
E_\sigma(\phi_\sigma)+\frac{q(-L)}2\phi_\sigma(-L)^2\longrightarrow\int_{-L}^{L}\left(a+\lambda_0-\frac{q'}2\right)\phi_0^2\,dx.
\]
Multiplying the local equation by $\phi_0$ and integrating by parts, using $\phi_0(-L)=\phi_0(L)=0$, yields
\[
d_J\int_{-L}^{L}|\phi_0'|^2\,dx=\int_{-L}^{L}\left(a+\lambda_0-\frac{q'}2\right)\phi_0^2\,dx.
\]
On the other hand, Proposition~\ref{prop:criticaltools}\textup{(ii)} and Fatou's lemma applied to the Fourier representation of the energies yield
\[
d_J\int_{-L}^{L}|\phi_0'|^2\,dx\le\liminf_{\sigma\to0}E_\sigma(\phi_\sigma).
\]
The boundary term $q(-L)\phi_\sigma(-L)^2/2$ is nonnegative. Hence the convergence of the sum to the local Dirichlet energy forces simultaneously
\[
E_\sigma(\phi_\sigma)\longrightarrow d_J\int_{-L}^{L}|\phi_0'|^2\,dx,\qquad \phi_\sigma(-L)\longrightarrow0.
\]
The energy identity therefore yields the Dirichlet condition at the inflow endpoint.
\end{proof}

\begin{lemma}\label{lem:criticalstability}
Let $\delta_n\downarrow0$ and let $q_n,a_n\in C^2([-R,R])$ satisfy
$q_n\to q$ and $a_n\to a$ in $C^2([-R,R])$, with $\inf_n\min q_n>0$. Then
\[
\Lp\!\left(\delta_n^{-2}(J_{\delta_n}*_{(-R,R)}-I)+q_n\partial_x+a_n\right)
\longrightarrow
\lambda_1^D(d_J\partial_{xx}+q\partial_x+a,(-R,R)).
\]
The corresponding $L^2$-normalized eigenfunctions converge strongly in $L^2(-R,R)$.
\end{lemma}

\begin{proof}[\textbf{Proof of Lemma~\ref{lem:criticalstability}}]
Write
\[
\lambda_n=\Lp\!\left(\delta_n^{-2}(J_{\delta_n}*_{(-R,R)}-I)
+q_n\partial_x+a_n\right)
\]
and let $\nu_n$ be the Dirichlet principal value of
$d_J\partial_{xx}+q_n\partial_x+a_n$ on $(-R,R)$. Let $\psi_n$ be its positive $L^2$-normalized Dirichlet eigenfunction. The gauge
\[
u(x)=\exp\!\left(-\int_0^x\frac{q_n(t)}{2d_J}\,dt\right)v(x)
\]
reduces the latter problem to the self-adjoint Dirichlet operator
\[
d_Jv''+\left(a_n-\frac{q_n'}2-\frac{q_n^2}{4d_J}\right)v.
\]
The transformed potentials converge in $C^1([-R,R])$. The Rayleigh formula for this auxiliary local problem, followed by the one-dimensional equation, therefore gives
\[
\nu_n\to\nu,\qquad \psi_n\to\psi
\quad\hbox{in }C^2([-R,R]),
\]
where $\nu$ and $\psi$ are the limiting local principal value and its positive $L^2$-normalized eigenfunction. The original local equation may be written
$\psi_n''=-d_J^{-1}q_n\psi_n'-d_J^{-1}(a_n+\nu_n)\psi_n$. The preceding $C^2$ convergence and the uniform $C^2$ coefficient bounds first control $\psi_n''$; differentiating this identity once and twice then controls the third and fourth derivatives uniformly. Thus $\{\psi_n\}$ is uniformly bounded in $C^4([-R,R])$. The endpoint Hopf derivatives are bounded away from zero because they converge to those of $\psi$.

Fix $\varepsilon>0$. Extend each $\psi_n$ through the endpoints by a $C^4$ function which is negative just outside $[-R,R]$ and whose $C^4$ norm is bounded independently of $n$. The zero extension dominates this smooth extension. Taylor's formula and the endpoint calculation in Proposition~\ref{prop:criticaltools}\textup{(i)}, now with uniform constants, give
\[
\delta_n^{-2}(J_{\delta_n}*_{(-R,R)}\psi_n-\psi_n)
+q_n\psi_n'+a_n\psi_n+(\nu_n+\varepsilon)\psi_n\ge0
\]
for all sufficiently large $n$. In the interior the residual is
$\varepsilon\psi_n+O(\delta_n^2)$; in the innermost boundary collars the zero-extension correction is bounded below by $c/\delta_n$, and in the intermediate collars the uniform Hopf estimate $\psi_n\ge c\,\dist(\cdot,\partial(-R,R))$ absorbs the $O(\delta_n^2)$ remainder. Proposition~\ref{thm:comparison} consequently yields
\[
\limsup_{n\to\infty}\lambda_n\le\nu.
\]

Let $\phi_n$ be the nonlocal principal eigenfunction normalized by $\|\phi_n\|_2=1$ and extended by zero. Symmetry and integration by parts give the exact identity
\[
E_{\delta_n}(\phi_n)+\frac{q_n(-R)}2\phi_n(-R)^2
=\int_{-R}^{R}\left(a_n+\lambda_n-\frac{q_n'}2\right)\phi_n^2\,dx.
\]
Its nonnegative left-hand side gives a uniform lower bound for $\lambda_n$; the upper bound just proved then gives a uniform energy bound. Proposition~\ref{prop:criticaltools}\textup{(ii)} supplies, along a subsequence,
\[
\phi_n\to\phi\quad\hbox{strongly in }L^2(-R,R),
\qquad \phi\in H_0^1(-R,R),\quad \|\phi\|_2=1.
\]
After also taking $\lambda_n\to\lambda$, test the eigenvalue equation against $\zeta\in C_c^\infty(-R,R)$. Symmetry transfers the nonlocal operator to $\zeta$, and
\[
\delta_n^{-2}(J_{\delta_n}*\zeta-\zeta)\to d_J\zeta''
\quad\hbox{uniformly}.
\]
The $C^2$ convergence of $q_n,a_n$ and the strong $L^2$ convergence of $\phi_n$ therefore yield
\[
d_J\phi''+q\phi'+a\phi+\lambda\phi=0
\quad\hbox{in }\mathcal D'(-R,R).
\]
The limit is nonnegative and nonzero. The strong maximum principle and simplicity of the local Dirichlet principal eigenfunction imply $\lambda=\nu$ and $\phi=\psi$. Thus every subsequence has the same spectral and eigenfunction limit, which proves convergence of the full sequence.
\end{proof}

\begin{proposition}\label{prop:crossovers}
The same compactness mechanism yields two useful crossover limits.

\textup{(i)} Let $m_\sigma$ be a real-valued function such that
\[
\delta_\sigma:=\sigma^{2-m_\sigma}\longrightarrow\delta\in(0,\infty).
\]
Consider
\[
\cA_{\sigma,m_\sigma,L}u=\frac1{\sigma^{m_\sigma}}\left(\int_{-L}^{L}J_\sigma(x-y)u(y)\,dy-u(x)\right)+q(x)u'(x)+a(x)u(x).
\]
Then
\[
\Lp(\cA_{\sigma,m_\sigma,L})\longrightarrow\lambda_\delta,
\]
where $\lambda_\delta$ is the Dirichlet principal eigenvalue of $\delta d_Ju''+q(x)u'+a(x)u$
on $(-L,L)$. If the principal eigenfunctions are normalized in $L^2$, they converge strongly in $L^2$ to the normalized local principal eigenfunction. Moreover,
\[
\delta_\sigma E_\sigma(\phi_\sigma)\longrightarrow\delta d_J\int_{-L}^{L}|\phi_\delta'|^2\,dx.
\]

\textup{(ii)} Assume $q,a\in C^2([-L,L])$, $q\ge q_0>0$, let $m_\sigma\to2$, and suppose
\[
(2-m_\sigma)|\log\sigma|\longrightarrow\theta\in\R.
\]
Then $\delta_\sigma=\sigma^{2-m_\sigma}\to e^{-\theta}$ and
\[
\Lp(\cA_{\sigma,m_\sigma,L})\longrightarrow
\lambda_1^D\left(e^{-\theta}d_J\partial_{xx}+q(x)\partial_x+a(x),(-L,L)\right).
\]
\end{proposition}

\begin{corollary}\label{cor:crossover-homogeneous}
If $q\equiv c>0$ and $a\equiv a_0$, then
\[
\Lp(\cA_{\sigma,m_\sigma,L})\longrightarrow
e^{-\theta}d_J\frac{\pi^2}{4L^2}
+e^{\theta}\frac{c^2}{4d_J}-a_0
=
e^{-\theta}\frac{D_2(J)\pi^2}{8L^2}
+e^{\theta}\frac{c^2}{2D_2(J)}-a_0.
\]
\end{corollary}

\begin{proof}
Put $D_\theta=e^{-\theta}d_J$. The gauge transform
$u=e^{-cx/(2D_\theta)}v$ reduces the local Dirichlet operator to
\[
D_\theta v''+\left(a_0-\frac{c^2}{4D_\theta}\right)v.
\]
Its principal value on $(-L,L)$ is
$D_\theta\pi^2/(4L^2)+c^2/(4D_\theta)-a_0$.
\end{proof}

\begin{proof}[\textbf{Proof of Proposition~\ref{prop:crossovers}\textup{(i)}}]
Write the nonlocal part in the form
\[
\frac1{\sigma^{m_\sigma}}(J_\sigma*_{\Omega_L}u-u)=\delta_\sigma\frac1{\sigma^2}(J_\sigma*_{\Omega_L}u-u).
\]
Let $\phi_\delta$ be the positive Dirichlet eigenfunction of the local operator with diffusion coefficient $\delta d_J$. We construct a uniform supersolution. The smooth-extension argument in Proposition~\ref{prop:criticaltools}\textup{(i)} yields
\[
\frac1{\sigma^2}\left(\int J(z)\widetilde\phi_\delta(x-\sigma z)\,dz-\phi_\delta(x)\right)\ge d_J\phi_\delta''(x)-C\sigma^2
\]
away from the innermost boundary collars. Multiplication by $\delta_\sigma$ and the convergence $\delta_\sigma\to\delta$ yield
\[
\cA_{\sigma,m_\sigma,L}\phi_\delta+(\lambda_\delta+\varepsilon)\phi_\delta\ge\varepsilon\phi_\delta-o(1)
\]
in the interior. Near either boundary, the zero-extension correction in Proposition~\ref{prop:criticaltools}\textup{(i)} is positive of order $\delta_\sigma/\sigma$. Since $\delta_\sigma$ stays bounded above and below by positive constants, this term still dominates all bounded drift and zeroth-order terms in the innermost $O(\sigma)$ collars. The intermediate collars are treated by the Hopf estimate $\phi_\delta\ge c\dist(x,\partial\Omega_L)$. Hence, for small $\sigma$,
\[
\cA_{\sigma,m_\sigma,L}\phi_\delta+(\lambda_\delta+\varepsilon)\phi_\delta\ge0
\]
throughout the interval. The bounded-interval comparison theorem yields
\[
\limsup_{\sigma\to0}\Lp(\cA_{\sigma,m_\sigma,L})\le\lambda_\delta.
\]

Let $\lambda_\sigma$ and $\phi_\sigma$ denote the nonlocal principal eigenpair, with $\|\phi_\sigma\|_2=1$. Multiplication by $\phi_\sigma$ and integration yield
\[
\delta_\sigma E_\sigma(\phi_\sigma)+\frac{q(-L)}2\phi_\sigma(-L)^2=\int_{-L}^{L}\left(a+\lambda_\sigma-\frac{q'}2\right)\phi_\sigma^2\,dx.
\]
As in the critical proof, the left-hand side is nonnegative and the upper bound for $\lambda_\sigma$ provides a uniform lower bound for the eigenvalues. Since $\delta_\sigma\ge\delta/2$ for small $\sigma$, the nonlocal energies $E_\sigma(\phi_\sigma)$ are uniformly bounded. Proposition~\ref{prop:criticaltools}\textup{(ii)} therefore yields strong $L^2$ compactness of the zero extensions and an $H_0^1$ limit.

Along a subsequence let $\lambda_\sigma\to\lambda$ and $\phi_\sigma\to\phi$ strongly in $L^2$. For $\zeta\in C_c^\infty(-L,L)$,
\[
\delta_\sigma\int\zeta\cM_{\sigma,2}\phi_\sigma=\delta_\sigma\int\phi_\sigma\cM_{\sigma,2}\zeta\longrightarrow\delta d_J\int\phi\zeta''.
\]
The drift and zeroth-order terms pass to the limit as before. Hence $\delta d_J\phi''+q\phi'+a\phi+\lambda\phi=0$,
with $\phi\in H_0^1(-L,L)$, $\phi\ge0$ and $\|\phi\|_2=1$. The local strong maximum principle and simplicity identify $\lambda=\lambda_\delta$ and $\phi=\phi_\delta$. This also removes the subsequence.

Finally, the energy identity converges to
\[
\delta d_J\int|\phi_\delta'|^2=\int\left(a+\lambda_\delta-\frac{q'}2\right)\phi_\delta^2.
\]
Lower semicontinuity and nonnegativity of the boundary term yield
\[
\delta_\sigma E_\sigma(\phi_\sigma)
\longrightarrow\delta d_J\int_{-L}^{L}|\phi_\delta'|^2\,dx,
\]
as in the proof of Theorem~\ref{thm:maincritical}.
\end{proof}

\begin{proof}[\textbf{Proof of Proposition~\ref{prop:crossovers}\textup{(ii)}}]
The identity
\[
\delta_\sigma=\exp\big(-(2-m_\sigma)|\log\sigma|\big)
\]
yields $\delta_\sigma\to e^{-\theta}$. Proposition~\ref{prop:crossovers}\textup{(i)} applies directly and yields the heterogeneous local Dirichlet principal value displayed above. No constancy of the drift or the potential is used.
\end{proof}

\begin{proposition}\label{prop:supercritical}
Let $m>2$ be fixed.  Assume $J\in C_c^2(\R)$ is even, nonnegative, normalized and
positive on a neighborhood of the origin, and assume
\[
q\in C^2([-L,L]),\qquad a\in C^2([-L,L]),\qquad q\ge q_0>0.
\]
Set
\[
\varepsilon_\sigma:=\sigma^{m-2},
\qquad
\mu_\sigma:=\varepsilon_\sigma\Lp(\cA_{\sigma,m,L}),
\qquad
\mu_D:=d_J\frac{\pi^2}{4L^2}.
\]
Then $\mu_\sigma\longrightarrow\mu_D$.
If $\phi_\sigma>0$ is normalized by $\|\phi_\sigma\|_{L^2(-L,L)}=1$, then, after zero
extension,
\[
\widetilde\phi_\sigma\longrightarrow
e_1:=L^{-1/2}\cos\frac{\pi x}{2L}\,\mathbf 1_{[-L,L]}
\quad\hbox{strongly in }L^2(\R),
\]
and $E_\sigma(\phi_\sigma)\longrightarrow\mu_D$.
In addition,
\[
\varepsilon_\sigma\phi_\sigma(-L)^2\longrightarrow0.
\]
Consequently,
\[
\Lp(\cA_{\sigma,m,L})
=
\frac{D_2(J)\pi^2}{8L^2}\,\sigma^{2-m}
+o(\sigma^{2-m}).
\]
\end{proposition}

\begin{proof}[\textbf{Proof of Proposition~\ref{prop:supercritical}}]
Write
\[
\lambda_\sigma=\Lp(\cA_{\sigma,m,L}),\qquad
\varepsilon_\sigma=\sigma^{m-2},
\qquad
\mu_\sigma=\varepsilon_\sigma\lambda_\sigma.
\]
Multiplication of the eigenvalue equation by $\varepsilon_\sigma$ yields
\begin{equation}
\cM_{\sigma,2}\phi_\sigma
+\varepsilon_\sigma q(x)\phi_\sigma'
+\varepsilon_\sigma a(x)\phi_\sigma
+\mu_\sigma\phi_\sigma=0.
\label{eq:6_20}
\end{equation}
Since $m>2$, $\varepsilon_\sigma\to0$.  The directional boundary condition nevertheless
persists for every positive $\sigma$, and therefore compactness has to be obtained before the
transport term is discarded.

\smallskip
\noindent\emph{Step 1: a uniform upper bound.}
For $\varepsilon\in[0,1]$, let $\nu_\varepsilon$ denote the Dirichlet principal eigenvalue of
\begin{equation}
d_Ju''+\varepsilon q(x)u'+\varepsilon a(x)u
\quad\hbox{in }(-L,L),\qquad u(-L)=u(L)=0,
\label{eq:6_21}
\end{equation}
and normalize its positive eigenfunction $u_\varepsilon$ by $\|u_\varepsilon\|_2=1$.
Simplicity of the first Dirichlet eigenvalue and classical one-dimensional elliptic
regularity imply
\begin{equation}
\nu_\varepsilon\longrightarrow\mu_D,\qquad
u_\varepsilon\longrightarrow e_1
\quad\hbox{in }C^2([-L,L]),
\label{eq:6_22}
\end{equation}
as $\varepsilon\downarrow0$.  The equation in \eqref{eq:6_21}, together with $q,a\in C^2$, yields
a uniform $C^4$ bound.  The Hopf lemma and \eqref{eq:6_22} also yield constants $c_0,C_0>0$ such that
\begin{equation}
u_\varepsilon'(-L)\ge c_0,\qquad
-u_\varepsilon'(L)\ge c_0,\qquad
\|u_\varepsilon\|_{C^4([-L,L])}\le C_0
\label{eq:6_23}
\end{equation}
for all sufficiently small $\varepsilon$.

Fix $\eta>0$ and set $\varepsilon=\varepsilon_\sigma$.  We claim that
\begin{equation}
\cM_{\sigma,2}u_{\varepsilon_\sigma}
+\varepsilon_\sigma q u_{\varepsilon_\sigma}'
+\varepsilon_\sigma a u_{\varepsilon_\sigma}
+(\nu_{\varepsilon_\sigma}+\eta)u_{\varepsilon_\sigma}\ge0
\label{eq:6_24}
\end{equation}
in $(-L,L)$ for all sufficiently small $\sigma$.
The verification is uniform in the vanishing parameter $\varepsilon_\sigma$.

Extend $u_\varepsilon$ to a $C^4$ function $U_\varepsilon$ on a fixed neighborhood of
$[-L,L]$ in such a way that $U_\varepsilon<0$ immediately outside the interval.
The bounds in \eqref{eq:6_23} allow the extensions to be chosen with a common $C^4$ bound.
If $\widetilde u_\varepsilon$ denotes the zero extension, then
$\widetilde u_\varepsilon\ge U_\varepsilon$ on that neighborhood.  Therefore, for every
$x\in[-L,L]$,
\[
\begin{aligned}
\cM_{\sigma,2}u_\varepsilon(x)
&=\frac1{\sigma^2}\left(
\int_\R J(z)\widetilde u_\varepsilon(x-\sigma z)\,dz-u_\varepsilon(x)\right)\\
&\ge
\frac1{\sigma^2}\left(
\int_\R J(z)U_\varepsilon(x-\sigma z)\,dz-U_\varepsilon(x)\right)\\
&=d_Ju_\varepsilon''(x)+O(\sigma^2),
\end{aligned}
\]
where the remainder is uniform in $x$ and $0\le\varepsilon\le\varepsilon_0$.
Using \eqref{eq:6_21}, the left-hand side of \eqref{eq:6_24} is consequently bounded below by
\begin{equation}
\eta u_{\varepsilon_\sigma}(x)-C\sigma^2.
\label{eq:6_26}
\end{equation}
By the uniform Hopf bounds in \eqref{eq:6_23}, there are $r_0,c_1>0$ such that
\[
u_\varepsilon(x)\ge c_1\,\dist(x,\partial\Omega_L)
\qquad
\hbox{if }\dist(x,\partial\Omega_L)\le r_0
\]
for all sufficiently small $\varepsilon$.  Thus \eqref{eq:6_26} proves \eqref{eq:6_24} whenever
\begin{equation}
\dist(x,\partial\Omega_L)\ge C_\eta\sigma^2
\label{eq:6_28}
\end{equation}
after increasing $C_\eta$.

It remains only to consider the two collars of physical width $O(\sigma^2)$.  At the right
endpoint write $x=L-\sigma t$; then \eqref{eq:6_28} yields $0\le t\le C_\eta\sigma$.  With
$\beta_{\varepsilon,+}=-u_\varepsilon'(L)\ge c_0$, Taylor expansion at $L$ yields
uniformly in this collar
\begin{equation}
\begin{aligned}
&\int_\R J(z)\widetilde u_\varepsilon(L-\sigma t-\sigma z)\,dz
-u_\varepsilon(L-\sigma t)\\
&\qquad=
\beta_{\varepsilon,+}\sigma
\int_{z\le-t}(-t-z)J(z)\,dz+O(\sigma^2).
\end{aligned}
\label{eq:6_29}
\end{equation}
Since $t=O(\sigma)$ and $J$ is positive on a neighborhood of the origin, there is
$c_2>0$, independent of $\sigma$ and $\varepsilon$, such that $\int_{z\le-t}(-t-z)J(z)\,dz\ge c_2$
for all sufficiently small $\sigma$.  Hence the truncated nonlocal correction in \eqref{eq:6_29},
after division by $\sigma^2$, is bounded below by $c_0c_2/(2\sigma)$.  It dominates all
remaining terms in \eqref{eq:6_24}.  The reflected computation applies at $-L$.  This proves
\eqref{eq:6_24} on the whole interval.

The comparison theorem yields $\mu_\sigma\le\nu_{\varepsilon_\sigma}+\eta$.
Using \eqref{eq:6_22} and then letting $\eta\downarrow0$ yields
\begin{equation}
\limsup_{\sigma\to0}\mu_\sigma\le\mu_D.
\label{eq:6_30}
\end{equation}

\smallskip
\noindent\emph{Step 2: the normalized energy identity.}
Normalize $\|\phi_\sigma\|_2=1$ and extend $\phi_\sigma$ by zero.  Multiplying \eqref{eq:6_20} by
$\phi_\sigma$ and integrating yields
\begin{equation}
E_\sigma(\phi_\sigma)
+\frac{\varepsilon_\sigma q(-L)}2\phi_\sigma(-L)^2
=
\mu_\sigma
+\varepsilon_\sigma\int_{-L}^{L}
\left(a-\frac{q'}2\right)\phi_\sigma^2\,dx.
\label{eq:6_31}
\end{equation}
Indeed, symmetry yields
$-\int\phi_\sigma\cM_{\sigma,2}\phi_\sigma=E_\sigma(\phi_\sigma)$, while
\[
\int q\phi_\sigma'\phi_\sigma
=
-\frac{q(-L)}2\phi_\sigma(-L)^2
-\frac12\int q'\phi_\sigma^2
\]
because $\phi_\sigma(L)=0$.
The boundary term in \eqref{eq:6_31} is nonnegative.  Hence \eqref{eq:6_30} and
$\varepsilon_\sigma\to0$ imply $E_\sigma(\phi_\sigma)\le C$.

\smallskip
\noindent\emph{Step 3: compactness and the sharp lower bound.}
Proposition~\ref{prop:criticaltools}\textup{(ii)} applies to the zero extensions.  Along a subsequence,
\[
\widetilde\phi_\sigma\to\widetilde\phi
\quad\hbox{strongly in }L^2(\R),
\]
where $\widetilde\phi\in H^1(\R)$ is supported in $[-L,L]$.  Therefore
$\phi\in H_0^1(-L,L)$ and $\|\phi\|_2=1$.
From \eqref{eq:6_31},
\[
\mu_\sigma
\ge
E_\sigma(\phi_\sigma)
-\varepsilon_\sigma\left\|a-\frac{q'}2\right\|_\infty.
\]
The Fourier representation of the energy and Fatou's lemma yield
\begin{equation}
\liminf_{\sigma\to0}\mu_\sigma
\ge d_J\int_{-L}^{L}|\phi'|^2\,dx.
\label{eq:6_33}
\end{equation}
The sharp Poincar\'e inequality yields
\begin{equation}
d_J\int_{-L}^{L}|\phi'|^2\,dx
\ge d_J\frac{\pi^2}{4L^2}\int_{-L}^{L}\phi^2\,dx
=\mu_D.
\label{eq:6_34}
\end{equation}
Combining \eqref{eq:6_30}, \eqref{eq:6_33} and \eqref{eq:6_34},
\begin{equation}
\mu_\sigma\longrightarrow\mu_D.
\label{eq:6_35}
\end{equation}
Equality holds in the Poincar\'e inequality for every subsequential limit.  The equality
case is one-dimensional, so positivity and normalization yield $\phi(x)=L^{-1/2}\cos\frac{\pi x}{2L}=e_1(x)$.
Every subsequence has the same limit, and consequently
\[
\widetilde\phi_\sigma\to e_1
\quad\hbox{strongly in }L^2(\R).
\]

\smallskip
\noindent\emph{Step 4: energy and inflow flux.}
Returning to \eqref{eq:6_31}, its right-hand side converges to $\mu_D$.  From
\eqref{eq:6_33}--\eqref{eq:6_34},
\[
\mu_D\le\liminf_{\sigma\to0}E_\sigma(\phi_\sigma).
\]
Since the boundary term in \eqref{eq:6_31} is nonnegative, it follows that
\[
E_\sigma(\phi_\sigma)\longrightarrow\mu_D,
\qquad
\frac{\varepsilon_\sigma q(-L)}2\phi_\sigma(-L)^2\longrightarrow0.
\]
As $q(-L)>0$, $\sigma^{m-2}\phi_\sigma(-L)^2\longrightarrow0$.
Finally, $\mu_\sigma=\sigma^{m-2}\lambda_\sigma$, and \eqref{eq:6_35} yields
\[
\lambda_\sigma
=
\mu_D\sigma^{2-m}+o(\sigma^{2-m})
=
\frac{D_2(J)\pi^2}{8L^2}\sigma^{2-m}+o(\sigma^{2-m}).
\]
\end{proof}

\begin{proposition}\label{prop:criticaladjoint}
For later perturbation arguments we also record the adjoint critical limit and the resulting first-variation limit.

\textup{(i)} Let $\phi_\sigma^*>0$ be the positive adjoint eigenfunction of $\cA_{\sigma,2,L}$. Normalize it by $\|\phi_\sigma^*\|_{L^2(-L,L)}=1$. Then
\[
\phi_\sigma^*\longrightarrow\phi_0^*\quad \hbox{strongly in }L^2(-L,L),
\]
where $\phi_0^*>0$ is the normalized Dirichlet principal eigenfunction of
\[
d_Jv''-(qv)'+av+\lambda_0v=0\quad \hbox{in }(-L,L),\qquad v(-L)=v(L)=0.
\]
The zero-extension energies satisfy
\[
E_\sigma(\phi_\sigma^*)\longrightarrow d_J\int_{-L}^{L}|(\phi_0^*)'|^2\,dx,
\]
and $\phi_\sigma^*(L)\to0$.

\textup{(ii)} Let $h\in C([-L,L])$ and, for $|\tau|$ small, replace $a$ by $a+\tau h$. Denote the corresponding principal eigenvalue by $\lambda_\sigma(\tau)$. Then
\[
\partial_\tau\lambda_\sigma(0)=-\frac{\int_{-L}^{L}h(x)\phi_\sigma(x)\phi_\sigma^*(x)\,dx}{\int_{-L}^{L}\phi_\sigma(x)\phi_\sigma^*(x)\,dx}.
\]
After normalizing the direct and adjoint eigenfunctions in $L^2$,
\[
\partial_\tau\lambda_\sigma(0)\longrightarrow-\frac{\int_{-L}^{L}h(x)\phi_0(x)\phi_0^*(x)\,dx}{\int_{-L}^{L}\phi_0(x)\phi_0^*(x)\,dx}.
\]
The right-hand side is the first variation of the local Dirichlet principal eigenvalue.
\end{proposition}

\begin{proof}[\textbf{Proof of Proposition~\ref{prop:criticaladjoint}\textup{(i)}}]
The nonlocal part is self-adjoint because $J$ is even. Thus the adjoint equation is
\[
\cM_{\sigma,2}\phi_\sigma^*-(q\phi_\sigma^*)'+(a+\lambda_\sigma)\phi_\sigma^*=0,
\]
with the one-sided outflow condition $\phi_\sigma^*(-L)=0$. Here $\lambda_\sigma$ is the same principal eigenvalue as for the direct problem by the Green identity.

Multiply the adjoint equation by $\phi_\sigma^*$ and integrate. The nonlocal term yields $-E_\sigma(\phi_\sigma^*)$. For the drift term we compute without omitting the boundary contribution:
\[
\begin{aligned}
-\int_{-L}^{L}\phi_\sigma^*(q\phi_\sigma^*)'\,dx
&=-\int q\phi_\sigma^*(\phi_\sigma^*)'\,dx-\int q'(\phi_\sigma^*)^2\,dx\\
&=-\frac12[q(\phi_\sigma^*)^2]_{-L}^{L}-\frac12\int q'(\phi_\sigma^*)^2\,dx\\
&=-\frac{q(L)}2\phi_\sigma^*(L)^2-\frac12\int q'(\phi_\sigma^*)^2\,dx.
\end{aligned}
\]
Consequently
\[
E_\sigma(\phi_\sigma^*)+\frac{q(L)}2\phi_\sigma^*(L)^2=\int_{-L}^{L}\left(a+\lambda_\sigma-\frac{q'}2\right)(\phi_\sigma^*)^2\,dx.
\]
The right-hand side is uniformly bounded because $\lambda_\sigma\to\lambda_0$. Hence the adjoint nonlocal energies are uniformly bounded.

Extend $\phi_\sigma^*$ by zero. Proposition~\ref{prop:criticaltools}\textup{(ii)} yields, along a subsequence, strong $L^2(\R)$ convergence to a function $\phi^*\in H^1(\R)$ supported in $[-L,L]$. Thus $\phi^*\in H_0^1(-L,L)$ and $\|\phi^*\|_2=1$. For $\zeta\in C_c^\infty(-L,L)$, symmetry of the kernel yields
\[
\int\zeta\cM_{\sigma,2}\phi_\sigma^*=\int\phi_\sigma^*\cM_{\sigma,2}\zeta\longrightarrow d_J\int\phi^*\zeta''.
\]
For the drift term,
\[
-\int\zeta(q\phi_\sigma^*)'\,dx=\int q\phi_\sigma^*\zeta'\,dx\longrightarrow\int q\phi^*\zeta'\,dx.
\]
Passing to the limit yields
\[
d_J(\phi^*)''-(q\phi^*)'+a\phi^*+\lambda_0\phi^*=0
\]
in distributions. The local adjoint Dirichlet principal eigenfunction is unique and positive, hence $\phi^*=\phi_0^*$. Every subsequence has the same limit, so convergence holds for the full family.

Finally, the exact energy identity converges to
\[
d_J\int_{-L}^{L}|(\phi_0^*)'|^2\,dx=\int_{-L}^{L}\left(a+\lambda_0-\frac{q'}2\right)(\phi_0^*)^2\,dx.
\]
Lower semicontinuity and nonnegativity of the boundary term yield, as in the direct problem,
\[
E_\sigma(\phi_\sigma^*)\to d_J\int| (\phi_0^*)'|^2,\qquad \phi_\sigma^*(L)\to0.
\]
\end{proof}

\begin{proof}[\textbf{Proof of Proposition~\ref{prop:criticaladjoint}\textup{(ii)}}]
For every fixed $\sigma>0$, Proposition~\ref{prop:perturbcore}\textup{(iii)} yields the first-variation formula. By Theorem~\ref{thm:maincritical} and Proposition~\ref{prop:criticaladjoint}\textup{(i)},
\[
\phi_\sigma\to\phi_0,\qquad \phi_\sigma^*\to\phi_0^*
\]
strongly in $L^2$. Hence their products converge in $L^1$:
\[
\|\phi_\sigma\phi_\sigma^*-\phi_0\phi_0^*\|_{L^1}\le\|\phi_\sigma-\phi_0\|_2\|\phi_\sigma^*\|_2+\|\phi_0\|_2\|\phi_\sigma^*-\phi_0^*\|_2\longrightarrow0.
\]
Since $h$ is bounded, both numerator and denominator in the first-variation formula converge. The denominator tends to the strictly positive number $\int\phi_0\phi_0^*$. This proves the limit.
\end{proof}

\section{Subcritical scaling with variable advection}

We first prove the transport-dominated part of Theorem~\ref{thm:mainsubcritical} for variable coefficients.

\begin{proposition}\label{thm:variablesubcritical}
Assume $q,a\in C^2([-L,L])$ and $q_*:=\min_{[-L,L]}q>0$. Then
\[
\sigma^{2-m}\Lp(\cA_{\sigma,m,L})\to\frac{q_*^2}{2D_2(J)}\quad(1<m<2),\qquad
\sigma\Lp(\cA_{\sigma,1,L})\to\cH(q_*),
\]
and
\[
\frac{\sigma}{|\log\sigma|}\Lp(\cA_{\sigma,m,L})\to\frac{q_*(1-m)}{R_+}\qquad(0\le m<1).
\]
\end{proposition}

\begin{proof}[\textbf{Proof of Proposition~\ref{thm:variablesubcritical}}]
For $s\ge0$ set $\phi_s(x)=e^{-s(x+L)/\sigma}$. Since the interval convolution is bounded above by the full convolution,
\[
-\frac{\cA_{\sigma,m,L}\phi_s(x)}{\phi_s(x)}
\ge \frac{q(x)s}{\sigma}-a(x)-\frac{M(s)-1}{\sigma^m}.
\]
Hence
\begin{equation}
\begin{aligned}
\Lp(\cA_{\sigma,m,L})
&\ge \sup_{s\ge0}\left\{\inf_{x\in[-L,L]}
\left(\frac{q(x)s}{\sigma}-a(x)\right)-\frac{M(s)-1}{\sigma^m}\right\}\\
&\ge \sigma^{-m}\cH(q_*\sigma^{m-1})-\max a .
\end{aligned}
\label{eq:bounded-affine}
\end{equation}
The small- and large-argument asymptotics of $\cH$ in Theorem~\ref{thm:heteroHamiltonian} yield the required lower limits.

We turn to the upper bounds. Let $x_0\in[-L,L]$ satisfy $q(x_0)=q_*$. In each regime we choose $x_\sigma\in(-L,L)$ with $x_\sigma\to x_0$, $q(x_\sigma)\to q_*$, and with the distance from $x_\sigma$ to the boundary much larger than the localization scale used below. If $x_0$ is an endpoint this is achieved by moving a distance equal to the square root of that scale into the interval.

Assume first that $1<m<2$. Put $\varepsilon_\sigma=\sigma^{2-m}$ and $\delta_\sigma=\sigma^{m-1}$, so that $\sigma=\varepsilon_\sigma\delta_\sigma$. For fixed $R>0$ and small $\sigma$ the interval
$I_{\sigma,R}=x_\sigma+\varepsilon_\sigma(-R,R)$ is contained in $(-L,L)$. Domain monotonicity and the change of variables $x=x_\sigma+\varepsilon_\sigma y$ imply
\[
\begin{aligned}
\varepsilon_\sigma\Lp(\cA_{\sigma,m,L})
&\le \Lp\!\left(\delta_\sigma^{-2}(J_{\delta_\sigma}*_{(-R,R)}-I)
+q(x_\sigma+\varepsilon_\sigma y)\partial_y\right.\\
&\hspace{35mm}\left.+\varepsilon_\sigma a(x_\sigma+\varepsilon_\sigma y),\,(-R,R)\right).
\end{aligned}
\]
The drift converges to $q_*$ in $C^2([-R,R])$ and the zeroth-order coefficient converges to $0$ in $C^2([-R,R])$. Lemma~\ref{lem:criticalstability} therefore yields
\[
\limsup_{\sigma\to0}\varepsilon_\sigma\Lp(\cA_{\sigma,m,L})
\le d_J\frac{\pi^2}{4R^2}+\frac{q_*^2}{4d_J}.
\]
Letting $R\to\infty$ and combining with \eqref{eq:bounded-affine} proves the case $1<m<2$.

Let $m=1$. Choose $x_\sigma$ at distance much larger than $\sigma$ from the boundary and put $x=x_\sigma+\sigma y$. For fixed $R$,
\[
\sigma\Lp(\cA_{\sigma,1,L})
\le
\Lp\!\left(J*_{(-R,R)}-I+q(x_\sigma+\sigma y)\partial_y
+\sigma a(x_\sigma+\sigma y),(-R,R)\right).
\]
The bounded-interval stability theorem yields, as $\sigma\to0$,
\[
\limsup_{\sigma\to0}\sigma\Lp(\cA_{\sigma,1,L})
\le \Lp(J*_{(-R,R)}-I+q_*\partial_y,(-R,R)).
\]
The right-hand side decreases under exhaustion to the homogeneous whole-line value $\cH(q_*)$. This proves the case $m=1$.

Finally let $0\le m<1$ and put $\delta_\sigma=\sigma^{1-m}$. With the same localization $x=x_\sigma+\sigma y$, set
$q_\sigma(y)=q(x_\sigma+\sigma y)$ and $a_\sigma(y)=\sigma a(x_\sigma+\sigma y)$. For fixed $R>2R_J+1$,
\[
\sigma\Lp(\cA_{\sigma,m,L})
\le\Lp\!\left(\delta_\sigma(J*_{(-R,R)}-I)+q_\sigma\partial_y+a_\sigma,(-R,R)\right).
\]
Let $c_\sigma=\max_{[-R,R]}q_\sigma$ and $b_\sigma=\min_{[-R,R]}a_\sigma$. Then $c_\sigma\to q_*$ and $b_\sigma\to0$. Denote by $\lambda_\sigma^c$ the principal value of
$\delta_\sigma(J*_{(-R,R)}-I)+c_\sigma\partial_y+b_\sigma$.
For small $\sigma$,
\[
\lambda_\sigma^c\ge\delta_\sigma\cH(c_\sigma/\delta_\sigma)-b_\sigma>\delta_\sigma-b_\sigma.
\]
If $\varphi_\sigma$ is the corresponding positive eigenfunction, its equation implies
\[
c_\sigma\varphi_\sigma'
=-\delta_\sigma J*_{(-R,R)}\varphi_\sigma
+(\delta_\sigma-b_\sigma-\lambda_\sigma^c)\varphi_\sigma<0.
\]
Thus $\varphi_\sigma'<0$. Since $q_\sigma\le c_\sigma$ and $a_\sigma\ge b_\sigma$,
\[
\big[\delta_\sigma(J*_{(-R,R)}-I)+q_\sigma\partial_y+a_\sigma+\lambda_\sigma^c\big]\varphi_\sigma\ge0,
\]
and the comparison theorem yields
\begin{equation}
\sigma\Lp(\cA_{\sigma,m,L})\le\lambda_\sigma^c.
\label{eq:bounded-strong-upper}
\end{equation}

Let $s_\sigma>0$ solve $\delta_\sigma M'(s_\sigma)=c_\sigma$ and conjugate the constant-drift problem by $e^{-s_\sigma(y+R)}$. The constant function in the centered problem yields the lower estimate, while a positive $C^2$ test function which is affine in boundary collars of width larger than $R_J$ yields the upper estimate. The first moment cancels because $\delta_\sigma M'(s_\sigma)=c_\sigma$. On the middle interval the remainder is bounded by $C\delta_\sigma M''(s_\sigma)$, and
\[
\delta_\sigma M''(s_\sigma)\le C
\]
because $J$ is compactly supported and $\delta_\sigma M'(s_\sigma)=c_\sigma$. In the boundary collars restriction of the convolution removes a negative affine continuation. Consequently,
\[
\lambda_\sigma^c=\delta_\sigma\cH(c_\sigma/\delta_\sigma)-b_\sigma+\rho_\sigma,
\qquad 0\le\rho_\sigma\le C_R .
\]
Here are the details behind the uniform remainder. After conjugation the centered operator is
\[
\mathcal C_\sigma u(y)=\delta_\sigma\left\{
\int_{\Omega_R}J(y-z)e^{s_\sigma(y-z)}u(z)\,dz-M(s_\sigma)u(y)\right\}
+c_\sigma u'(y),
\]
and $\rho_\sigma=\Lp(\mathcal C_\sigma,\Omega_R)$. Since $\mathcal C_\sigma1\le0$, the constant test gives $\rho_\sigma\ge0$. Choose $w\in C^2(\overline\Omega_R)$ with $w>0$ in $\Omega_R$, $w(\pm R)=0$, affine on collars of width $2R_J$, and bounded below by a positive constant on the remaining middle interval. Extend those affine pieces beyond the endpoints. For the full-line expression, the identity $c_\sigma=\delta_\sigma M'(s_\sigma)$ gives the exact formula
\[
\mathcal C_\sigma^{\R}w(y)=\delta_\sigma\int_\R J(z)e^{s_\sigma z}
\{w(y-z)-w(y)+zw'(y)\}\,dz.
\]
It vanishes wherever $w$ is affine. On the middle interval Taylor's formula gives
\[
|\mathcal C_\sigma^{\R}w(y)|
\le \frac{\|w''\|_\infty}{2}\,\delta_\sigma M''(s_\sigma).
\]
Moreover, for $s\ge0$,
\[
\int_{z>0}z^2J(z)e^{sz}\,dz
\le R_J\int_{z>0}zJ(z)e^{sz}\,dz,
\]
while both negative-half-line moments are uniformly bounded. Since
$\delta_\sigma M'(s_\sigma)=c_\sigma$ and $(c_\sigma)$ is bounded, this proves
$\delta_\sigma M''(s_\sigma)\le C$. Near either endpoint the affine continuation is negative outside $\Omega_R$; restricting the convolution to $\Omega_R$ deletes this negative contribution and therefore only increases the centered operator. Hence
$\mathcal C_\sigma w+C_Rw\ge0$ throughout $\Omega_R$, and Proposition~\ref{thm:comparison} gives $\rho_\sigma\le C_R$.

Since $\delta_\sigma\to0$ and
$\cH(v)=R_+^{-1}v\log v\,(1+o(1))$,
\eqref{eq:bounded-strong-upper} yields
\[
\limsup_{\sigma\to0}\frac{\sigma\Lp(\cA_{\sigma,m,L})}{|\log\sigma|}
\le\frac{q_*(1-m)}{R_+}.
\]
The reverse inequality follows from \eqref{eq:bounded-affine}. This completes the proof.
\end{proof}

\begin{proof}[\textbf{Proof of Theorem~\ref{thm:mainsubcritical}}]
The transport-dominated regimes are Proposition~\ref{thm:variablesubcritical}. The critical regime is Theorem~\ref{thm:maincritical}, and the supercritical regime is Proposition~\ref{prop:supercritical}.
\end{proof}

For completeness, we record first an exact conjugation valid for variable coefficients.

\begin{lemma}\label{lem:generalgauge}
Let $\Psi\in C^1([-L,L])$ and put $\phi=e^{-\Psi}u$. Then
\[
\begin{aligned}
\cA_{\sigma,m,L}\phi(x)
&=e^{-\Psi(x)}\bigg[
\frac1{\sigma^m}\left\{\int_{-L}^{L}J_\sigma(x-y)e^{\Psi(x)-\Psi(y)}u(y)\,dy-u(x)\right\}\\
&\hspace{32mm}+q(x)u'(x)+(a(x)-q(x)\Psi'(x))u(x)\bigg].
\end{aligned}
\]
Consequently, multiplication by $e^{-\Psi}$ preserves the generalized principal value after conjugation.
\end{lemma}

\begin{proof}[\textbf{Proof of Lemma~\ref{lem:generalgauge}}]
Indeed, $\phi'=e^{-\Psi}(u'-\Psi'u)$, while
$\phi(y)=e^{-\Psi(x)}e^{\Psi(x)-\Psi(y)}u(y)$. Hence multiplication by
$e^{-\Psi}$ maps the admissible positive test functions bijectively onto those of the conjugated operator.
\end{proof}

We now assume $q\equiv c>0$. The statements below are refinements of Proposition~\ref{thm:variablesubcritical}; the general phase diagram does not use this assumption. Recall the functions $M$ and $\cH$ introduced in Section~1, and, for $v\ge0$, let $s(v)$ be the unique nonnegative solution of $M'(s(v))=v$. Then $\cH(v)=vs(v)-M(s(v))+1$.

\begin{proposition}\label{thm:conjugation}
Let $v_\sigma=c\sigma^{m-1}$, $s_\sigma=s(v_\sigma)$ and $\kappa_\sigma=s_\sigma/\sigma$. If $\phi(x)=e^{-\kappa_\sigma(x+L)}u(x)$, then
\[
\cA_{\sigma,m,L}\phi=e^{-\kappa_\sigma(x+L)}\left[\cB_\sigma u+a(x)u-\frac{\cH(v_\sigma)}{\sigma^m}u\right],
\]
where
\[
\cB_\sigma u=\frac1{\sigma^m}\left(\int_{-L}^{L}J_\sigma(x-y)e^{\kappa_\sigma(x-y)}u(y)\,dy-M(s_\sigma)u(x)\right)+cu'(x).
\]
Consequently,
\[
\Lp(\cA_{\sigma,m,L})=\frac{\cH(c\sigma^{m-1})}{\sigma^m}+\Lp(\cB_\sigma+a).
\]
\end{proposition}

\begin{proof}[\textbf{Proof of Proposition~\ref{thm:conjugation}}]
Differentiate $\phi$ and factor out the exponential:
\[
\phi'=e^{-\kappa_\sigma(x+L)}(u'-\kappa_\sigma u).
\]
Moreover,
\[
\int J_\sigma(x-y)\phi(y)\,dy=e^{-\kappa_\sigma(x+L)}\int J_\sigma(x-y)e^{\kappa_\sigma(x-y)}u(y)\,dy.
\]
Adding and subtracting $\sigma^{-m}M(s_\sigma)u$ yields a scalar coefficient
\[
\frac{M(s_\sigma)-1}{\sigma^m}-c\kappa_\sigma=-\frac{v_\sigma s_\sigma-M(s_\sigma)+1}{\sigma^m}=-\frac{\cH(v_\sigma)}{\sigma^m}.
\]
Multiplication by the positive exponential is a bijection between the positive test functions in the definition of $\Lp$, and the identity follows.
\end{proof}

\begin{lemma}\label{thm:centered}
Assume $q\equiv c>0$. Fix $\eta\in(0,1)$. There exists $C_\eta>0$ such that $|\Lp(\cB_\sigma+a)|\le C_\eta$
for every $m\in[1,2-\eta]$ and all sufficiently small $\sigma$.
\end{lemma}

\begin{proof}[\textbf{Proof of Lemma~\ref{thm:centered}}]
The lower bound is immediate from the constant function. Since the interval convolution is dominated by the full convolution,
\[
\cB_\sigma1=\frac1{\sigma^m}\left(\int_{-L}^{L}J_\sigma(x-y)e^{\kappa_\sigma(x-y)}\,dy-M(s_\sigma)\right)\le0.
\]
Consequently $\Lp(\cB_\sigma+a)\ge-\max a\ge-\|a\|_\infty$.

We prove the upper bound with one fixed function. Set $k=\pi/(2L)$ and $w(x)=\sin(k(x+L)),\quad -L\le x\le L$.
Then $w>0$ in $(-L,L)$ and $w(-L)=w(L)=0$. Let $W(x)=\sin(k(x+L))$ be the same sine function on the whole line. For points whose distance to the boundary is at least $R_J\sigma$, the interval and full-space convolutions coincide. Using
\[
W(x-\sigma z)=W(x)\cos(k\sigma z)-\cos(k(x+L))\sin(k\sigma z),
\]
we obtain
\[
\begin{aligned}
\cB_\sigma w(x)
&=A_\sigma w(x)+B_\sigma\cos(k(x+L)),\\
A_\sigma&=\frac1{\sigma^m}\int_\R J(z)e^{s_\sigma z}\big(\cos(k\sigma z)-1\big)\,dz,\\
B_\sigma&=ck-\frac1{\sigma^m}\int_\R J(z)e^{s_\sigma z}\sin(k\sigma z)\,dz.
\end{aligned}
\]
The identity $M'(s_\sigma)=c\sigma^{m-1}$ centers the first moment. Taylor's formula, with remainders uniform for $m\in[1,2-\eta]$, yields
\[
A_\sigma=-\frac{k^2}{2}M''(s_\sigma)\sigma^{2-m}+O(\sigma^{4-m})
\]
and
\[
\begin{aligned}
\frac1{\sigma^m}\int J(z)e^{s_\sigma z}\sin(k\sigma z)\,dz
&=k\sigma^{1-m}M'(s_\sigma)-\frac{k^3}{6}M^{(3)}(s_\sigma)\sigma^{3-m}+O(\sigma^{5-m})\\
&=ck-\frac{k^3}{6}M^{(3)}(s_\sigma)\sigma^{3-m}+O(\sigma^{5-m}).
\end{aligned}
\]
Thus
\[
|A_\sigma|\le C\sigma^{2-m},\qquad |B_\sigma|\le C\sigma^{3-m}.
\]
Since $s_\sigma$ remains in a compact set for $m\in[1,2-\eta]$, all derivatives of $M$ appearing here are uniformly bounded.

Choose $C_0>0$ large and first consider points satisfying $w(x)\ge C_0\sigma^{3-m}$.
The preceding estimates imply
\[
\cB_\sigma w(x)\ge-C\sigma^{2-m}w(x)-C\sigma^{3-m}\ge-C_1w(x).
\]
It remains to treat the two very thin collars where $w(x)<C_0\sigma^{3-m}$. Because $3-m>1$, these collars have width $o(\sigma)$. We treat the right endpoint; the left endpoint is obtained by reflection.

Write $x=L-\sigma t$, where $0\le t\le C\sigma^{2-m}$. The full sine continuation is negative for $y>L$ close to $L$. Hence truncating the convolution removes a negative contribution and increases $\cB_\sigma w$. More precisely, the correction is
\[
R_\sigma^+(x)=-\frac1{\sigma^m}\int_{y>L}J_\sigma(x-y)e^{\kappa_\sigma(x-y)}W(y)\,dy.
\]
With $y=x-\sigma z$, the condition $y>L$ is $z<-t$ and
\[
W(L-\sigma(t+z))=k\sigma(t+z)+O(\sigma^3|t+z|^3).
\]
Since $t+z<0$ on the integration region,
\[
R_\sigma^+(x)=k\sigma^{1-m}\int_{z<-t}J(z)e^{s_\sigma z}(-t-z)\,dz+O(\sigma^{3-m}).
\]
Because $J(0)>0$, there exist $\rho,j_0>0$ with $J(z)\ge j_0$ on $[-\rho,\rho]$. Since $t=o(1)$ and $s_\sigma$ is uniformly bounded, for small $\sigma$
\[
\int_{z<-t}J(z)e^{s_\sigma z}(-t-z)\,dz\ge c_0>0.
\]
Therefore $R_\sigma^+(x)\ge c_1\sigma^{1-m}$, whereas the full-space centered expression is only $O(\sigma^{3-m})$. The positive correction dominates and yields $\cB_\sigma w(x)\ge0$ in the right thin collar. The same computation at $-L$ uses the fact that the sine continuation is again negative outside the interval and yields $\cB_\sigma w(x)\ge0$ there.

Combining the interior and boundary estimates, there is a constant $C_2$, independent of $m\in[1,2-\eta]$ and small $\sigma$, such that
\[
\cB_\sigma w+(C_2+\|a\|_\infty)w+aw\ge0\ \hbox{in }(-L,L).
\]
The function $w$ is positive in the interval and vanishes at the outflow endpoint. The comparison theorem therefore yields $\Lp(\cB_\sigma+a)\le C_2+\|a\|_\infty$.
Together with the lower bound, this proves the theorem.
\end{proof}

\begin{lemma}\label{thm:centeredbarrier}
Assume $q\equiv c>0$. Let $\Omega=(\alpha,\beta)$ be a fixed bounded interval, and let $\cB_{\sigma,\Omega}$ denote the centered operator in Proposition~\ref{thm:conjugation} with the integral restricted to $\Omega$. For fixed $0\le m<2$, put
\[
\eta_\sigma=
\begin{cases}
\sigma,&0\le m<1,\\
\sigma^{2-m},&1\le m<2.
\end{cases}
\]
There are $\sigma_0>0$ and $C_\Omega>0$, independent of $\sigma$, such that
\[
0\le\Lp(\cB_{\sigma,\Omega})\le C_\Omega\eta_\sigma
\qquad(0<\sigma<\sigma_0).
\]
\end{lemma}

\begin{proof}[\textbf{Proof of Lemma~\ref{thm:centeredbarrier}}]
The constant function yields the lower bound because
\[
\cB_{\sigma,\Omega}1
=\frac1{\sigma^m}\left(
\int_\Omega J_\sigma(x-y)e^{s_\sigma(x-y)/\sigma}\,dy-M(s_\sigma)\right)\le0.
\]

For the upper bound choose $\rho>0$ with $4\rho<\beta-\alpha$ and a function $w\in C^3([\alpha,\beta])$ satisfying
\[
\begin{gathered}
w>0\ \hbox{in }\Omega,\qquad w(\alpha)=w(\beta)=0,\\
w(x)=x-\alpha\quad(\alpha\le x\le\alpha+2\rho),\qquad
w(x)=\beta-x\quad(\beta-2\rho\le x\le\beta),\\
w\ge c_\rho>0\quad\hbox{on }[\alpha+\rho,\beta-\rho].
\end{gathered}
\]
Take $\sigma_0$ so small that $R_J\sigma_0<\rho$.

Affine boundary pieces yield exact first-moment cancellation. If $p(x)=A+Bx$, then its full-line centered expression is
\[
\frac1{\sigma^m}\int_{\mathbb R}J(z)e^{s_\sigma z}
\{p(x-\sigma z)-p(x)\}\,dz+cp'(x)
=-B\sigma^{1-m}M'(s_\sigma)+cB=0,
\]
because $M'(s_\sigma)=c\sigma^{m-1}$.

On the middle interval, no truncation occurs for small $\sigma$. Taylor's formula yields
\[
\begin{aligned}
|\cB_{\sigma,\Omega}w(x)|
&\le \frac{\sigma^{2-m}}2M''(s_\sigma)\|w''\|_\infty
+\frac{\sigma^{3-m}}6M_3(s_\sigma)\|w'''\|_\infty,\\
M_3(s)&=\int_{\mathbb R}|z|^3J(z)e^{sz}\,dz .
\end{aligned}
\]
If $1\le m<2$, then $s_\sigma$ remains bounded and this is at most $C\sigma^{2-m}$. If $0\le m<1$, then $s_\sigma\to+\infty$. For all large $s$, $M''(s)+M_3(s)\le C\{1+M'(s)\}$.
Indeed, on $z\ge0$, $z^2\le R_Jz$ and $z^3\le R_J^2z$; the negative-half-line contributions are uniformly bounded, while their contribution to $M'(s)$ is uniformly bounded in absolute value. Since $M'(s_\sigma)=c\sigma^{m-1}$, the Taylor bound is then at most $C\sigma$.

It remains to verify that truncation has the favorable sign. In the left collar extend $w(y)=y-\alpha$ affinely to $y<\alpha$; this extension is negative outside $\Omega$. The full-line centered expression is zero, and restriction to $\Omega$ deletes the integral of a negative function against the positive tilted kernel. Hence $\cB_{\sigma,\Omega}w\ge0$ there. The continuation $w(y)=\beta-y$ is negative for $y>\beta$, so the identical argument yields $\cB_{\sigma,\Omega}w\ge0$ in the right collar. On the middle interval $w\ge c_\rho$, and the preceding Taylor estimate therefore yields
\[
\cB_{\sigma,\Omega}w+C_\Omega\eta_\sigma w\ge0
\quad\hbox{throughout }\Omega
\]
after increasing $C_\Omega$. The bounded comparison theorem yields
$\Lp(\cB_{\sigma,\Omega})\le C_\Omega\eta_\sigma$.
\end{proof}

\begin{lemma}\label{thm:centerlimit}
Assume $q\equiv c>0$. For every fixed $0\le m<2$,
\[
\Lp(\cB_\sigma+a)\longrightarrow-\max_{[-L,L]}a\quad \hbox{as }\sigma\to0.
\]
The convergence is locally uniform with respect to $m$ on compact subsets of $[0,1)$ and of $[1,2)$.
\end{lemma}

\begin{proof}[\textbf{Proof of Lemma~\ref{thm:centerlimit}}]
Let $a_{\max}=\max_{[-L,L]}a$. The constant function yields $\Lp(\cB_\sigma+a)\ge-a_{\max}$.
Fix $\varepsilon>0$. By continuity of $a$, there is an interval $\Omega=(\alpha,\beta)\Subset(-L,L)$ such that $a(x)\ge a_{\max}-\varepsilon\quad x\in \Omega$.
Domain and potential monotonicity yield
\[
\Lp(\cB_\sigma+a,(-L,L))\le\Lp(\cB_\sigma+a,\Omega)\le\Lp(\cB_\sigma,\Omega)-a_{\max}+\varepsilon.
\]
It remains to prove $\Lp(\cB_\sigma,\Omega)\to0$. Lemma~\ref{thm:centeredbarrier} yields
\[
0\le\Lp(\cB_\sigma,\Omega)\le C_\Omega\eta_\sigma\longrightarrow0
\]
for every fixed $0\le m<2$.
Hence
\[
-a_{\max}\le\liminf_{\sigma\to0}\Lp(\cB_\sigma+a)\le\limsup_{\sigma\to0}\Lp(\cB_\sigma+a)\le-a_{\max}+\varepsilon.
\]
Letting $\varepsilon\downarrow0$ proves the result. Uniformity follows from the explicit barrier bounds; the two formulae for $\eta_\sigma$ account for the harmless change at $m=1$.
\end{proof}

\begin{proposition}\label{thm:subcritical}
Assume $q\equiv c>0$ and $a\in C([-L,L])$. For every fixed $0\le m<2$,
\[
\Lp(\cA_{\sigma,m,L})=\frac{\cH(c\sigma^{m-1})}{\sigma^m}-\max_{[-L,L]}a+o(1).
\]
If $1<m<2$, then
\[
\sigma^{2-m}\Lp(\cA_{\sigma,m,L})\longrightarrow\frac{c^2}{2D_2(J)}.
\]
If $m=1$, then
\[
\sigma\Lp(\cA_{\sigma,1,L})\longrightarrow\cH(c).
\]
\end{proposition}

\begin{proof}[\textbf{Proof of Proposition~\ref{thm:subcritical}}]
Proposition~\ref{thm:conjugation} and Lemma~\ref{thm:centerlimit} imply the sharp representation. Since $J$ is even,
\[
M(s)=1+\frac{D_2(J)}2s^2+\frac{D_4(J)}{24}s^4+O(s^6),
\]
so series inversion of $M'(s(v))=v$ yields
\[
s(v)=\frac{v}{D_2(J)}-\frac{D_4(J)}{6D_2(J)^4}v^3+O(v^5).
\]
Substitution in $\cH(v)=vs(v)-M(s(v))+1$ yields
\[
\cH(v)=\frac{v^2}{2D_2(J)}-\frac{D_4(J)}{24D_2(J)^4}v^4+O(v^6).
\]
For $1<m<2$, $v_\sigma=c\sigma^{m-1}\to0$, and therefore
\[
\frac{\cH(v_\sigma)}{\sigma^m}=\frac{c^2}{2D_2(J)}\sigma^{m-2}+O(\sigma^{3m-4}).
\]
The ratio of the displayed error to the leading term is $O(\sigma^{2m-2})\to0$, while the bounded term $-\max a+o(1)$ is also $o(\sigma^{m-2})$. This proves the first limit. For $m=1$, $v_\sigma=c$ and the sharp representation yields $\Lp=\sigma^{-1}\cH(c)-\max a+o(1)$.
\end{proof}

\begin{proposition}\label{thm:higher}
Assume $q\equiv c>0$ and $a\in C([-L,L])$. For compactly supported $J$, the functions $M$ and $\cH$ are analytic near the origin and $\cH(v)=\sum_{k\ge1}h_{2k}v^{2k}$,
where, writing $D_j=D_j(J)=\int_{\mathbb R}z^jJ(z)\,dz$,
\[
h_2=\frac1{2D_2},\qquad
h_4=-\frac{D_4}{24D_2^4},\qquad
h_6=\frac{D_4^2}{72D_2^7}-\frac{D_6}{720D_2^6}.
\]
For every integer $K\ge1$ and every fixed $1<m<2$,
\[
\Lp(\cA_{\sigma,m,L})
=\sum_{k=1}^{K}h_{2k}c^{2k}\sigma^{(2k-1)m-2k}
-\max a
+O\!\left(\sigma^{(2K+1)m-(2K+2)}\right)+o(1).
\]
More intrinsically, define $N(m)=\max\{k\ge1:(2k-1)m-2k\le0\}$.
Then the expansion through every nonvanishing order is
\[
\Lp(\cA_{\sigma,m,L})
=\sum_{k=1}^{N(m)}h_{2k}c^{2k}\sigma^{(2k-1)m-2k}
-\max a+o(1).
\]
The $k$th correction becomes order one at $m=2k/(2k-1)$; the successive thresholds are $4/3,6/5,8/7,\ldots$.
\end{proposition}

\begin{proof}[\textbf{Proof of Proposition~\ref{thm:higher}}]
The exact conjugation yields $\Lp=\sigma^{-m}\cH(c\sigma^{m-1})-\max a+o(1)$.
Compact support makes $M$ entire. Since $M'(0)=0$ and $M''(0)=D_2>0$, the analytic inverse-function theorem yields an odd analytic inverse $s=s(v)$ of $M'(s)=v$ near zero. Therefore $\cH'(v)=s(v),\qquad \cH(0)=0$,
and $\cH$ is even and analytic. Series reversion in
\[
v=D_2s+\frac{D_4}{6}s^3+\frac{D_6}{120}s^5+O(s^7)
\]
yields
\[
s(v)=\frac v{D_2}-\frac{D_4}{6D_2^4}v^3
+\left(\frac{D_4^2}{12D_2^7}-\frac{D_6}{120D_2^6}\right)v^5
+O(v^7).
\]
Integrating $s=\cH'$ proves the three coefficient formulae. Taylor expansion through degree $2K$ has remainder $O(v^{2K+2})$. Substitution of $v=c\sigma^{m-1}$ and multiplication by $\sigma^{-m}$ produces $O\!\left(\sigma^{(2K+1)m-(2K+2)}\right)$.
For fixed $m>1$, the exponent $(2k-1)m-2k$ is strictly increasing in $k$ and eventually positive. Taking $K=N(m)$ makes the remainder exponent, which is the exponent of the $(N(m)+1)$st term, strictly positive. This proves the complete finite expansion and the threshold statement.
\end{proof}

\begin{corollary}\label{cor:higher-thresholds}
For $4/3<m<2$, $\Lp=\frac{c^2}{2D_2}\sigma^{m-2}-\max a+o(1)$.
At $m=4/3$,
\[
\Lp=\frac{c^2}{2D_2}\sigma^{-2/3}+h_4c^4-\max a+o(1).
\]
For $6/5<m<4/3$,
\[
\Lp=\frac{c^2}{2D_2}\sigma^{m-2}
+h_4c^4\sigma^{3m-4}-\max a+o(1),
\]
whereas at $m=6/5$ one must add the finite term $h_6c^6$.
\end{corollary}

\begin{proof}
Apply Proposition~\ref{thm:higher}. The exponents of the first three Hamiltonian corrections are $m-2$, $3m-4$, and $5m-6$. Their signs change precisely at $m=2$, $m=4/3$, and $m=6/5$, respectively. Retaining exactly the nonvanishing terms gives the four displayed regimes.
\end{proof}

\section{Strong-transport refinements}

Assume $q\equiv c>0$ and consider $0\le m<1$. In this refinement $v_\sigma=c\sigma^{m-1}\to+\infty$ and the optimizing tilt $s_\sigma=s(v_\sigma)$ is unbounded. The bounded-tilt sine estimate is no longer uniform, but the conjugation formula remains valid. Recall $R_+=\sup\{z>0:J(z)>0\}$ from Section~1. Since $J$ is even and nontrivial, $R_+>0$ and $J$ has positive mass in every left neighborhood of $R_+$.

The affine-collar construction in Lemma~\ref{thm:centeredbarrier} has constants which remain controlled when $s_\sigma\to+\infty$ and yields
\[
\Lp(\cB_\sigma+a)=-\max_{[-L,L]}a+o(1)
\qquad(0\le m<1).
\]

\begin{proposition}\label{thm:boundedstrong}
Assume $q\equiv c>0$ and $a\in C([-L,L])$. For every fixed $0\le m<1$,
\[
\Lp(\cA_{\sigma,m,L})
=\frac{\cH(c\sigma^{m-1})}{\sigma^m}
-\max_{[-L,L]}a+o(1)
\]
and
\[
\lim_{\sigma\to0}\frac{\sigma}{|\log\sigma|}\Lp(\cA_{\sigma,m,L})=\frac{c(1-m)}{R_+}.
\]
If, in addition, for some $K>0$ and $\nu>1$,
\[
J(R_+-t)=Kt^\nu(1+o(1))\quad\hbox{as }t\downarrow0,
\]
then, with $p=\nu+1$ and $A=K\Gamma(p)$,
\[
\Lp(\cA_{\sigma,m,L})
=\frac{c}{R_+\sigma}
\left[
(1-m)|\log\sigma|+p\log|\log\sigma|+C_{m,c,J}+o(1)
\right],
\]
where
\[
C_{m,c,J}=\log c+p\log(1-m)-p\log R_+-\log(R_+A)-1.
\]
\end{proposition}

\begin{proof}[\textbf{Proof of Proposition~\ref{thm:boundedstrong}}]
Proposition~\ref{thm:conjugation} and Lemma~\ref{thm:centerlimit} yield
\[
\Lp(\cA_{\sigma,m,L})
=\frac{\cH(c\sigma^{m-1})}{\sigma^m}
-\max_{[-L,L]}a+o(1).
\]
Since
\[
\log(c\sigma^{m-1})=(1-m)|\log\sigma|+\log c,\qquad
\sigma^{-m}c\sigma^{m-1}=\frac c\sigma,
\]
the leading edge asymptotic for $\cH$ yields the normalized limit.

For the refinement put $A=K\Gamma(p)$. Laplace's method at the right edge yields
\[
M(s)=Ae^{R_+s}s^{-p}(1+o(1)),\qquad
M'(s)=R_+Ae^{R_+s}s^{-p}(1+o(1)).
\]
Thus, for $M'(s(v))=v$,
\[
R_+s(v)=\log v+p\log\log v-p\log R_+-\log(R_+A)+o(1),
\]
while $M(s(v))=v/R_++o(v)$. Therefore
\[
\cH(v)=\frac v{R_+}
\left[\log v+p\log\log v-p\log R_+-\log(R_+A)-1+o(1)\right].
\]
For $v_\sigma=c\sigma^{m-1}$, $\log v_\sigma=(1-m)|\log\sigma|+\log c$,
and
\[
\log\log v_\sigma
=\log|\log\sigma|+\log(1-m)+o(1).
\]
Since $\sigma^{-m}v_\sigma=c/\sigma$, substitution yields
\[
\frac{\cH(v_\sigma)}{\sigma^m}
=\frac{c}{R_+\sigma}
\left[
(1-m)|\log\sigma|+p\log|\log\sigma|+C_{m,c,J}+o(1)
\right].
\]
Finally $-\max a+o(1)=o(\sigma^{-1})$, so it is absorbed by the $o(1)$ term inside the bracket. This proves the full edge-controlled expansion.
\end{proof}

\section{Large dispersal ranges}

We begin with the constant-kernel problem used in the large-range comparison.

\begin{proposition}\label{thm:constantkernel}
Let $\varepsilon>0$, $\delta\ge0$, $q\in C([-L,L])$ with $q\ge q_0>0$, and $a\in C([-L,L])$. Consider
\[
\cT_{\varepsilon,\delta}u
=q(x)u'(x)+\varepsilon\int_{-L}^{L}u(y)\,dy-\delta u(x)+a(x)u(x).
\]
Its principal eigenvalue $\lambda_{\varepsilon,\delta}$ is the unique real number satisfying
\[
1=\varepsilon\int_{-L}^{L}\int_x^L\frac1{q(t)}
\exp\left(\int_x^t\frac{a(r)+\lambda_{\varepsilon,\delta}-\delta}{q(r)}\,dr\right)dt\,dx.
\]
Up to multiplication by a positive constant, the corresponding eigenfunction is
\[
\phi_{\varepsilon,\delta}(x)
=\int_x^L\frac1{q(t)}
\exp\left(\int_x^t\frac{a(r)+\lambda_{\varepsilon,\delta}-\delta}{q(r)}\,dr\right)dt.
\]
Put
\[
T_q=\int_{-L}^{L}\frac{dx}{q(x)},\qquad
A_q=\int_{-L}^{L}\frac{a(x)}{q(x)}\,dx.
\]
Then, as $\varepsilon\downarrow0$,
\[
\lambda_{\varepsilon,\delta}
=\delta+\frac1{T_q}\left[
\log\frac1\varepsilon+2\log\log\frac1\varepsilon
-2\log T_q-\log q(-L)-A_q
\right]+o(1).
\]
The remainder is independent of $\delta$; in particular it is uniform when $\delta$ ranges in arbitrary subsets of $[0,\infty)$.
\end{proposition}

\begin{corollary}\label{cor:constantkernel-homogeneous}
If $q\equiv c>0$ and $a\equiv a_0$, put
$z=2L(a_0+\lambda_{\varepsilon,\delta}-\delta)/c$. Then
\[
\frac{e^z-1-z}{z^2}=\frac{c}{4\varepsilon L^2},\qquad
\lambda_{\varepsilon,\delta}=\frac{cz}{2L}-a_0+\delta,\qquad
\phi_z(x)=e^{z(L-x)/(2L)}-1.
\]
The root is unique, and $z>0$ exactly when $c/(4\varepsilon L^2)>1/2$.
\end{corollary}

\begin{proof}
Direct integration yields
\[
1=\frac{\varepsilon}{c\gamma^2}\big(e^{2L\gamma}-1-2L\gamma\big),
\qquad
\gamma=\frac{a_0+\lambda-\delta}{c}.
\]
Putting $z=2L\gamma$ yields the scalar equation and the explicit eigenfunction. Since
$F'(z)=\int_0^1t(1-t)e^{tz}\,dt>0$, the root is unique on $\R$, and $z>0$ exactly when $F(0)=1/2$ is smaller than the prescribed right-hand side.
\end{proof}

\begin{proof}[\textbf{Proof of Proposition~\ref{thm:constantkernel}}]
Let $S=\int_{-L}^{L}\phi$. The eigenvalue equation, together with the outflow condition $\phi(L)=0$, is
\[
q(x)\phi'(x)+(a(x)+\lambda-\delta)\phi(x)=-\varepsilon S.
\]
Terminal integration yields
\[
\phi(x)=\varepsilon S\int_x^L\frac1{q(t)}
\exp\left(\int_x^t\frac{a(r)+\lambda-\delta}{q(r)}\,dr\right)dt.
\]
The right-hand side is strictly positive for $x<L$. Integrating in $x$ and cancelling $S>0$ yields the scalar equation in the statement. Its right-hand side, without the prefactor $\varepsilon$, is a continuous strictly increasing function of $\lambda$, tends to zero as $\lambda\to-\infty$ and to $+\infty$ as $\lambda\to+\infty$. Hence the scalar equation has a unique real root. The displayed positive function then solves the principal eigenvalue problem, and uniqueness follows from geometric simplicity.

We now derive the asymptotics. Set $\rho=\lambda_{\varepsilon,\delta}-\delta$ and introduce the travel-time coordinate
\[
\tau=Q(x):=\int_{-L}^x\frac{dr}{q(r)},\qquad 0\le\tau\le T_q,
\]
with inverse $x=X(\tau)$. The scalar integral becomes
\[
I(\rho)=\int_0^{T_q}q(X(r))\int_r^{T_q}
\exp\left(\rho(s-r)+\int_r^s a(X(\xi))\,d\xi\right)ds\,dr,
\qquad \varepsilon I(\rho)=1.
\]
Since $\varepsilon I(\rho)=1$ and $I$ is finite at every finite $\rho$, the identity $I(\rho)=\varepsilon^{-1}\to+\infty$, together with the strict monotonicity of $I$, first yields $\rho\to+\infty$. We may therefore use a corner Laplace expansion. Put
\[
Q(r):=q(X(r)),\qquad A(r):=a(X(r)),\qquad
A_q=\int_0^{T_q}A(r)\,dr=\int_{-L}^{L}\frac{a(x)}{q(x)}\,dx.
\]
With the corner variables $u=r$ and $v=T_q-s$, the integral takes the form
\[
\begin{aligned}
I(\rho)
&=e^{\rho T_q+A_q}
\int_{\substack{u,v\ge0\\u+v\le T_q}}
Q(u)e^{-\rho(u+v)}\\
&\qquad\times
\exp\left(-\int_0^uA(\xi)\,d\xi-\int_{T_q-v}^{T_q}A(\xi)\,d\xi\right)dv\,du.
\end{aligned}
\]
After the dilation $U=\rho u$, $V=\rho v$,
\[
\rho^2e^{-\rho T_q-A_q}I(\rho)
=\int_{\substack{U,V\ge0\\U+V\le\rho T_q}}
Q(U/\rho)e^{-(U+V)}R_\rho(U,V)\,dV\,dU,
\]
where
\[
R_\rho(U,V)=
\exp\left(-\int_0^{U/\rho}A(\xi)\,d\xi
-\int_{T_q-V/\rho}^{T_q}A(\xi)\,d\xi\right).
\]
Extend the integrand by zero from the triangle $U+V\le\rho T_q$ to the whole positive quadrant. For every fixed $(U,V)$ the triangle indicator tends to one and
$Q(U/\rho)R_\rho(U,V)\to Q(0)=q(-L)$. The extended integrand is bounded by $q_1e^{2\|a\|_\infty T_q}e^{-(U+V)}$, which is integrable on the positive quadrant. Dominated convergence therefore yields the sharp corner Laplace asymptotic
\begin{equation}
I(\rho)=q(-L)e^{A_q}\frac{e^{\rho T_q}}{\rho^2}(1+o(1))
\qquad(\rho\to+\infty).
\label{eq:9_1}
\end{equation}
Taking logarithms in \eqref{eq:9_1} and using $\varepsilon I(\rho)=1$ yields
\[
\rho T_q
=\log\frac1\varepsilon+2\log\rho-\log q(-L)-A_q+o(1).
\]
The coarse consequence $\rho T_q/\log(1/\varepsilon)\to1$ yields
\[
\log\rho=\log\log\frac1\varepsilon-\log T_q+o(1).
\]
Reinsertion yields
\[
\rho=\frac1{T_q}\left[
\log\frac1\varepsilon+2\log\log\frac1\varepsilon
-2\log T_q-\log q(-L)-A_q
\right]+o(1).
\]
Since $\lambda_{\varepsilon,\delta}=\delta+\rho$, the stated expansion follows. The scalar equation for $\rho$ does not contain $\delta$, so the remainder is independent of $\delta$.
\end{proof}

\begin{proposition}\label{thm:largerange}
Assume $q\in C([-L,L])$, $q\ge q_0>0$, $a\in C([-L,L])$, $m\ge0$ and $J(0)>0$. Put
\[
T_q=\int_{-L}^{L}\frac{dx}{q(x)},\qquad
A_q=\int_{-L}^{L}\frac{a(x)}{q(x)}\,dx,
\]
and
\[
C_m:=\mathbf 1_{\{m=0\}}
+\frac1{T_q}\left[
2\log\frac{m+1}{T_q}-\log\big(J(0)q(-L)\big)-A_q
\right].
\]
Then
\[
\Lp(\cA_{\sigma,m,L})
=\frac{m+1}{T_q}\log\sigma
+\frac{2}{T_q}\log\log\sigma
+C_m+o(1)
\qquad(\sigma\to+\infty).
\]
In particular, $\Lp(\cA_{\sigma,m,L})\to+\infty$.
\end{proposition}

\begin{proof}[\textbf{Proof of Proposition~\ref{thm:largerange}}]
Since $|x-y|\le2L$ on $[-L,L]^2$, define
\[
j_-(\sigma):=\min_{|z|\le2L/\sigma}J(z),\qquad
j_+(\sigma):=\max_{|z|\le2L/\sigma}J(z).
\]
Continuity and $J(0)>0$ yield $j_-(\sigma),j_+(\sigma)\longrightarrow J(0)$,
and, for all sufficiently large $\sigma$,
\[
\frac{j_-(\sigma)}{\sigma}\le J_\sigma(x-y)\le\frac{j_+(\sigma)}{\sigma}
\qquad (x,y\in[-L,L]).
\]
Thus the gain kernel of $\cA_{\sigma,m,L}$ is squeezed between rank-one kernels with
\[
\varepsilon_\pm(\sigma)=j_\pm(\sigma)\sigma^{-(m+1)},
\qquad
\delta_\sigma=\sigma^{-m}.
\]
Kernel monotonicity yields
\begin{equation}
\lambda_{\varepsilon_+(\sigma),\delta_\sigma}
\le\Lp(\cA_{\sigma,m,L})
\le\lambda_{\varepsilon_-(\sigma),\delta_\sigma}.
\label{eq:9_2}
\end{equation}
For either sign,
\[
\log\frac1{\varepsilon_\pm(\sigma)}
=(m+1)\log\sigma-\log J(0)+o(1),
\]
and hence
\[
\log\log\frac1{\varepsilon_\pm(\sigma)}
=\log\log\sigma+\log(m+1)+o(1).
\]
Proposition~\ref{thm:constantkernel} therefore yields
\[
\begin{aligned}
\lambda_{\varepsilon_\pm(\sigma),\delta_\sigma}
={}&\frac{m+1}{T_q}\log\sigma+\frac2{T_q}\log\log\sigma\\
&+\delta_\sigma
+\frac1{T_q}\left[
2\log\frac{m+1}{T_q}-\log\big(J(0)q(-L)\big)-A_q
\right]+o(1).
\end{aligned}
\]
Since $\delta_\sigma\to\mathbf 1_{\{m=0\}}$ for fixed $m\ge0$, both comparison values in \eqref{eq:9_2} have the same three-term expansion. The sandwich proves the result.
\end{proof}

\begin{lemma}\label{thm:sigmacontinuity}
Fix $m\ge0$, $L>0$, $q\in C([-L,L])$ with $q\ge q_0>0$ and $a\in C([-L,L])$. Then
\[
(0,\infty)\ni\sigma\longmapsto\Lp(\cA_{\sigma,m,L})
\]
is continuous. The principal eigenfunction normalized at one fixed interior point depends continuously on $\sigma$ in $C^1([-L,L])$ on compact subintervals of $(0,\infty)$.
\end{lemma}

\begin{proof}[\textbf{Proof of Lemma~\ref{thm:sigmacontinuity}}]
Let $\sigma_n\to\sigma_*>0$. Write the operator in general-kernel form. On $[-L,L]^2$,
\[
K_n(x,y)=\sigma_n^{-(m+1)}J\left(\frac{x-y}{\sigma_n}\right)
\]
converges uniformly to $K_*(x,y)$, while the zeroth-order coefficient $b_n(x)=a(x)-\sigma_n^{-m}$
converges uniformly to $b_*(x)$. Since $\sigma_n$ remains in a compact subset of $(0,\infty)$ and $J(0)>0$, the kernels have a common strictly positive neighborhood along the diagonal. Proposition~\ref{thm:generalKstability} yields convergence of the principal values and $C^1$ convergence of the normalized eigenfunctions.
\end{proof}

\begin{lemma}\label{thm:optimalrange}
Assume $q\in C([-L,L])$, $q\ge q_0>0$, $a\in C([-L,L])$ and $0\le m<2$. Then there exists $\sigma_*\in(0,\infty)$ such that
\[
\Lp(\cA_{\sigma_*,m,L})=\min_{\sigma>0}\Lp(\cA_{\sigma,m,L}).
\]
The set of minimizers is nonempty and compact.
\end{lemma}

\begin{proof}[\textbf{Proof of Lemma~\ref{thm:optimalrange}}]
Lemma~\ref{thm:sigmacontinuity} yields continuity. At the zero-range end, for $s\ge0$ the test function
$\phi_s(x)=e^{-s(x+L)/\sigma}$ satisfies
\[
\Lp(\cA_{\sigma,m,L})\ge \sigma^{-m}\cH(q_0\sigma^{m-1})-\max_{[-L,L]}a,
\]
and the right-hand side tends to $+\infty$ for every $0\le m<2$. At the large-range end, Proposition~\ref{thm:largerange} yields the same divergence. Hence every sublevel set below one fixed finite value is contained in a compact subinterval of $(0,\infty)$. The continuous function attains its global minimum there, and the set of minimizers is closed in that compact interval.
\end{proof}

\begin{proof}[\textbf{Proof of Theorem~\ref{thm:mainlarge}}]
The asymptotic expansion is Proposition~\ref{thm:largerange}. For $0\le m<2$, existence and compactness of the set of minimizers follow from Lemma~\ref{thm:optimalrange}.
\end{proof}

\begin{lemma}\label{thm:boundedthreshold}
Let $b\in C([-L,L])$ satisfy $b>0$ and, for fixed $\sigma>0$ and $m\ge0$, set
\[
\Lambda_{\sigma,m}(r)=\Lp(\cA_{\sigma,m,L}+rb),\qquad r\in\R.
\]
Then $\Lambda_{\sigma,m}$ is continuous and strictly decreasing from $+\infty$ to $-\infty$. Consequently there is a unique $r_{\sigma,m}$ such that $\Lambda_{\sigma,m}(r_{\sigma,m})=0$. If $b\equiv1$, then $r_{\sigma,m}=\Lp(\cA_{\sigma,m,L})$.

At the critical scale $m=2$, let $r_0$ be the unique zero of
\[
r\longmapsto\lambda_1^D\big(d_J\partial_{xx}+q(x)\partial_x+a+rb,(-L,L)\big).
\]
Then $r_{\sigma,2}\to r_0$ as $\sigma\to0$.
\end{lemma}

\begin{proof}[\textbf{Proof of Lemma~\ref{thm:boundedthreshold}}]
Let $b_0=\min b>0$ and $b_1=\max b$. Continuity follows from Lemma~\ref{thm:monotonicity}. If $r_2>r_1$, then potential monotonicity and the constant-shift identity imply
\[
\Lambda_{\sigma,m}(r_2)\le\Lambda_{\sigma,m}(r_1)-(r_2-r_1)b_0<\Lambda_{\sigma,m}(r_1).
\]
For $r\ge0$,
\[
\Lambda_{\sigma,m}(0)-rb_1\le\Lambda_{\sigma,m}(r)\le\Lambda_{\sigma,m}(0)-rb_0,
\]
and the inequalities are reversed when $r\le0$. Hence $\Lambda_{\sigma,m}(r)\to-\infty$ as $r\to+\infty$ and $\Lambda_{\sigma,m}(r)\to+\infty$ as $r\to-\infty$. This proves existence and uniqueness of $r_{\sigma,m}$. If $b\equiv1$, the constant-shift identity yields $r_{\sigma,m}=\Lp(\cA_{\sigma,m,L})$.

The local Dirichlet principal value has the same strict order property, so $r_0$ is well defined. For fixed $\varepsilon>0$, its values at $r_0-\varepsilon$ and $r_0+\varepsilon$ have respectively positive and negative signs. Theorem~\ref{thm:maincritical}, applied to these two fixed potentials, yields the same signs for the nonlocal principal values for all sufficiently small $\sigma$. Thus $r_0-\varepsilon<r_{\sigma,2}<r_0+\varepsilon$. Letting $\varepsilon\downarrow0$ proves the convergence.
\end{proof}

\section{Blow-up of the zeroth-order and dispersal terms}

In the last part of the paper, we investigate the behavior of $\Lp$ as the zeroth-order and dispersal terms blow up, as well as when the dispersal intensity degenerates. This is the nonlocal fixed-drift counterpart of the coefficient-amplitude analysis in \cite[Theorems~9.3 and~9.5]{BR}. The drift prevents a literal transfer of the self-adjoint elliptic conclusions. In particular, the potential amplitude retains the order-theoretic concavity mechanism, whereas the dispersal amplitude does not inherit the self-adjoint monotonicity argument.

Let $\Omega\subset\R$ be an interval and let $V\in C_b(\Omega)$. For $\gamma\in\R$ and $\alpha>0$, define
\[
\begin{aligned}
\cA^{V}_{\gamma,\Omega}u(x)
&=d\left(\int_\Omega J(x-y)u(y)\,dy-u(x)\right)+q(x)u'(x)+\gamma V(x)u(x),\\
\cA^{D}_{\alpha,\Omega}u(x)
&=\alpha d\left(\int_\Omega J(x-y)u(y)\,dy-u(x)\right)+q(x)u'(x)+V(x)u(x),
\end{aligned}
\]
and set
\[
\Lambda_V^\Omega(\gamma)=\Lp(\cA^{V}_{\gamma,\Omega}),\qquad
\Lambda_D^\Omega(\alpha)=\Lp(\cA^{D}_{\alpha,\Omega}).
\]

The bounded coefficient-amplitude statement and the whole-line heterogeneous Hamiltonian representation were stated in Theorems~\ref{thm:coeffamplitudes} and~\ref{thm:heteroHamiltonian}. We now prove them. The bounded potential parameter is controlled by order and geometric interpolation, while analyticity follows from the Fredholm analytic implicit-function theorem. On the line, the logarithmic transform converts the generalized principal value into an exact nonlinear Collatz--Wielandt Hamiltonian; affine phases then recover the exponential-moment structure.

\begin{proof}[\textbf{Proof of Theorem~\ref{thm:coeffamplitudes}}]
We separate the order properties from the analytic perturbation argument.

\smallskip
\noindent\emph{Step 1: concavity, Lipschitz continuity and finiteness of $\Lambda_V^\Omega$.}
Fix $\gamma_1,\gamma_2\in\R$, $0<\theta<1$, and put $\gamma_\theta=\theta\gamma_1+(1-\theta)\gamma_2$. For $i=1,2$, choose $\lambda_i<\Lambda_V^\Omega(\gamma_i)$ and a positive admissible function $\phi_i$ such that
\[
\cA^V_{\gamma_i,\Omega}\phi_i+\lambda_i\phi_i\le0\quad\hbox{in }\Omega.
\]
Set $\phi=\phi_1^\theta\phi_2^{1-\theta}$. For fixed $x\in\Omega$, H\"older's inequality and the weighted arithmetic--geometric mean inequality yield
\[
\begin{aligned}
\frac{\int_\Omega J(x-y)\phi(y)\,dy}{\phi(x)}
&=\int_\Omega J(x-y)
\left(\frac{\phi_1(y)}{\phi_1(x)}\right)^\theta
\left(\frac{\phi_2(y)}{\phi_2(x)}\right)^{1-\theta}dy\\
&\le
\left(\frac{\int_\Omega J(x-y)\phi_1(y)\,dy}{\phi_1(x)}\right)^\theta
\left(\frac{\int_\Omega J(x-y)\phi_2(y)\,dy}{\phi_2(x)}\right)^{1-\theta}\\
&\le \theta\frac{\int_\Omega J(x-y)\phi_1(y)\,dy}{\phi_1(x)}
+(1-\theta)\frac{\int_\Omega J(x-y)\phi_2(y)\,dy}{\phi_2(x)}.
\end{aligned}
\]
Moreover,
\[
\frac{\phi'}{\phi}=\theta\frac{\phi_1'}{\phi_1}+(1-\theta)\frac{\phi_2'}{\phi_2},
\qquad
\gamma_\theta V=\theta\gamma_1V+(1-\theta)\gamma_2V.
\]
Therefore
\[
\frac{\cA^V_{\gamma_\theta,\Omega}\phi}{\phi}
\le \theta\frac{\cA^V_{\gamma_1,\Omega}\phi_1}{\phi_1}
+(1-\theta)\frac{\cA^V_{\gamma_2,\Omega}\phi_2}{\phi_2}
\le-\theta\lambda_1-(1-\theta)\lambda_2.
\]
Taking $\lambda_i\uparrow\Lambda_V^\Omega(\gamma_i)$ proves concavity. At $\gamma=0$ the constant function is admissible at level zero because the truncated mass of $J$ is at most one; hence $\Lambda_V^\Omega(0)\ge0$. Potential monotonicity yields
\[
|\Lambda_V^\Omega(\gamma_1)-\Lambda_V^\Omega(\gamma_2)|
\le \|(\gamma_1-\gamma_2)V\|_\infty,
\]
and the same comparison with constant potentials shows that all these values are finite.

\smallskip
\noindent\emph{Step 2: asymptotic slopes.}
Let $S=\sup_\Omega V$. For $\gamma>0$, the constant test yields $\Lambda_V^\Omega(\gamma)\ge-\gamma S$.
Fix $\varepsilon>0$, continuity of $V$ provides a bounded interval $B\Subset\Omega$ on which $V\ge S-\varepsilon$. By domain and potential monotonicity,
\[
\Lambda_V^\Omega(\gamma)
\le \Lambda_V^B(\gamma)
\le \Lambda_V^B(0)-\gamma(S-\varepsilon).
\]
Division by $\gamma$ and then $\gamma\to+\infty$ yields
\[
-S\le\liminf_{\gamma\to+\infty}\frac{\Lambda_V^\Omega(\gamma)}\gamma
\le\limsup_{\gamma\to+\infty}\frac{\Lambda_V^\Omega(\gamma)}\gamma
\le-(S-\varepsilon).
\]
Letting $\varepsilon\downarrow0$ proves the first slope. The second follows by applying the first one to $-V$ and $-\gamma$.

\smallskip
\noindent\emph{Step 3: real analyticity on a bounded interval.}
Assume now that $\Omega=(a_-,a_+)$ is bounded and $q\in C^1(\overline\Omega)$. Set
\[
X=\{u\in C^1(\overline\Omega):u(a_+)=0\},\qquad Z=C(\overline\Omega),
\]
and fix $x_0\in\Omega$. We treat simultaneously the two affine families. For a real parameter $\tau$ write $\mathcal L_\tau=\mathcal L_0+\tau\cB$,
where either $\cB u=V u$ and $\tau=\gamma$, or
\[
\cB u=d\left(\int_\Omega J(x-y)u(y)\,dy-u(x)\right),\qquad \tau=\alpha.
\]
At a fixed real $\tau_0$ let $(\lambda_0,\phi_0)$ be the normalized principal eigenpair, $\phi_0(x_0)=1$. On the complexifications $X_{\mathbb C},Z_{\mathbb C}$ define
\[
\mathscr F(\tau,\lambda,u)=\big((\mathcal L_\tau+\lambda)u,\ u(x_0)-1\big).
\]
Because $\tau\mapsto\mathcal L_\tau$ is affine, $\mathscr F$ is holomorphic in all variables. Its derivative in $(\lambda,u)$ at $(\tau_0,\lambda_0,\phi_0)$ is
\[
(\mu,h)\longmapsto
\big((\mathcal L_{\tau_0}+\lambda_0)h+\mu\phi_0,\ h(x_0)\big).
\]
The transport part with the terminal condition at $a_+$ is an isomorphism $X\to Z$ and the nonlocal gain is compact, so the first component is Fredholm of index zero. The augmented operator is therefore Fredholm of index zero as well. It is injective: pairing the first component with the positive adjoint eigenfunction yields $\mu=0$, geometric simplicity yields $h=C\phi_0$, and $h(x_0)=0$ yields $C=0$. Hence the augmented operator is a bounded isomorphism, and the same is true after complexification. The analytic implicit-function theorem produces a unique holomorphic eigenpair branch near $\tau_0$. Positivity on the real axis identifies this branch with the principal eigenpair. Since $\tau_0$ was arbitrary, $\Lambda_V^\Omega$ is real analytic on $\R$ and $\Lambda_D^\Omega$ is real analytic on $(0,\infty)$.

Differentiating the eigenvalue equation and pairing with the positive adjoint eigenfunction, normalized by $\int_\Omega\phi\phi^*=1$, yields
\[
(\Lambda_V^\Omega)'(\gamma)
=-\int_\Omega V\phi_\gamma\phi_\gamma^*,
\]
and
\[
(\Lambda_D^\Omega)'(\alpha)
=-d\int_\Omega\phi_\alpha^*(x)
\left(\int_\Omega J(x-y)\phi_\alpha(y)\,dy-\phi_\alpha(x)\right)dx.
\]
The second formula has no fixed sign in general because the factor $J*\phi_\alpha-\phi_\alpha$ changes sign.

\smallskip
\noindent\emph{Step 4: strict concavity for a nonconstant potential.}
Suppose $V$ is not constant and that concavity is not strict. Then there exist $\gamma_1\ne\gamma_2$ and $0<\theta<1$ such that
\[
\Lambda_V^\Omega(\gamma_\theta)
=\theta\Lambda_V^\Omega(\gamma_1)+(1-\theta)\Lambda_V^\Omega(\gamma_2).
\]
Let $\phi_i$ be the positive principal eigenfunction at $\gamma_i$, and set $\phi=\phi_1^\theta\phi_2^{1-\theta}$. The computation in Step~1 now yields
\[
\cA^V_{\gamma_\theta,\Omega}\phi+
\Lambda_V^\Omega(\gamma_\theta)\phi\le0.
\]
The geometric mean is an admissible $C^1$ function up to the outflow endpoint. Indeed, if $\Omega=(a_-,a_+)$, the endpoint identity gives $\phi_i'(a_+)<0$ and
\[
\phi_i(x)=-\phi_i'(a_+)(a_+-x)+o(a_+-x),\qquad
\phi_i'(x)=\phi_i'(a_+)+o(1).
\]
Consequently
\[
\phi(x)=\{-\phi_1'(a_+)\}^{\theta}
\{-\phi_2'(a_+)\}^{1-\theta}(a_+-x)+o(a_+-x),
\]
and the logarithmic derivative formula shows that $\phi'$ has the corresponding finite limit at $a_+$. Thus $\phi(a_+)=0$ and $\phi\in X$. Multiply the last inequality by the positive adjoint eigenfunction at $\gamma_\theta$ and integrate. The Green identity turns the integral into zero. Since the adjoint eigenfunction is strictly positive and the residual is continuous and nonpositive, the residual vanishes pointwise. In the computation of Step~1 the gaps in H\"older's inequality and in the arithmetic--geometric mean inequality are separately nonnegative. Pointwise equality between the first and last expressions therefore forces both gaps to vanish. Thus equality holds in both inequalities in Step~1.

Fix $x\in\Omega$. Equality in H\"older's inequality implies that, on the set where $J(x-y)>0$,
\[
\frac{\phi_1(y)}{\phi_1(x)}=C_x\frac{\phi_2(y)}{\phi_2(x)}.
\]
Because $J>0$ in a neighborhood of the origin and both ratios equal one at $y=x$, continuity yields $C_x=1$. Hence $\phi_1/\phi_2$ is locally constant. The interaction chains generated by the neighborhood on which $J>0$ connect the interval, so $\phi_1=C\phi_2$ on $\Omega$. Subtracting the two eigenvalue equations yields
\[
(\gamma_1-\gamma_2)V(x)+\Lambda_V^\Omega(\gamma_1)-\Lambda_V^\Omega(\gamma_2)=0
\quad\hbox{for all }x\in\Omega,
\]
which forces $V$ to be constant, a contradiction. Thus $\Lambda_V^\Omega$ is strictly concave.

Finally, a differentiable concave function has a nonincreasing derivative. The derivative limits exist, and the fundamental theorem of calculus yields
\[
\frac{\Lambda_V^\Omega(\gamma)-\Lambda_V^\Omega(0)}{\gamma}
=\frac1\gamma\int_0^\gamma(\Lambda_V^\Omega)'(t)\,dt.
\]
For a monotone function the Ces\`aro average has the same limit as the function itself. The asymptotic slopes from Step~2 therefore yield the two limits of $(\Lambda_V^\Omega)'$. Strict concavity makes the derivative strictly decreasing. This completes the proof.
\end{proof}

\begin{proof}[\textbf{Proof of Theorem~\ref{thm:heteroHamiltonian}}]
\smallskip
\noindent\emph{Step 1: nonlinear Collatz--Wielandt representation.}
For $\psi\in C^1(\R)$ put $\phi=e^{-\psi}$. Compact support of $J$ makes every finite-jump exponential moment below well defined. With the change of variables $z=x-y$,
\[
\frac{1}{\phi(x)}\int_\R J(x-y)\phi(y)\,dy
=\int_\R J(z)e^{\psi(x)-\psi(x-z)}\,dz,
\qquad
\frac{\phi'}\phi=-\psi'.
\]
Consequently
\[
-\frac{\cA^D_{\alpha,\R}\phi(x)}{\phi(x)}
=\mathfrak H_\alpha[\psi](x).
\]
The principal value is finite: the constant test gives the lower bound $\Lambda_D^\R(\alpha)\ge-\sup_\R V$, while domain monotonicity on one fixed bounded interval gives a finite upper bound. For this fixed $\psi$, every $\lambda<\inf_x\mathfrak H_\alpha[\psi](x)$ is admissible in the definition of $\Lambda_D^\R(\alpha)$. Hence
\[
\Lambda_D^\R(\alpha)
\ge\sup_{\psi\in C^1(\R)}\inf_x\mathfrak H_\alpha[\psi](x).
\]
Conversely, if $\lambda$ is admissible and $\phi>0$ is an admissible test function, then $\psi=-\log\phi$ satisfies
\[
\lambda\le\mathfrak H_\alpha[\psi](x)\quad\hbox{for every }x,
\]
so $\lambda\le\inf_x\mathfrak H_\alpha[\psi](x)$. Taking the supremum over all admissible levels proves the reverse inequality and therefore the exact formula.

The whole-line exhaustion theorem provides a positive principal eigenfunction $\phi_\alpha$. For $\psi_\alpha=-\log\phi_\alpha$, the eigenvalue equation yields
\[
\mathfrak H_\alpha[\psi_\alpha](x)\equiv\Lambda_D^\R(\alpha),
\]
so the supremum is attained.

\smallskip
\noindent\emph{Step 2: the optimal affine phase.}
Take $\psi(x)=sx$. Then
\[
\mathfrak H_\alpha[sx](x)=q(x)s-V(x)-\alpha d(M(s)-1).
\]
If $s<0$, positivity of $q$ and $M(s)\ge1$ imply $\inf_x\mathfrak H_\alpha[sx](x)\le-\sup_\R V$,
which is the value at $s=0$. Thus only $s\ge0$ need be considered, and the exact formula yields
\[
\Lambda_D^\R(\alpha)\ge
\sup_{s\ge0}\left\{\inf_x(q(x)s-V(x))-\alpha d(M(s)-1)\right\}.
\]
Since $q(x)\ge q_0$, $\inf_x(q(x)s-V(x))\ge q_0s-\sup_\R V$,
and optimization in $s$ yields
\[
\Lambda_D^\R(\alpha)\ge
\alpha d\,\cH\left(\frac{q_0}{\alpha d}\right)-\sup_\R V.
\]
The function $g(s)=\inf_x(q(x)s-V(x))$ is concave as the infimum of affine functions of $s$. Since $M''(s)>0$ for all $s$ and $M(s)$ grows exponentially as $s\to+\infty$, the function $s\longmapsto g(s)-\alpha d(M(s)-1)$
is strictly concave and tends to $-\infty$ at $+\infty$. Hence its maximizing affine slope is unique.

\smallskip
\noindent\emph{Step 3: asymptotics of the kernel Hamiltonian.}
Since $J$ is even,
\[
M(s)=1+\frac{D_2(J)}2s^2+\frac{D_4(J)}{24}s^4+O(s^6),\qquad
M'(s)=D_2(J)s+\frac{D_4(J)}6s^3+O(s^5).
\]
Inverting $M'(s)=v$ at the origin and substituting into $\cH(v)=vs(v)-M(s(v))+1$ yields
\[
\cH(v)=\frac{v^2}{2D_2(J)}-\frac{D_4(J)}{24D_2(J)^4}v^4+O(v^6).
\]

For the large-argument behavior, $M(s)\le e^{R_+s}$ for $s\ge0$. With
$s_v=R_+^{-1}(\log v-2\log\log v)$,
\[
\cH(v)\ge \frac v{R_+}(\log v-2\log\log v)-\frac{v}{(\log v)^2}+1.
\]
Conversely, for $\varepsilon\in(0,R_+)$ put
$m_\varepsilon=\int_{R_+-\varepsilon}^{R_+}J(z)\,dz>0$. Then
$M(s)\ge m_\varepsilon e^{(R_+-\varepsilon)s}$ and
\[
\cH(v)\le \frac{v}{R_+-\varepsilon}\log v+O(v).
\]
Letting $\varepsilon\downarrow0$ proves $\cH(v)/(v\log v)\to R_+^{-1}$.

Under the edge hypothesis, Laplace's method yields
\[
M(s)=Ae^{R_+s}s^{-p}(1+o(1)),\qquad
M'(s)=R_+Ae^{R_+s}s^{-p}(1+o(1)).
\]
If $M'(s(v))=v$, then
\[
R_+s(v)=\log v+p\log\log v-p\log R_+-\log(R_+A)+o(1),
\]
while $M(s(v))=v/R_++o(v)$. Substitution proves the stated expansion of $\cH$. This completes the proof.
\end{proof}

\begin{proposition}\label{thm:commute}
Assume $q,a\in C_b^2(\R)$ and $q\ge q_0>0$, and let
\[
\lambda_0:=\Lp(d_J\partial_{xx}+q(x)\partial_x+a(x),\R).
\]
At the critical scaling $m=2$,
\[
\lim_{L\to\infty}\lim_{\sigma\to0}\Lp(\cA_{\sigma,2,(-L,L)})
=
\lim_{\sigma\to0}\lim_{L\to\infty}\Lp(\cA_{\sigma,2,(-L,L)})
=\lambda_0.
\]
More precisely, for every fixed $L$,
\[
\lim_{\sigma\to0}\Lp(\cA_{\sigma,2,(-L,L)})
=
\lambda_1^D(d_J\partial_{xx}+q(x)\partial_x+a(x),(-L,L)),
\]
whereas for every fixed $\sigma>0$,
\[
\lim_{L\to\infty}\Lp(\cA_{\sigma,2,(-L,L)})
=\Lp(\cA_{\sigma,2,\R}).
\]
\end{proposition}

\begin{corollary}\label{cor:commuting-homogeneous}
If $q\equiv c>0$ and $a\equiv a_0$, the common iterated limit in
Proposition~\ref{thm:commute} is $c^2/(2D_2(J))-a_0$. For fixed $L$ the inner limit is
$d_J\pi^2/(4L^2)+c^2/(4d_J)-a_0$, while for fixed $\sigma$ the exhaustion limit is
$\sigma^{-2}\cH(c\sigma)-a_0$.
\end{corollary}

\begin{proof}
The fixed-$L$ formula follows from the gauge transform
$u=e^{-cx/(2d_J)}v$. The fixed-$\sigma$ formula is the homogeneous whole-line
identity in Corollary~\ref{cor:homogeneous-wholeline-amplitude}. Finally,
$\cH(v)=v^2/(2D_2(J))+O(v^4)$ as $v\to0$, and hence
$\sigma^{-2}\cH(c\sigma)\to c^2/(2D_2(J))$.
\end{proof}

\begin{proof}[\textbf{Proof of Proposition~\ref{thm:commute}}]
For fixed $L$, Theorem~\ref{thm:maincritical} yields
\[
\Lp(\cA_{\sigma,2,(-L,L)})
\longrightarrow
\lambda_1^D(d_J\partial_{xx}+q(x)\partial_x+a(x),(-L,L)).
\]
The local Dirichlet principal values decrease under exhaustion to $\lambda_0$.
This proves the first iterated limit.

For fixed $\sigma>0$, the exhaustion theorem for the nonlocal operator yields
\[
\Lp(\cA_{\sigma,2,(-L,L)})\downarrow\Lp(\cA_{\sigma,2,\R})
\qquad(L\to\infty).
\]
The critical row of Theorem~\ref{thm:mainwholephase} then implies
$\Lp(\cA_{\sigma,2,\R})\to\lambda_0$ as $\sigma\to0$. The two iterated limits therefore coincide.
\end{proof}

\begin{proof}[\textbf{Proof of Theorem~\ref{thm:mainwholephase}}]
The logarithmic formula is Theorem~\ref{thm:heteroHamiltonian} after replacing $J$ by $J_\sigma$.
For $s\ge0$, the affine phase $\psi(x)=sx/\sigma$ yields
\begin{equation}
\lambda_{\sigma,m}\ge \sigma^{-m}\cH(q_*\sigma^{m-1})-\sup_\R a.
\label{eq:whole-affine}
\end{equation}
This already yields the required lower limits for $m<2$ and, after multiplication by
$\sigma^{m-2}$, the lower limit $0$ for $m>2$.

We prove the matching upper bounds by localization near points where the drift is almost minimal. For each $\sigma$ choose $x_\sigma$ so that $q(x_\sigma)\le q_*+\sigma$. The bounded derivatives of $q$ imply that every rescaling used below converges locally uniformly to the constant drift $q_*$.

If $1<m<2$, put $\varepsilon_\sigma=\sigma^{2-m}$ and
$\delta_\sigma=\sigma^{m-1}$, so that $\sigma=\varepsilon_\sigma\delta_\sigma$. With
$x=x_\sigma+\varepsilon_\sigma y$, multiplication of the eigenvalue equation by
$\varepsilon_\sigma$ yields the rescaled operator
\[
\mathcal B_\sigma v
=\delta_\sigma^{-2}(J_{\delta_\sigma}*v-v)
+q(x_\sigma+\varepsilon_\sigma y)v'
+\varepsilon_\sigma a(x_\sigma+\varepsilon_\sigma y)v.
\]
Hence $\varepsilon_\sigma\lambda_{\sigma,m}=\Lp(\mathcal B_\sigma,\R)$. For every fixed $R>0$, domain monotonicity yields
$\Lp(\mathcal B_\sigma,\R)\le\Lp(\mathcal B_\sigma,(-R,R))$.
On $[-R,R]$ the drift converges in $C^2$ to $q_*$ and the zeroth-order coefficient converges in $C^2$ to $0$. Lemma~\ref{lem:criticalstability} therefore yields
\[
\Lp(\mathcal B_\sigma,(-R,R))
\longrightarrow
\lambda_1^D(d_J\partial_{yy}+q_*\partial_y,(-R,R))
=d_J\frac{\pi^2}{4R^2}+\frac{q_*^2}{4d_J}.
\]
Letting $R\to\infty$ yields
$\limsup\varepsilon_\sigma\lambda_{\sigma,m}\le q_*^2/(4d_J)=q_*^2/(2D_2(J))$.
Together with \eqref{eq:whole-affine}, this proves the row $1<m<2$.

Let $m=1$ and set $x=x_\sigma+\sigma y$. Then
\[
\sigma\lambda_{\sigma,1}
=\Lp\!\left(J*-I+q(x_\sigma+\sigma y)\partial_y+\sigma a(x_\sigma+\sigma y),\R\right).
\]
For fixed $R$, domain monotonicity followed by Proposition~\ref{thm:stability} yields
\[
\limsup_{\sigma\to0}\sigma\lambda_{\sigma,1}
\le\Lp(J*-I+q_*\partial_y,(-R,R)).
\]
Exhaustion as $R\to\infty$ and the homogeneous whole-line formula yield the upper bound
$\cH(q_*)$. The lower bound follows from \eqref{eq:whole-affine}.

It remains to treat $0\le m<1$. Put $\delta_\sigma=\sigma^{1-m}$ and again use
$x=x_\sigma+\sigma y$. Put $q_\sigma(y)=q(x_\sigma+\sigma y)$ and
$a_\sigma(y)=\sigma a(x_\sigma+\sigma y)$. Then
\begin{equation}
\sigma\lambda_{\sigma,m}
=\Lp\!\left(\delta_\sigma(J*-I)+q_\sigma\partial_y+a_\sigma,\R\right).
\label{eq:whole-rescaled}
\end{equation}
Fix $R>R_J+1$ and set $\Omega_R=(-R,R)$,
$c_\sigma=\max_{\Omega_R}q_\sigma$ and $b_\sigma=\min_{\Omega_R}a_\sigma$.
Then $c_\sigma\to q_*$ and $b_\sigma\to0$.
Let $\lambda_\sigma^c$ be the principal eigenvalue on $\Omega_R$ of
\[
\delta_\sigma(J*_{\Omega_R}-I)+c_\sigma\partial_y+b_\sigma.
\]
For small $\sigma$, $\lambda_\sigma^c>\delta_\sigma-b_\sigma$.
Indeed, by domain monotonicity it is bounded below by the corresponding homogeneous whole-line value
$\delta_\sigma\cH(c_\sigma/\delta_\sigma)-b_\sigma$, which diverges.
If $\varphi_\sigma$ is its positive eigenfunction, then
\[
c_\sigma\varphi_\sigma'
=-\delta_\sigma J*_{\Omega_R}\varphi_\sigma
+(\delta_\sigma-b_\sigma-\lambda_\sigma^c)\varphi_\sigma<0.
\]
Thus $\varphi_\sigma'<0$. Since $q_\sigma\le c_\sigma$ and $a_\sigma\ge b_\sigma$,
\[
\big[\delta_\sigma(J*_{\Omega_R}-I)+q_\sigma\partial_y+a_\sigma+\lambda_\sigma^c\big]\varphi_\sigma\ge0.
\]
The bounded comparison theorem and \eqref{eq:whole-rescaled} therefore yield
\begin{equation}
\sigma\lambda_{\sigma,m}\le\lambda_\sigma^c.
\label{eq:whole-upper}
\end{equation}

We record the required weak-dispersal estimate for $\lambda_\sigma^c$. Let $s_\sigma>0$ solve
$\delta_\sigma M'(s_\sigma)=c_\sigma$ and conjugate by $e^{-s_\sigma(y+R)}$. Then
\begin{equation}
\lambda_\sigma^c
=\delta_\sigma\cH(c_\sigma/\delta_\sigma)-b_\sigma+\rho_\sigma,
\qquad 0\le\rho_\sigma\le C_R,
\label{eq:weak-centered}
\end{equation}
with $C_R$ independent of $\sigma$. The lower bound follows from the constant test function in the centered problem. For the upper bound choose a positive $C^2$ function $w$ which vanishes at $\pm R$, is affine in boundary collars of width larger than $R_J$, and is bounded below away from zero on the remaining middle interval. For the full-line centered operator,
\[
\delta_\sigma\!\int J(z)e^{s_\sigma z}
\{w(y-z)-w(y)+zw'(y)\}\,dz
\]
vanishes on the affine pieces and is bounded on the middle interval by
$C\delta_\sigma M''(s_\sigma)$. Since $\delta_\sigma M'(s_\sigma)=c_\sigma$ and $J$ is supported in
$[-R_J,R_J]$, the positive-half-line contribution satisfies
$\delta_\sigma\int_{z>0}z^2J(z)e^{s_\sigma z}\,dz
\le R_J(c_\sigma+O(\delta_\sigma))$, while the negative-half-line contribution is
$O(\delta_\sigma)$. Hence $\delta_\sigma M''(s_\sigma)\le C$ uniformly. Restricting the convolution to
$\Omega_R$ deletes a negative affine continuation and hence has the favorable sign in both boundary collars. Consequently the centered operator satisfies
$\mathcal C_\sigma w+C_Rw\ge0$, and the comparison theorem yields \eqref{eq:weak-centered}.

Now $\delta_\sigma\downarrow0$ and
$\cH(v)=R_+^{-1}v\log v\,(1+o(1))$ as $v\to\infty$. From \eqref{eq:whole-upper}--\eqref{eq:weak-centered},
\[
\limsup_{\sigma\to0}\frac{\sigma\lambda_{\sigma,m}}{|\log\sigma|}
\le\frac{q_*(1-m)}{R_+},
\]
because $\log(c_\sigma/\delta_\sigma)=(1-m)|\log\sigma|+O(1)$.
The affine lower bound \eqref{eq:whole-affine} yields the reverse inequality.

For $m=2$, let $\mathcal L_0=d_J\partial_{xx}+q(x)\partial_x+a(x)$ and
$\lambda_0=\Lp(\mathcal L_0,\R)$. For every fixed $L$, domain monotonicity and
Theorem~\ref{thm:maincritical} yield
\[
\limsup_{\sigma\to0}\lambda_{\sigma,2}
\le\lambda_1^D(\mathcal L_0,(-L,L)),
\]
and the right-hand side decreases to $\lambda_0$ as $L\to\infty$.
Conversely, the one-dimensional local theory on the whole line provides a positive eigenfunction
$\phi_0$ at the generalized principal level $\lambda_0$; see, for instance, \cite[Theorem~1.4]{BR}.  Thus
$\mathcal L_0\phi_0+\lambda_0\phi_0=0$ on $\R$.  The local Harnack inequality yields
$\sup_{[x-1,x+1]}\phi_0\le C\inf_{[x-1,x+1]}\phi_0$ with $C$ independent of $x$.
We justify the relative derivative estimates needed in the nonlocal remainder. By the mean-value theorem there is $\xi_x\in[x,x+1]$ such that
\[
|\phi_0'(\xi_x)|=|\phi_0(x+1)-\phi_0(x)|\le C\phi_0(x).
\]
The vector $(\phi_0,\phi_0')$ solves a first-order system whose coefficient matrix is uniformly bounded. Gronwall's inequality between $x$ and $\xi_x$, followed by Harnack, gives
$|\phi_0'(x)|\le C\phi_0(x)$ uniformly in $x$. Writing the equation as
\[
d_J\phi_0''=-q\phi_0'-(a+\lambda_0)\phi_0
\]
first gives $|\phi_0''|\le C\phi_0$. Differentiating once gives
$d_J\phi_0'''=-q\phi_0''-(q'+a+\lambda_0)\phi_0'-a'\phi_0$, and differentiating once more expresses the fourth derivative as a linear combination of $\phi_0,\phi_0',\phi_0'',\phi_0'''$ with coefficients bounded by the $C_b^2$ norms of $q$ and $a$. The preceding bounds therefore close the induction and yield
\[
|\phi_0^{(k)}(x)|\le C\phi_0(x),\qquad 1\le k\le4,\quad x\in\R.
\]
Taylor's formula, the evenness of $J$, compact support and the Harnack comparison of $\phi_0(x-\theta\sigma z)$ with $\phi_0(x)$ now yield
\[
\sigma^{-2}(J_\sigma*\phi_0-\phi_0)
=d_J\phi_0''+O(\sigma^2\phi_0)
\]
uniformly on $\R$.  Hence $\cA_{\sigma,2,\R}\phi_0+(\lambda_0-C\sigma^2)\phi_0\le0$ and
$\lambda_{\sigma,2}\ge\lambda_0-C\sigma^2$.

Finally assume $m>2$. Set $\mu_\sigma=\sigma^{m-2}\lambda_{\sigma,m}$. Then \eqref{eq:whole-affine} implies $\liminf\mu_\sigma\ge0$. By domain monotonicity and Theorem~\ref{thm:mainsubcritical}, for every fixed $L$,
\[
\limsup_{\sigma\downarrow0}\mu_\sigma\le \frac{d_J\pi^2}{4L^2}.
\]
Letting $L\to\infty$ completes the proof.
\end{proof}

\begin{corollary}\label{cor:wholephase-homogeneous}
If $q\equiv c>0$ and $a\equiv a_0$, then
$\lambda_{\sigma,m}=\sigma^{-m}\cH(c\sigma^{m-1})-a_0$ for every $\sigma>0$ and $m\ge0$. Hence, for every fixed $m>1$ and $N\ge1$,
\[
\lambda_{\sigma,m}=-a_0+\sum_{k=1}^{N}h_{2k}c^{2k}\sigma^{(2k-1)m-2k}+O\!\left(\sigma^{(2N+1)m-2N-2}\right),
\]
where $\cH(v)=\sum_{k\ge1}h_{2k}v^{2k}$ near $0$ and $h_2=(2D_2)^{-1}$, $h_4=-D_4/(24D_2^4)$, $h_6=(10D_4^2-D_2D_6)/(720D_2^7)$. In particular,
\[
\lambda_{\sigma,2}=\frac{c^2}{2D_2}-a_0-\frac{D_4c^4}{24D_2^4}\sigma^2+\frac{(10D_4^2-D_2D_6)c^6}{720D_2^7}\sigma^4+O(\sigma^6).
\]
If $0\le m<1$ and $J(R_+-t)=Kt^\nu(1+o(1))$ as $t\downarrow0$, with $p=\nu+1$ and $A=K\Gamma(p)$, then
\[
\lambda_{\sigma,m}=\frac{c}{R_+\sigma}\left[(1-m)|\log\sigma|+p\log|\log\sigma|+C_{m,c,J}+o(1)\right]-a_0,
\]
where $C_{m,c,J}=\log c+p\log(1-m)-p\log R_+-\log(R_+A)-1$.
\end{corollary}

\begin{proof}
For $q\equiv c$ and $a\equiv a_0$, the affine phase is exact and Theorem~\ref{thm:heteroHamiltonian} yields $\lambda_{\sigma,m}=\sigma^{-m}\cH(c\sigma^{m-1})-a_0$. Since $M$ is entire,
\[
\cH(v)=\frac{v^2}{2D_2}-\frac{D_4}{24D_2^4}v^4+\frac{10D_4^2-D_2D_6}{720D_2^7}v^6+O(v^8),
\]
which proves the expansion for $m>1$. Under the edge hypothesis, Laplace's method yields $M(s)=Ae^{R_+s}s^{-p}(1+o(1))$ and $M'(s)=R_+Ae^{R_+s}s^{-p}(1+o(1))$. If $M'(s(v))=v$, then $R_+s(v)=\log v+p\log\log v-p\log R_+-\log(R_+A)+o(1)$ and $M(s(v))=v/R_++o(v)$. Hence
\[
\cH(v)=\frac v{R_+}\left[\log v+p\log\log v-p\log R_+-\log(R_+A)-1+o(1)\right],
\]
and $v=c\sigma^{m-1}$ yields the stated strong-transport formula.
\end{proof}

\paragraph{Data availability.} No data were generated or analysed in this study.

\paragraph{Conflict of interest.} The author declares that there is no conflict of interest.

\end{document}